\documentclass[10pt, reqno, final]{amsart}
\usepackage{etoolbox}

\makeatletter
\patchcmd{\@settitle}{\uppercasenonmath\@title}{}{}{}
\patchcmd{\@setauthors}{\MakeUppercase}{}{}{}
\makeatother

\title[The method of stochastic characteristics]{The method of stochastic characteristics for linear second-order hypoelliptic equations} 
\author[J. F\"{o}ldes, D. P. Herzog]
     {Juraj F\"{o}ldes, David P. Herzog\\
  \scriptsize{email: foldes@virginia.edu, 
             dherzog@iastate.edu}
}

\usepackage{amsfonts, amssymb, amsmath, amsthm,enumerate,esint,multicol, mathptmx}
\usepackage[colorlinks=true, pdfstartview=FitV, linkcolor=blue,
            citecolor=blue, urlcolor=b?lue]{hyperref}
\usepackage[usenames]{color}
\usepackage{xcolor}
\definecolor{Red}{rgb}{0.7,0,0.1}
\definecolor{Green}{rgb}{0,0.7,0}
\usepackage{accents}
\usepackage{comment}
\usepackage{graphicx}
\usepackage[capitalize,nameinlink,noabbrev]{cleveref}

\usepackage{hyperref}

\colorlet{darkblue}{blue!50!black}

\hypersetup{
    colorlinks,
    citecolor=darkblue,
    filecolor=red,
    linkcolor=darkblue,
    urlcolor=blue,
    pdfnewwindow=true,
    pdfstartview={FitH}
}

\numberwithin{equation}{section}

\newtheorem{Theorem}{Theorem}[section]
\newtheorem{Proposition}[Theorem]{Proposition}
\newtheorem{Lemma}[Theorem]{Lemma}
\newtheorem{Corollary}[Theorem]{Corollary}

\theoremstyle{remark}
\newtheorem{Remark}[Theorem]{Remark}
\theoremstyle{definition}
\newtheorem{Definition}[Theorem]{Definition}
\theoremstyle{definition}
\newtheorem{Example}[Theorem]{Example}
\theoremstyle{definition}

\newcommand{\PP}{\mathbf{P}}
\newcommand{\LL}{\mathcal{L}}
\newcommand{\B}{\mathcal{B}}
\newcommand{\N}{\mathbf{N}}

\newcommand{\s}{\text{stoc}}

\newcommand{\E}{\mathbf{E}}

\newcommand{\RR}{\mathbf{R}}
\newcommand{\X}{\mathcal{X}}

\begin{document}
\maketitle

\begin{abstract}

We study hypoelliptic stochastic differential equations (SDEs) and their connection to degenerate-elliptic boundary value problems on bounded or unbounded domains.  In particular, we provide probabilistic conditions that guarantee  that the formal stochastic representation of a solution is smooth on the interior of the domain and continuously approaches the prescribed boundary data at a given boundary point.  The main general results are proved using  fine properties of the process stopped at the boundary of the domain combined with hypoellipticity of the operators associated to the SDE. The main general results are then applied to deduce properties of the associated Green's functions and to obtain a generalization of Bony's Harnack inequality.  We moreover revisit the transience and recurrence dichotomy for hypoelliptic diffusions and its relationship to invariant measures.   
\end{abstract}

\section{Introduction}

\subsection{Overview}
Let $\mathfrak{m}\in \N$ and $U\subset \RR^\mathfrak{m}$ be nonempty, open set with nonempty boundary $\partial U$.  Let $\X\subset \RR^\mathfrak{m}$ be an open set containing $\overline{U}=U\cup \partial U$ and let $M_{\mathfrak{m} \times r}$ denote the set of $\mathfrak{m} \times r$ matrices with entries in $\RR$.  Consider a linear second-order differential operator of the form
\begin{align}
\label{eqn:gen}
L= \sum_{i=1}^\mathfrak{m} (b(x))_i \frac{\partial}{\partial x_i} +\frac{1}{2}\sum_{i,j=1}^\mathfrak{m} (\sigma(x) \sigma(x)^T)_{i,j} \frac{\partial^2}{\partial x_i \partial x_j} \,,
\end{align}
where $b\in C^\infty(\X; \RR^\mathfrak{m})$ and $\sigma \in C^\infty(\X; M_{\mathfrak{m} \times r})$.  In this paper, we study the formal stochastic representation $u_\s$ (see~\eqref{eqn:ustoch} below) corresponding to the combined Dirichlet and Poisson problems for $L$ on $U$:
\begin{align}
\label{eqn:pp}
\begin{cases}
L u = -f &\text{ on } \,\,U \,, \\
u(x) \rightarrow g(x_*) & \text{ as } x\rightarrow x_* \in \partial U, \,x\in U \,,
\end{cases}
\end{align} 
where $f\in C^\infty(U)$ and $g$ is continuous in a neighborhood of $\partial U$. Importantly, our assumptions allow for $U$ to be unbounded and for $L$ to be \emph{hypoelliptic}\footnote{The operator $L$ is called \emph{hypoelliptic on} $U$ if for all distributions $v$ on $U$ with $Lv \in C^\infty(V)$ for some open $V\subset U$, we have $v\in C^\infty(V)$.  }  on $U$.

Formally, the method of stochastic characteristics (see~\cref{sec:stchar} below) produces the formula $u_\s$ as a candidate expression for a classical solution of~\eqref{eqn:pp}. 
However,  in this setting, $u_\s$ may not be defined let alone be a classical solution of~\eqref{eqn:pp}.  The goal of this paper is to provide practical, probabilistic conditions under which $u_\s$ 
is well defined and satisfies the properties required by the problem~\eqref{eqn:pp}:
\begin{itemize}
\item $u_\s\in C^2(U)$ and $Lu_\s=f$ on $U$ in the classical sense;
\item $u(x)\rightarrow g(x_*)$ as $x\rightarrow x_* \in \partial U, \,x\in U$.
\end{itemize}
Results for classical well-posedness of~\eqref{eqn:pp} in the setting of a bounded domain $U$ are also  obtained.  

Although we employ tools from analysis, the manuscript primarily focuses on probabilistic methods.  One may compare our work with the classical work of Oksendal~\cite[Chapter 9]{Oksen_13} adapted 
to our setting, but we aim to be more self-contained on this particular topic.  Specifically, we start from a probabilistic construction of fundamental solutions and build up to a solution theory for the boundary-value problem~\eqref{eqn:pp}.  
As a consequence, one may also compare this work to a portion of the classical elliptic theory text of Gilbarg and Trudinger~\cite{GT_15}, but we use probabilistic techniques in the possibly unbounded $U$, hypoelliptic $L$ setting.

\subsection{The method of stochastic characteristics}
\label{sec:stchar}
In order to employ probabilistic techniques, we associate to $L$ a diffusion process $x_t$ on $\X$ with infinitesimal generator $L$.  That is, we suppose that $x_t$ satisfies an It\^{o} stochastic differential equation (SDE) of the form
\begin{align} 
\label{eqn:sde}
dx_t &= b(x_t ) \, dt + \sigma(x_t) \, d\mathbf{W}_t\,,
\end{align}
where $\mathbf{W}_t=(W_t^1, W_t^2, \ldots, W_t^r)^T$ is a standard, $r$-dimensional Brownian motion defined on, and adapted to, a filtered probability space $(\Omega, \mathcal{F}, \mathcal{F}_t, \PP, \E)$. The \emph{formal stochastic representation} $u_\s$ corresponding to the problem~\eqref{eqn:pp} then has the form 
\begin{align}
\label{eqn:ustoch}
u_\s(x) = \E_x \int_0^\tau f(x_s) \, ds + \E_x g(x_\tau), \qquad \, x\in U , 
\end{align}
where 
\begin{align}\label{def:tau}
\tau= \inf\{ t > 0 \,: \, x_t \notin U\}
\end{align}  
is the first positive exit time from $U$. Note that in~\eqref{eqn:ustoch}, the symbol $\E_x$ means the expected value for the law $\PP_x$ of the process $x_t$ with $x_0=x$.  
The formula~\eqref{eqn:ustoch} is formally derived using Dynkin's/It\^{o}'s formula~\eqref{eqn:Dynkin} applied to a  sufficiently \emph{nice} classical solution of~\eqref{eqn:pp}, which may 
or may not exist in this setting.  Specifically, in Dynkin's formula~\eqref{eqn:Dynkin} below, formally set $\phi(t,x)=u(x)$ and $\sigma= \tau$ to arrive at~\eqref{eqn:ustoch}.  Compared with the usual method of characteristics for first-order partial differential equations, this is the analogous method for linear second-order equations, but this time the characteristic trajectories $x_t$ are random.

Aside from it being unclear its precise relationship to~\eqref{eqn:pp}, observe  that the expression $u_\s$ in~\eqref{eqn:ustoch} is  itself formal in two ways.  First, depending on the set $U$ and the coefficients $b, \sigma$, both the process $x_t$ and the expected values in~\eqref{eqn:ustoch} may not be defined, so $u_\s$ in turn is not defined.  Second, even if $u_\s$ is defined on $U$, it may not satisfy the equation~\eqref{eqn:pp} in the classical sense.  
However, employing this formal argument via Dynkin's formula~\eqref{eqn:Dynkin},  $u_\s$ is a \emph{best guess} at a solution of~\eqref{eqn:pp}, and so it is natural to study $u_\s$ in relation to~\eqref{eqn:pp}.  

If the operator $L$ is hypoelliptic on $U$, then it has a smoothing effect reminiscent of second-order uniformly elliptic operators on $U$ with smooth coefficients.  However, even if $U$ is a bounded domain with boundary $\partial U$ satisfying the \emph{exterior cone condition}\footnote{$\partial U$ satisfies the \emph{exterior cone condition} if for all $x_* \in \partial U$ there exists a basis $x_1, x_2, \ldots, x_\mathfrak{m}$ of $\RR^{\mathfrak{m}}$ and a parameter $\lambda_*>0$ such that 
$\text{Cone}(x_*, \lambda_*):= \{ x_* + \textstyle{\sum}_{i=1}^\mathfrak{m} \lambda_i x_i \, : \, \lambda_i \in (0, \lambda_*) \}\subset \overline{U}^c $.}, 
there are many examples where the problem~\eqref{eqn:pp} is ill-posed in the classical sense (see \cref{ex:non} below).  
This is different compared to equation~\eqref{eqn:pp} with is uniformly elliptic operator $L$ of the form~\eqref{eqn:gen}.  From a probabilistic perspective, this difference can be explained intuitively,  especially as it relates to  satisfaction of the  boundary condition in equation~\eqref{eqn:pp}.     
Indeed, if $L$ is uniformly elliptic, then the noise in~\eqref{eqn:sde} is present in each direction of the equation, and therefore for short times the particle visits all points in a small ball.  See \cite{CFH_21} for a rigorous formulation of this statement.  This is, however, not necessarily the case for a general hypoelliptic diffusion.  In particular, even if the process $x_t$ initiated at any point $x\in U$ is defined, it may not hit certain portions of the boundary with positive probability, and consequently the values of $g$ on these portions
do not influence $u_{\s}$ in~\eqref{eqn:ustoch}. 
Moreover, when started on the boundary $\partial U$, the process may have a positive probability of re-entering $U$ prior to exiting, so that $u_\s$ may not satisfy the boundary condition.

\subsection{Previous results and layout of the paper}

Despite these issues in  the hypoelliptic setting, understanding when the equation or boundary conditions in~\eqref{eqn:pp} are satisfied in the classical sense by $u_\s$ is of notable importance, as hypoelliptic operators, and their corresponding boundary-value problems like~\eqref{eqn:pp}, play a key role in a number of problems in science and engineering.  See, for example, hypoelliptic diffusions arising in finite-dimensional models of turbulence~\cite{Bec_05, Bec_07, Bed_20, BHW_12, EMat_01, FGHH_20, GHW_11, HN_21, Rom_04} or in statistical mechanics and machine learning~\cite{Cam_21, Cun_18, Gia_19, HerMat_19, LSS_20, Stoltz_10}.  
Direct applications also occur in the ergodic theory of SDEs, where functionals of return times to compact sets can be seen as the formal stochastic representations corresponding 
to equations of the form~\eqref{eqn:pp}.  Such functionals are essential to understand precise rates of convergence to equilibrium~\cite{HM_11, RB_06, EGZ_19}.   

Historically, various aspects of the problem~\eqref{eqn:pp} in the hypoelliptic setting have drawn interest from researchers dating back to Kolmogorov \cite{Kol_34}, who gave the first known example of a hypoelliptic diffusion that is not uniformly elliptic.  Later, the seminal work of H\"{o}rmander \cite{Hor_67} provided an efficient tool to determine hypoellipticity of $L$ by locally calculating the Lie algebra of vector fields 
that create the operator $L$ in the form~\eqref{eqn:gen},  see \cref{sec:hypo} for further information.  For classical well-posedness of the problem~\eqref{eqn:pp}, the pioneering work of Bony~\cite{Bony_69} is fundamental.  There, Bony gives conditions on a bounded open set 
$U$ and a hypoelliptic operator $M$ of the form $M=L-a$, where $a\in C^\infty(U)$ is positive and bounded away from zero on $U$, so that the problem~\eqref{eqn:pp}
with $M$ replacing $L$  is well-posed  in the classical sense.  Apart from the presence of the positive function $a$, which in particular aides in the existence part of the problem~\eqref{eqn:pp}, 
 a critical assumption guaranteeing continuity on $\overline{U}$ is that at every point on $\partial U$, the noise in the equation points in the normal direction.  
 Such an assumption is often not satisfied for many hypoelliptic operators of interest,
  for example if the noise is additive (that is, spatially constant), see~\cite{Bed_20, BHW_12, Cam_21, Cun_18, EMat_01, FGHH_20, GHW_11, HerMat_19, HN_21, LSS_20, Rom_04}.  
Thus, it is natural to investigate if classical well-posedness holds under weaker assumptions. 

More recently, the hypoelliptic Dirichlet problem ($f\equiv 0$ in~\eqref{eqn:pp}) in a bounded domain $U$ was revisited in~\cite{Ram_97}, where Bony's result was extended to operators $L$ of the form~\eqref{eqn:gen}; that is, with  $a \equiv 0$, assuming $L$ satisfies a maximum principle and again that at every point on $\partial U$ there is noise pointing in the normal direction.  
Similar to Bony's work, the results in~\cite{Ram_97} fail to apply in many problems of interest.  
Perhaps more closely related to the present paper is the work of  Kogoj~\cite{Kog_17}, which uses analytic methods from potential theory to establish a cone-type criterion for the existence of a generalized solution, in the sense of Perron-Wiener, of the Dirichlet problem for a hypoelliptic $L$ in parabolic form. 
More importantly, \cite{Kog_17} assumes the existence of a \emph{well-behaved} fundamental solution and that the diagonal of $\sigma(x) \sigma(x)^T$ never vanishes.  
Note that probabilistically the cone-type condition is natural as it gives particle a space to exit the domain, and indeed for uniformly elliptic operators, such condition is sufficient 
to guarantee boundary regularity 
(in the potential-theoretic sense). As described above, additional conditions are required in hypoelliptic setting, and the supplementary criteria provided in   \cite{Kog_17}
 appear difficult to check in concrete scenarios.  We refer the reader to~\cite{GS_90, LTU_17,NS_87} for earlier, related results.  We also refer to the recent work of the authors~\cite{CFH_21} which provides 
probabilistic methods for determining regular points  (defined below in \cref{sec:boundary}) on the boundary using the theory of large deviations and  laws  of the iterated logarithm.
Note that~\cite{CFH_21} was partly inspired by the work of Lachal~\cite{Lac_97}, where a functional law of the iterated logarithm was obtained for the iterated Kolmogorov stochastic differential equation.   
 However, applications to solvability of~\eqref{eqn:pp} were not discussed in detail in~\cite{CFH_21} or~\cite{Lac_97}.

Another goal of the present work is to provide a relatively self-contained presentation on the topic in this paper, referring primarily to graduate-level textbooks to obtain needed results.  Our hope is that the level of the paper is similar to the level in Oksendal's book~\cite{Oksen_13}, and that the interested graduate student who has had an introductory course in stochastic analysis will find the paper readable.  We did this for two reasons.  First, prior to starting this project, at times we found it difficult to locate results for~\eqref{eqn:pp} in the hypoelliptic setting and this sentiment was also confirmed in conversations with colleagues.  Second, there were claims in the literature 
about~\eqref{eqn:pp} that seemed both correct and intuitive, but we could not locate a proof in the existing literature.  Thus, we hope this paper will serve as a resource 
to which the (stochastic) analysis community can refer as needed, since either there is a physical proof in the paper or the statement is easy to locate in the literature using the provided references.         

The organization of this paper is as follows.  In \cref{sec:notation}, we introduce further notation, assumptions, terminology and make preliminary remarks about hypoellipticity.  
H\"{o}rmander's theorem, in particular, will be discussed in detail in \cref{sec:notation}.  In \cref{sec:boundary}, we elaborate on the importance of boundary behavior for $x_t$ as its relates to solvability 
of~\eqref{eqn:pp} in the classical sense.  In particular, in  \cref{sec:boundary} one can find our definition of \emph{regular} and \emph{irregular} points on $\overline{U}$. \cref{sec:closed} focuses on 
 special equations of the form~\eqref{eqn:pp}, in particular those involving functionals of $\tau$, and these special equations are used later to formulate our main  general results in \cref{sec:poisson2}. 
\cref{sec:poisson2} provide conditions on $L$, $\tau$, and $u_\s$ that guarantee $L u_\s=-f$ on $U$ in the sense of distributions, hence in the classical sense if $L$ is hypoelliptic and $f$ smooth.   Green's functions and a generalization of Bony's form of the Harnack inequality~\cite[Theorem 7.1]{Bony_69} are discussed in \cref{sec:green}, while in \cref{sec:recurrence}, as an application of these results, we re-derive the transience and recurrence dichotomy for degenerate diffusions.

\section{Notation and Preliminary Remarks}
\label{sec:notation}

In this section, we introduce notation, assumptions and terminology used throughout the paper.  We also make a few preliminary remarks.  First, we fix some standard notation.

\subsection{Basic notation} For Borel sets $V,V_1 \subset \RR^k$, $V_2 \subset \RR^n$, and $W\subset \RR^\ell$,  we use the following notation.
\begin{itemize}
\item[--] $C(V; W)$ denotes the set of continuous functions $\phi:V\to W$;

\item[--] $C^j(V; W)$, $j\geq 1$, denotes the set of $j$-times continuously differentiable functions $\phi: V\to W$;  

\item[--] $C^\infty(V; W)= \bigcap_{j=1}^\infty C^j(V; W)$; 

\item[--] $C^{j_1,j_2}(V_1 \times V_2; W)$ denotes the set of functions $\phi = \phi(x,y): V_1 \times V_2 \to W$ which are $j_1$-times continuously differentiable in $x$ and $j_2$-times continuously differentiable in $y$; 

\item[--] $B(V; W)$ denotes the set of bounded, Borel measurable functions $\phi: V\to W$; 
 
\item[--] In any of the above functions spaces, a subscript of $0$ indicate that the function is moreover compactly supported in its domain of definition, e.g. $C_0^j$ versus $C^j$;

\item[--] For $1\leq p \leq \infty$, $L^p(V)$ denotes the set of measurable functions $\phi: V\to \RR$ such that $\| \phi\|_{L^p(V)}^p:= \int_V |\phi |^p \, dx < \infty$ if $p\in [1, \infty)$ and $\|\phi \|_{L^\infty(V)}:= \text{ess sup}_{x\in V} | \varphi(x)|< \infty$ if $p=\infty$;  

\item[--] $\text{Lip}(V;W) $ denotes the set of functions $\phi: V \to W$ which are Lipschitz continuous on $V$; 

\item[--] For any $s\in \RR$, $H^s(\RR^k)$ denotes the usual Sobolev space $W^{s,2}(\RR^k)$;  
  
\item[--] If the target $W$ is clear from context or not important, we write $C^j(V)=C^j(V;W)$, $C_0^j(V)=C_0^j(V; W)$, $B(V)=B(V;W)$, etc.  When the context is clear, we may also drop $V$ and write $C^j, C_0^j$, etc.  

\item[--] $\B_V$ denotes the set of Borel measurable subsets of $V$;  

\item[--] $\B$ denotes the set of Borel measurable subsets of the open set $\X$, where $\X$ is as in the introduction.   

\item[--] For any $B\in \B$ bounded with $\overline{B} \subset \X$, we define 
\begin{align}
\label{eqn:tauB}
\tau_B= \inf\{ t>0 \, : \, x_t \in B^c\}, 
\end{align}
where $\inf \emptyset: =\infty$.

\item[--] We define a sequence $\X_n \subset \X$, $n \geq 1$ of bounded open sets with $\overline{\X_n} \subset \X_{n+1}$ for all $n$ and $\bigcup_{n=1}^\infty \X_n =\X$.  

\item[--] At times, we will also need a sequence $U_n \subset U$, $n\geq 1$, of bounded  open sets with $\overline{U_n} \subset U_{n+1}$ and $\bigcup_{n=1}^\infty U_n  = U$.  

\item[--] We define
\begin{align}
\label{eqn:exitpos}
\tau=\tau_U:= \lim_{k\to \infty} \tau_{U\cap \X_k}. 
\end{align}
and 
\begin{align}
\label{eqn:tauX}
\tau_\X= \lim_{k\to \infty} \tau_{\X_k}.  
\end{align}
\end{itemize}

\subsection{Nonexplosivity of $x_t$}Depending on the behavior of the coefficients $b$ and $\sigma$, the solution $x_t$ of equation~\eqref{eqn:sde} evolving on the neighborhood $\X$ of $\overline{U}$ is only
 a priori defined locally in time until the process exits $\X$.  That is, by the standard existence and uniqueness theorem for SDEs~\cite{Oksen_13}, for any initial condition $x \in \X$ at time $t=0$, 
 equation~\eqref{eqn:sde} has a unique (pathwise) solution $x_t$ for all times $0\leq t< \tau_\X$.  However, throughout the paper, we will assume that $\tau_\X$ is almost surely infinite as in the following definition.   

\vspace{0.1in}

\begin{Definition}
\label{def:nonexpl}
We say that the process $x_t$ is \emph{non-explosive} if \begin{align}\label{eqn:nonexpl}
\PP_x\{  \tau_\X< \infty \} =0 \,\,\,\text{ for all } x \in \X,  
\end{align}
where $\PP_x$ denotes the probability $\PP$ but indicates that $x_0=x$. \end{Definition}

Nonexplosivity implies that for all initial conditions $x\in \X$, equation~\eqref{eqn:sde} has a unique pathwise solution $x_t$ defined on $\X$ for all finite times $t\geq 0$ almost surely.  Furthermore, $t\mapsto x_t:[0, \infty) \to \X$ is continuous almost surely.      

 \begin{Remark}
 \label{rem:nonexpl}
Suppose that $U$ is bounded. Since $b$ and $\sigma$ are  smooth and defined on an open neighborhood of $\overline{U}$, 
one can extend $b$ and $\sigma$ to functions $\hat{b}\in C^\infty(\RR^\mathfrak{m}; \RR^\mathfrak{m})$ and $\hat{\sigma}\in C^\infty(\RR^\mathfrak{m}; M_{\mathfrak{m} \times r})$, respectively, such that $\hat{b}=b$ on $\overline{U}$ and $\hat{\sigma}=\sigma$ on $\overline{U}$ and such that $\hat{b}, \hat{\sigma}$ both have bounded derivatives of all orders, e.g. extend $b$, $\sigma$ by zero outside of a larger neighborhood and mollify.  
Consequently, replacing $\X$ with $\RR^\mathfrak{m}$ and $b, \sigma$ in \eqref{eqn:sde} with $\hat{b}, \hat{\sigma}$ we find that the resulting solution $\hat{x}_t$ is nonexplosive.  
Note that the problem~\eqref{eqn:pp} and its formal stochastic representation~\eqref{eqn:ustoch}  remain unchanged regardless of the chosen extension, so nonexplosivity is implicitly present.  On the other hand, if $U$ is unbounded, then one must verify nonexplosivity separately, typically using Lyapunov-type methods.  See \cref{lem:nonexpl} below for further details.           
\end{Remark}

By nonexplosivity, $x_t$ is a well-defined Markov process on $\X$. It thus induces a corresponding Markov semigroup $\{\mathcal{P}_t\}_{t\geq 0}$ which acts on functions $\phi\in B( \X; \RR)$ via 
\begin{align}
\label{eqn:semif}
\mathcal{P}_t \phi(x) := \E_x \phi(x_t)
\end{align}
and dually on $\B$-measures $\nu$ via 
\begin{align}
\label{eqn:semim}
\nu \mathcal{P}_t (B)= \int_{\X} \mathcal{P}_t \mathbf{1}_B(x) \nu(dx), \,\, B\in \B.
\end{align}
We let 
\begin{align}
\mathcal{P}_t(x, B):= \mathcal{P}_t \mathbf{1}_B(x), \,\,\,\, \,\,\,\, B\in \B, \, x\in \X, 
\end{align}
denote the corresponding Markov transitions.  
   
Note that if $L$ is as in~\eqref{eqn:gen} and $\phi \in C^{1,2}_0([0, \infty) \times \X; \RR)$, then \emph{Dynkin's formula}
\begin{align}
\label{eqn:Dynkin}
\E_x \phi(\sigma, x_\sigma) = \phi(0, x) + \E_x\int_0^\sigma ( \partial_t + L)\phi(s, x_s) \, ds 
\end{align}
 holds for any bounded stopping time $\sigma$ with respect to the filtration $\mathcal{F}_t$.  Dynkin's formula allows one to study various properties of $L$ by analyzing path properties of the stochastic process $x_t$.  

\begin{Remark}
\label{rem:dynkin}
By a  standard localization procedure applied to the equation~\eqref{eqn:sde} on $\X_{k+1}$, the same formula~\eqref{eqn:Dynkin} holds for any bounded stopping time $\sigma \leq \tau_{\X_k}$ and any $\phi\in C^{1,2}([0, \infty) \times \X_{k+1})$ regardless if $x_t$ is nonexplosive.  
\end{Remark}

As a simple application of Dynkin's formula, we briefly recall the following basic method for checking nonexplosivity.  See also~\cite{Khas_11, MT_93, RB_06}.    
\begin{Lemma}
\label{lem:nonexpl}
Suppose there exists $w\in C^2(\X; [0, \infty))$ and constants $C, D>0$ such that 
\begin{align*}
w_k := \inf_{x\in \partial \X_k} w(x) \to \infty \,\, \text{ as }\,\, k\to \infty \qquad \text{ and } \qquad L w \leq Cw+D \,\, \text{ on } \,\, \X.  
\end{align*}
Then $x_t$ is nonexplosive as in~\eqref{eqn:nonexpl}.  
\end{Lemma}

\begin{proof}
\cref{rem:dynkin}, nonnegativity of $w$ and Dynkin's formula together imply that if $\tau_k= \tau_{\X_k}$, then 
\begin{align*}
w_k e^{-Ct}\PP_x \{ \tau_k \leq t \} &\leq \E_x  w(x_{t\wedge \tau_k}) e^{-C(t\wedge \tau_k)}  \mathbf{1}_{\{\tau_k < t\}}    \leq \E_x w(x_{t\wedge \tau_k}) e^{-C(t\wedge \tau_k)} \\
&= w(x) + \E_x \int_0^{t\wedge \tau_k} - Cw(x_s) e^{-Cs } + e^{-Cs} L w(x_s) \, ds\\
&\leq w(x)+\int_0^t e^{-Cs} D \, ds \leq w(x)+ D.   
\end{align*}   
Rearranging the above  we obtain $\PP_x \{ \tau_k \leq t \} \leq e^{Ct}(w(x)+ D)/w_k$.  Passing $k\to \infty$, we find that $\PP_x \{ \tau_\X \leq t \} =0$ for any given $t\geq 0$ and $x\in \X$ 
and~\eqref{eqn:nonexpl} follows.  
\end{proof}

\subsection{The process $x_t$ stopped on $\partial U$} \label{sec:stopped}
We will often use the process $x_t$ \emph{stopped} on the boundary $\partial U$.  That is, define the stopping time
 \begin{align}\label{eqn:dotz}
 \tau_0= \inf\{ t\geq 0 \,: \, x_t \notin U\}.  
 \end{align} 
Then the \emph{stopped process} $\tilde{x}_t$ is defined by  
 \begin{align}\label{eqn:dfsp}
 \tilde{x}_t:= x_{t\wedge \tau_0}, \,\, t\geq 0.   
 \end{align}
  Because $x_t$ is assumed to be nonexplosive, the stopped process $\tilde{x}_t$ is a well-defined Markov process distributed on $\overline{U}$ for all times $t\geq 0$.  
  We let $\tilde{\mathcal{P}}_t$ and  $\tilde{\mathcal{P}}_t(x, \, \cdot \, )$, $x\in \overline{U}$, respectively denote the Markov semigroup and transitions associated to $\tilde{x}_t$. 

\begin{Remark}
For the process $x_t$ with $x_0 \in \overline{U}$, $\tau_0$ is the first time $x_t$ hits the boundary.  Note that $\tau_0$ is in general different from 
the first positive exit time 
$\tau$ from $U$ when $x_t$ initiates on the boundary $\partial U$, as the process started there may first enter the region $U$ before exiting.    
\end{Remark}

   \subsection{Hypoellipticity}
   \label{sec:hypo}
 Let us first define precisely the term \emph{hypoelliptic}. 
   
   \begin{Definition}
   Let $V\subset \RR^{k}$ be non-empty, open set and $M$ be a differential operator with coefficients belonging to $C^\infty(V)$.  
   We say that $M$ is \emph{hypoelliptic} on $V$ if for any distribution $v$ on $V$ with $Mv \in C^\infty (W)$ for some $W\subset V$ nonempty open, we have $v\in C^\infty(W)$.     
\end{Definition}

  Fundamental to our analysis are the smoothing properties of the differential operator $L$ afforded by hypoellipticity.  
In the classical paper~\cite{Hor_67}, H\"{o}rmander (see Theorem \ref{thm:Hor} below) provides a sufficient condition for hypoellipticity on open $V\subset \RR^k$ for operators $M$ of the form 
\begin{align}
\label{eqn:formM}
M=a+ X_0 + \tfrac{1}{2}\sum_{\ell=1}^j X_\ell^2,
\end{align}
 where $a \in C^\infty(V; \RR)$ and $X_0, X_1, \ldots, X_j$ are $C^\infty$ vector fields on $V$.  
 
 After some algebraic manipulations, note that 
\begin{align*}
 L, L^*, \partial_t \pm L, \partial_t \pm L^*, 
 \end{align*}
 where $L$ is as in~\eqref{eqn:gen} and $L^*$ denotes the formal $L^2(dx)$-adjoint of $L$, can all be written in the form~\eqref{eqn:formM} on the respective open sets $\X, \X, (0, \infty) \times \X, (0, \infty) \times \X$.  For example, observe that if we let $X_i= \sum_{j=1}^\mathfrak{m}\sigma_{ji}(x)\partial_{x_j}$, $i=1,\ldots, r$, and $Y_0= \sum_{j=1}^\mathfrak{m} b_j(x) \partial_{x_j}$, then 
 \begin{align}
 L &= Y_0- \sum_{\ell=1}^m\bigg[\sum_{i=1}^r \sum_{j=1}^m \sigma_{ji}(x) \partial_{x_j}(\sigma_{\ell i}(x)) \bigg]\frac{\partial}{\partial x_\ell} + \tfrac{1}{2}\sum_{j=1}^r X_j^2  \\
 \label{eqn:vfL}&=: X_0 + \tfrac{1}{2}\sum_{j=1}^r X_j^2.  \end{align}

 To introduce H\"{o}rmander's condition which implies hypoellipticity, we first define 
the \emph{Lie bracket} $[X,Y]$ of differentiable vector fields  $X=\sum_i X^i \partial_{x_i}$ and $Y= \sum_j Y^j \partial_{x_j}$ on an open set in $\RR^{k}$  by 
\begin{align*}
[X,Y]= \sum_{j=1}^k \sum_{i=1}^k (X^i \partial_{x_i}(Y^j) - Y^i \partial_{x_i}(X^j)) \partial_{x_j}.   
\end{align*}

\begin{Definition}
Suppose $V\subset \RR^{k}$ is an open set and $M$ is an operator of the form~\eqref{eqn:formM}, where $X_0, X_1, \ldots, X_m$ are 
$C^\infty$ vector fields on $V$ and $a\in C^\infty(V)$.  Define the following  $C^\infty$ vector fields on $V$:
\begin{equation}
\label{eqn:Horlist}
\begin{aligned}
&X_j,  &&j\in \{0,1, \ldots, m\}\\
 &[X_i, X_j],  &&i,j \in \{ 0,1, \ldots, m \} \\
&[[X_i, X_j], X_k], && i,j,k \in \{0,1,2, \ldots, m \}\\
 & \vdots && \vdots
\end{aligned}
\end{equation}  
If the vector fields in~\eqref{eqn:Horlist} span the tangent space at all points in $V$, we say $M$ satisfies the \emph{H\"{o}rmander condition} on $V$.  
\end{Definition}

\begin{Remark}
If $M$ of the form~\eqref{eqn:formM} is elliptic, then $X_1, X_2, \ldots, X_m$ span the tangent space at all points in $V$.  Thus the brackets in~\eqref{eqn:Horlist} can
be viewed as a generalization 
of ellipticity.  Relating this back to $L$ in~\eqref{eqn:gen} and the equation~\eqref{eqn:sde}, one can interpret Lie brackets  in~\eqref{eqn:vfL} as a propagation of the randomness implicitly through the equation.  
\end{Remark}

H\"{o}rmander's theorem provides even more refined estimates on the smoothing effect along every bracket in the list~\eqref{eqn:Horlist}.  Essentially, the more brackets one takes to reach a certain direction, 
the smaller the smoothing effect occurs along that direction.  While we do not state this general result rigorously, we need below the following simpler version to establish a generalization of Bony's Harnack inequality in \cref{sec:green}. 

\begin{Theorem}[H\"{o}rmander 1967 \cite{Hor_67}]
\label{thm:Hor}
If $M$ defined in~\eqref{eqn:formM} satisfies the H\"{o}rmander condition on a nonempty open set $V\subset \RR^{k}$,  then $M$ is hypoelliptic on $V$.  More specifically, there exists $\delta >0$ such that for any $\psi_1, \psi_2 \in C^\infty_0(V)$ with $\psi_2=1$ on an open neighborhood of $\text{\emph{supp}}(\psi_1)$ we have  
\begin{align}
\label{eqn:Horest}
\| \psi_1 u \|_{H^{s+\delta}(\RR^k)} \leq C_s( \| \psi_2 M u \|_{H^s(\RR^k)} + \| \psi_2 u \|_{H^s(\RR^k)})
\end{align}  
for any $s\in \RR$ and any distribution $u$ on $V$ with $\psi_2 u, \psi_2 Mu \in H^{s}(\RR^{k})$.  In \eqref{eqn:Horest}, $C_s$ is a constant depending only on $s$, $\psi_1$, $\psi_2$, 
and the domain $V$ and all functions are assumed to be zero outside of their compact supports.       
\end{Theorem}

\begin{Remark}
Note that hypoellipticity is a consequence of the inequality~\eqref{eqn:Horest} by a bootstrapping argument.  
As remarked above, although we mostly use hypoellipticity somewhat independently of H\"{o}rmander's result, we need the precise estimate~\eqref{eqn:Horest} for the proof of a generalization of Bony's Harnack inequality~\cite{Bony_69}.    
\end{Remark}

Often, instead of listing operators that are hypoelliptic, we make a simper hypothesis.  See the  \emph{parabolic H\"{o}rmander condition} below, which ensures that an entire list of operators is hypoelliptic.

\begin{Definition}
\label{def:parahor}
Suppose that $V\subset \RR^k$ is open and $M$ is an operator of the form~\eqref{eqn:formM} where the $X_0, X_1, \ldots, X_m$ are $C^\infty$ vector fields on 
$V$ and $a \in C^\infty(V)$.  Let $\mathcal{V}_0= \{ X_1, X_2, \ldots, X_m\}$ and for $k\geq 1$ inductively define
\begin{align*}
\mathcal{V}_k = \{ [X_i, X]\, : \, X\in \mathcal{V}_{k-1}, i=0,1,2, \ldots, m \} \qquad \textrm{and} \quad \mathcal{V}= \bigcup_{k=0}^\infty \mathcal{V}_k \,.
\end{align*}
 If $\mathcal{V}$ spans $\RR^{\mathfrak{m}}$ at all points in $V$, then we say that $M$ satisfies the 
\emph{parabolic H\"{o}rmander condition} on $V$.  
\end{Definition}

\begin{Remark}
Observe that for $M$ to satisfy the parabolic H\"{o}rmander condition, one cannot include $X_0$ in the list of a spanning set.  Rather, $X_0$ must be first commuted with another vector field, for example $[X_0, X_1]$. 
\end{Remark}

\begin{Example}
Let
\begin{align*}
M= \partial_{x_2}^2 + \partial_{x_1} =: X_1^2+ X_0. 
\end{align*} 
Then, $M$ satisfies the H\"{o}rmander condition on $\RR^2$ but $M$ does not satisfy the parabolic H\"{o}rmander condition on any open set in $\RR^2$.  
\end{Example}

Because the parabolic H\"{o}rmander condition limits the fields that can be taken in a spanning set, a routine calculation shows that if $V\subset \X$ is non-empty, open and $L$ as in~\eqref{eqn:gen} is written in the form~\eqref{eqn:vfL} and satisfies the parabolic H\"{o}rmander condition on $V$, then all operators 
\begin{align}\label{parhyp:list}
L, \, L^*, \, L +\beta, \, L^*+ \beta,\, \partial_t \pm L,\,  \partial_t \pm L^*
\end{align}
where $\beta \in \RR$ 
are hypoelliptic on the respective sets
\begin{align}
V, \,V, \,V,\, V,\, (0, \infty) \times V,\, (0, \infty) \times V.  
\end{align}
Below, we will see that 
 hypoellipticity of the operators in \eqref{parhyp:list} decides the existence and regularity of densities related to the law of $x_t$.

\subsection{Assumption list}  Here we provide an almost complete list of assumptions used in the paper.  It is meant as a reference except for \cref{rem:assump}.  Thus, aside from \cref{rem:assump},  
the reader should skip the rest of this section and come back to consult particular assumptions used later in the paper. 

\begin{Remark}
\label{rem:assump}
Throughout, we assume that $b\in C^\infty(\X; \RR^\mathfrak{m})$ and $\sigma \in C^\infty(\X; M_{\mathfrak{m} \times r})$ without explicitly mentioning it.     Also, if $U$ is bounded, 
one does \emph{not} need to explicitly assume that $x_t$ is nonexplosive using the extension argument from \cref{rem:nonexpl}.  
  \end{Remark}
Depending on the context, we apply the following assumptions as needed.

\vspace{0.1in}

\begin{raggedright}

\hypertarget{(U00)}{\bf(U00)} $U \subset \RR^\mathfrak{m}$ is nonempty, open set and $V:=U^c$ is nonempty, closed set with $\X \supset \overline{U}$ open.    

\hypertarget{(U0)}{\bf(U0)} $U \subset \RR^\mathfrak{m}$ is nonempty, open set and $\X\supset \overline{U}$ is an open set.

\hypertarget{(NE)}{\bf(NE)}  $x_t$ is nonexplosive as in \cref{def:nonexpl}.

\hypertarget{(L1)}{\bf(L1)} $L$ is hypoelliptic on $U$.  

\hypertarget{(L2)}{\bf(L2)}  $\partial_t -L^*$ is hypoelliptic on $(0, \infty) \times U$.

\hypertarget{(L3)}{\bf(L3)}  $\partial_t-L$ is hypoelliptic on $(0, \infty) \times U$.

\hypertarget{(L4)}{\bf(L4)}  $L^*$ is hypoelliptic on $U$.  

\hypertarget{(PH)}{\bf(PH)}  $L$ satisfies the parabolic H\"{o}rmander condition on $U$ as in~\cref{def:parahor}.

\hypertarget{(UIDgx)}{\bf (UID($g, x_*$))} Fix $g:\overline{U}\to \RR$ measurable and $x_* \in \partial U$.  Then for some $\delta >0$ the family 
\begin{align}
\label{eqn:defGg}
\mathcal{G}_{g, \delta}(x_*) := \{ g(x_\tau(x)) \, : \, |x-x_* |< \delta, \, x\in U\}
\end{align}
 is uniformly integrable.  Here, $\tau$ is as in \eqref{eqn:exitpos} and $x$ in $x_\tau(x)$ means $x_0 = x$.

\hypertarget{(UIDg)}{\bf (UID($g$))} Fix $g:\overline{U}\rightarrow \RR$ measurable.  Then,  for every $x_* \in \partial U$, condition \hyperlink{(UIDgx)}{{\bf (UID($g,x_*$))}}
 is satisfied.

\hypertarget{(UIPfx)}{\bf (UIP($f, x_*$))} 
 Fix $x_* \in \partial U$ and $f:\overline{U}\to \RR$ measurable.  Then for some $\delta >0$, 
\begin{align}
\label{eqn:defGf}
\mathcal{G}_\delta^f(x_*) := \{ \textstyle{\int_0^\tau f(x_s (x)) \, ds}\, : \, |x-x_*| < \delta, \, x\in U \}
\end{align}
 is uniformly integrable.      

\hypertarget{(UIPf)}{\bf (UIP($f$))}  Fix $f: \overline{U}\rightarrow \RR$ measurable.  Then for every $x_* \in \partial U$, condition \hyperlink{(UIPfx)}{{\bf (UIP($f, x_*$))}} is satisfied.

\hypertarget{(CEx)}{{\bf (CE($x_*$))}}
Fix $x_* \in \partial U$ and recall $\X_n$ introduced in~\eqref{eqn:tauX}. 
 For every $\delta_1 >0$ there exists $n\in \N$ and  $\delta_2 >0$ such that $|x-x_*| < \delta_2$, $x\in U$, implies
\begin{align}
\label{eqn:condexit}
\PP_x\{ \tau_{\X_n}< \tau \} < \delta_1.  
\end{align}

\hypertarget{(CE)}{{\bf (CE)}}  For every $x_* \in \partial U$, condition \hyperlink{(CEx)}{{\bf (CE($x_*$))}} is satisfied.

\end{raggedright}

\vspace{0.1in}

 In \cref{rmk:cfcx}, we provide sufficient conditions for 
\hyperlink{(CEx)}{{\bf (CE($x_*$))}} to hold.

\section{Remarks on Boundary Behavior}
\label{sec:boundary}

In order to solve the equation~\eqref{eqn:pp} in the classical sense, understanding the behavior of the process $x_t$ satisfying the equation~\eqref{eqn:sde} near the boundary $\partial U$ is critical.  
In this section, we  explore various conditions related to boundary behavior  used in the literature to ensure well-posedness of equation~\eqref{eqn:pp} (in the classical sense) when $L$ fails to be uniformly elliptic.         

\subsection{Nice points and regular points}
\begin{Definition}
\label{def:nice}
We call $x_* \in \partial U$ \emph{nice} if there exists an open neighborhood $U_{x_*}\subset \X$ of $x_*$ and a function $w\in C^2(U_{x_*})$ satisfying the following conditions:\begin{itemize}
\item[(i)]  $w>0$ on $U_{x_*} \setminus \{ x_* \}$ and $w(x_*)=0$;
\item[(ii)]  $Lw < 0$ on $U_{x_*}$.  
\end{itemize}   
\end{Definition}

\begin{Remark}
We show in \cref{prop:niceimpreg} below that if $x_* \in \partial U$ is nice and the relevant hypoellipticity is satisfied, then the process $x_t$ exits $U$ instantaneously when started from $x_*$. 
However, proving that $x_*$ is nice; that is, finding a Lyapunov function $w$ in  \cref{def:nice}, can be highly nontrivial or even impossible. One can ensure $x_* \in \partial U$ is nice provided $\partial U$ has an exterior normal vector to $U$ at $x_*$ and randomness pointing in the direction of the normal vector (see \cref{rem:noisepoint} below). 
Intuitively,  the process $x_t$ projected onto this normal direction for small times behaves like a scaled one-dimensional Brownian motion.  Then, the process must exit the domain instantaneously as the one-dimensional Brownian motion has no preferred direction and dominates the motion in small times.  See~\cite[Section 7]{CFH_21} for further details.
\end{Remark}

\begin{Remark}
If $U$ is furthermore assumed to be a bounded and $L$ is assumed to be hypoelliptic, then one can show that the Dirichlet problem~\eqref{eqn:pp} with $f\equiv 0$, has a unique classical solution if all points on the boundary are nice.  Although the main result in~\cite{Ram_97} establishes this fact, it is not exactly stated in this way.          
\end{Remark}

In this paper, we find it more convenient to phrase our hypotheses in terms of stopping times for the process $x_t$ solving~\eqref{eqn:sde}.  Hence, we define 
\begin{align}\label{dftb}
\overline{\tau}= \inf\{ t\geq  0 \, : \, x_t \notin \overline{U}\} \,,
\end{align}
and recall the first positive exit time $\tau $ from $U$ defined in equation~\eqref{eqn:exitpos}.

\begin{Definition}
\label{def:regular} 
We call an interior point $x\in U$ \emph{regular} if $\PP_{x}\{ \tau=\overline{\tau}  \}=1$.  A boundary point $x_* \in \partial U$ is called \emph{regular} if $\PP_{x_*}\{ \overline{\tau}=0 \}=1$.  Points on $\overline{U}$ are called \emph{irregular} otherwise.  We call the set $U$ \emph{interior regular} if all $x\in U$ are regular.  We call $ U$ \emph{boundary regular} if all $x_* \in \partial U$ are regular.         
\end{Definition}

\begin{Remark}
Observe that if $x\in U$ is regular, then the process $x_t$ started from $x$ exits $U$ and $\overline{U}$ at the same time.  In particular, $x_t$ cannot reach $\partial U$ and return to the interior
of $U$ with positive probability. 
 On the other hand, $x_* \in \partial U$ being regular means $x_t$ initiated at $x_*$ must exit $\overline{U}$ instantaneously.  
\end{Remark}

\begin{Remark}
Note that the event $\{ \overline{\tau}=0 \}$ belongs to the \emph{germ sigma field} $\bigcap_{t>0} \mathcal{F}_t$, hence has probability $0$ or $1$ by Blumenthal's $0$-$1$ law.  Thus,
 $x_* \in \partial U$ being regular is equivalent to $\PP_{x_*} \{ \overline{\tau} =0\}>0$.           
\end{Remark}

Using the strong Markov property for $x_t$, the next result states that boundary regular implies interior regular. 
\begin{Proposition}
\label{prop:bregimintreg}
If $U$ is boundary regular, then $U$ is interior regular.  
\end{Proposition}

\begin{proof}
Let $x\in U$ and suppose $\tau =\infty$.  Then $\overline{\tau}\geq \tau=\infty$, so $\tau= \overline{\tau}$.  If the event $\{\tau< \infty\}$ has positive probability, then the strong Markov property 
and boundary regularity gives
\begin{align*}
\PP_x\{ \tau= \overline{\tau}, \, \tau < \infty \}= \E_x [\E_x \mathbf{1}_{\{ \tau= \overline{\tau}, \, \tau < \infty \}} | \mathcal{F}_\tau]= \E_x \mathbf{1}_{\{ \tau < \infty \}} \PP_{x_\tau} \{ \overline{\tau}=0 \}=\PP_x\{ \tau< \infty\}.\end{align*}
This finishes the proof since 
\begin{align*}
\PP_x\{ \tau= \overline{\tau}\} &= \PP_x\{ \tau= \overline{\tau}, \, \tau =\infty\}+ \PP_x\{ \tau= \overline{\tau}, \, \tau <\infty\}\\
&= \PP_x\{ \tau=\infty \} + \PP_x\{ \tau < \infty\}=1.  
\end{align*}  
\end{proof}
As the next example shows, the converse of \cref{prop:bregimintreg} is false even if the parabolic H\"{o}rmander condition is satisfied.  

\begin{Example}
\label{ex:non}
Let $U\subset \RR^2$ be the interior of the open square with vertices $(-1,1)$, $(1,1)$, $(1,-1)$, $(-1,-1)$, and consider the following SDE on $U$
\begin{align*}
dx^1_t&=-(x_t^2)^2 \, dt \,, \\
dx^2_t&=\sqrt{2} \, dB_t \,,
\end{align*}
where $B_t$ is a standard, one-dimensional Brownian motion.  Let $I$ be the open segment connecting $(1, -1)$ and $(1, 1)$. 
Note that all points on $I \subset \partial U$ are not regular since $x_t^1$ is decreasing for all times.  On the other hand, $x_t=(x_t^1, x_t^2)$ has generator 
\begin{align*}
L= -(x^2)^2 \partial_{x^1} + \partial_{x^2}^2=: X_0 + X_1^2,  
\end{align*}        
which satisfies the parabolic H\"{o}rmander condition on $\RR^2$ since 
$$
X_1=\partial_{x^2}\,\, \text{ and } \,\,  [X_1, [X_1,X_0]]=-2 \partial_{x^1}
$$ 
span the tangent space at all points in $\RR^2$.  Also, observe that the process initiated at $x\in U$ can only exit $U$ on $\partial U \setminus I$.  
Hence, $U$ is interior regular since every point on $\partial U \setminus I$ is regular.         
\end{Example}

Next, we clarify a the relationship between boundary regular and nice.  
\begin{Proposition}
\label{prop:niceimpreg}
Suppose $x_* \in \partial U$ is nice and that $\partial_t \pm L$ and $ \partial_t \pm L^*$ are hypoelliptic on $(0, \infty) \times B_\delta(x_*)$ for some $\delta >0$ such that $B_\delta(x_*) \subset \X$.  
Then, $x_*$ is regular.  
\end{Proposition}

\begin{proof}
Let $x_*\in \partial U$ be nice and fix $\delta>0$ such that $\partial_t \pm L^*$, $\partial_t \pm L^*$ are hypoelliptic on $(0, \infty)\times B_\delta(x_*)$. Choose an open neighborhood  $U_{x_*}$ of $x_*$ such that there exists a function $w\in C^2(U_{x_*})$ satisfying properties (i) and (ii) in \cref{def:nice}.  Without loss of generality, we may assume that $V:=U_{x_*}=B_\delta(x_*)$.    
It follows that the distribution $\mu_t$  of the stopped process 
$x_t^{V}:=x_{t\wedge \tau_{V}}$ (see \cref{thm:kol}) satisfies Fokker-Planck equation 
$(\partial_t + L^*)(\mu_t) = 0$ in the sense of distributions on $(0, \infty)\times V$.  Thus, when restricted to subsets of $(0, \infty) \times V $, $\mu_t$ is absolutely continuous with respect to Lebesgue measure with density $p_t$ which is smooth on $ (0, \infty)\times V$.  In other words, the law of $x_t^V$ restricted to subsets of $V$, for fixed $t$, is absolutely continuous with respect to Lebesgue measure on $V$.

If $x_*$ is not regular, then $\PP_{x_*}\{\overline{\tau} > 0\} = 1$.  Thus, there exists $\epsilon >0$ such that $B_{\epsilon}(x_*) \subset V$ and  $\PP_{x_*}\{ \tau_{B_{\epsilon}(x_*)}< \overline{\tau} \}>0$, for otherwise 
the distribution of $x_t^V$ started at $x_*$ would have non-zero mass concentrated at $x_*$, violating the absolute continuity of $\mu_t$ above.   
For $\sigma(t):= t\wedge \tau_{B_{\epsilon}(x_*)} \wedge \overline{\tau}$,  Dynkin's formula (see \eqref{eqn:Dynkin}) and positivity of $w$ on $\partial B_{\epsilon}(x_*)$ gives
\begin{align*}
c \PP_{x_*} \{ \tau_{B_\epsilon(x_*)} < \overline{\tau} \wedge t \} \leq \E_{x_*} w(x_{\sigma(t)}) \leq w(x_*) =0  
\end{align*}    
for some constant $c>0$ independent of $t$.  Passing $t\rightarrow \infty$ we obtain $\PP_{x_*} \{ \tau_{B_\epsilon(x_*)} < \overline{\tau} \} =0$, a contradiction.  
\end{proof}

\begin{Remark}
\label{rem:noisepoint}
Fix $x_* \in \partial U$ and  as in \cref{prop:niceimpreg} suppose that $\partial_t \pm L$ and $ \partial_t \pm L^*$ are hypoelliptic on $(0, \infty) \times B_\delta(x_*)$ for some 
$\delta >0$ such that $B_\delta(x_*) \subset \X$.  
 Moreover, suppose that $\partial U$ satisfies the \emph{exterior sphere condition} at $x_*$; that is, there exists $\lambda >0$ such that 
 $x_*\in \overline{B_\lambda(x')}$  and $B_\lambda(x') \subseteq \overline{U}^c$, where $B_\lambda(x')$ is
the open ball centered at $x' = x_* + \lambda \nu(x_*)$ with $\nu(x_*)$ being a unit exterior normal vector to $\partial U$ at $x_*$. 
It is standard to see that we can decrease $\lambda$ in the definition of exterior sphere condition if necessary. 
 
Then, if the noise has a non-zero component in  $\nu(x_*)$ at $x_*$, then $x_*$ is regular, see for example~\cite{Bony_69}.  
Such an assumption in fact implies that $x_*$ is nice.  
Indeed, suppose 
\begin{align}
\label{eqn:normal}
\sum_{i,j=1}^\mathfrak{m} (\sigma(x_*) \sigma^T(x_*))_{ij} v_i(x_*) v_j(x_*) >0 .  
\end{align}      
Then, by choosing $\beta >0$ large enough, the function 
\begin{align}
w(x)= e^{-\beta|x'-x_*|^2} - e^{-\beta |x'-x|^2}
\end{align}
satisfies \cref{def:nice} at $x_*$ on $U_{x_*} = B(x_*, |x' - x_*|/2) \cap U$ if $x' = x_* + \lambda \nu(x_*)$ and $\lambda > 0$ is sufficiently small.  
Note that \eqref{eqn:normal} and  large $\beta >0$ allows one to to disregard  
terms in $Lw$ of order $\beta$. 
\end{Remark}

\subsection{Modification of the domain}
We conclude this section by restating~\cite[Corollary 7.10]{CFH_21} in our context.  Intuitively, it asserts that starting from a \emph{reasonable} domain $U$, provided there is noise in a fixed direction 
for all points on the boundary $\partial U$ (see \cref{rem:noisepoint}), one can slightly  modify $U$ to produce an approximate  domain $V$ which is boundary regular.  
Thus, by \cref{prop:bregimintreg}, $V$ is both interior regular and boundary regular.

In order to state the result, we say that $U$ has \emph{non-flat} boundary $\partial U$ if for each $x_* \in \partial U$ and every $r>0$, the set $\partial U \cap B_r(x_*)$ is not a subset of a hyperplane.  Then the following result was proved in~\cite[Corollary 7.10]{CFH_21}.  

\begin{Theorem}
\label{thm:modification}
Suppose that $U\subset \RR^\mathfrak{m}$, $\mathfrak{m} \geq 2$, is a bounded, convex, non-flat domain with $C^1$ boundary $\partial U$.  Suppose, furthermore, that there exists a fixed unit vector $v\in \RR^\mathfrak{m}$ such that $v$ belongs to the column space of $\sigma(x)$ for all $x\in \partial U$.  Then, for every $\epsilon >0$, there exists a non-empty, open convex domain $V_\epsilon \subset U$ with piecewise linear boundary $\partial V_\epsilon$ such that $V_\epsilon$ and $V_\epsilon^c$ are boundary regular for $x_t$ and $|U|-|V_\epsilon| < \epsilon$.  Here, $| \cdot |$ denotes Lebesgue measure on $\RR^\mathfrak{m}$.     
\end{Theorem}

To prove \cref{thm:modification} one defines $V_\epsilon$ to be the interior of the convex hull of a sufficient number of points $x_1, x_2, \ldots, x_{n(\epsilon)} \in \partial U$ selected independently and randomly according to Hausdorff measure on $\partial U$.  One can then show that with probability one, the faces generated by these points on $\partial V_\epsilon$ are not parallel to the fixed vector $v$ in \cref{thm:modification}, so that the process exits $\overline{V_\epsilon}$ instantaneously when started there.  The fact that the volumes of $V_\epsilon$ and $U$ are close is intuitive provided $n(\epsilon)$
is large enough, which was rigorously proved in ~\cite{SW_03}.

\section{Interior regularity of functionals of $\tau$}
\label{sec:closed}

Our goal in this section is to obtain interior smoothness of various functionals of the first positive exit time $\tau$ from $U$.  Given the relevant hypoellipticity, the key first step  is to show that the law of the stopped process $\tilde{x}_t$,  introduced in \cref{sec:stopped},  when restricted to Borel subsets of $U$, denoted by $\B_U$, has a density $\tilde{p}_t(x,y)$ with respect to Lebesgue measure on $U$.  Furthermore, the mapping $(t,x,y) \mapsto \tilde{p}_t(x,y) \in C^\infty((0, \infty) \times U \times U)$ and satisfies, respectively, the forward and backward Kolmogorov equations in the classical sense:
 \begin{align}
 \label{eqn:forkol}
 &\partial_t \tilde{p}_t(x,y) = L_y^* \tilde{p}_t(x,y) \,\,\, \text{ on } \,\,\, (0, \infty) \times U, \, x\in U \text{ fixed},\\
  \label{eqn:backkol}&\partial_t \tilde{p}_t(x,y) = L_x \tilde{p}_t(x,y) \,\,\, \text{ on } \,\,\, (0, \infty) \times U, \, y\in U \text{ fixed}.   \end{align}      
This is done in \cref{sec:kol}.  In \cref{sec:exittimes},  we deduce interior smoothness and the equations satisfied by
\begin{align}
x\mapsto \E_x \varphi(\tau): U\to \RR
\end{align}
for some choices of smooth $\varphi:[0, \infty) \to \RR$, for example $\varphi(x)=x^k$, $k\in \N$ or $\varphi(x)= \mathbf{1}_{(t, \infty)}(x)$ for fixed $t$. 
 
 First, however, we need to establish a few auxiliary results.   
\subsection{Auxiliary results}
\begin{Lemma}
\label{lem:approx}
Suppose that conditions  \hyperlink{(U0)}{{\bf(U0)}} and \hyperlink{(NE)}{{\bf(NE)}} are satisfied.   Then, the following assertions hold true.  
\begin{itemize}
\item[(i)]  For any $t>0$, $x\in \X$, and $M > 0$ we have 
$$
\lim_{y \to x}\E \bigg\{\sup_{s\in [0, t]} |x_s(x)- x_s(y)|^2\wedge M\bigg\} = 0 \,,
$$ 
where $s \mapsto x_s(z)$  satisfies~\eqref{eqn:sde} 
with $x_0(z) = z$. 
\item[(ii)]  If $x_*\in \partial U$ is regular, then for any $\epsilon, \delta >0$, there is $\rho > 0$ such that 
$\PP_{x}  \{ \tau \geq \delta \} \leq \epsilon$ for any $x \in U$ with $|x - x_*| < \rho$.  That is, as $x\rightarrow x_* \in \partial U, x\in U$, $\tau(x)\rightarrow \tau(x_*)$ in probability.       
\end{itemize}
\end{Lemma}

\begin{proof}

The proof of part (i) follows a standard Gr\"{o}nwall-type comparison argument.  Let $\tau_k = \tau_{\X_k}$ and  we denote the dependence of $\tau_k$ on 
the initial condition $x$ of $x_t$ by writing $\tau_k(x)$.  Observe that for $x,y \in \X$ and $s\leq \tau_k(x) \wedge \tau_k(y)$, the processes $x_s(x)$ and $x_s(y)$ respectively agree pathwise with processes $x_{s,k}(x)$ and $x_{s,k}(y)$ satisfying It\^{o} SDEs with globally Lipschitz coefficients (cf. \cref{rem:nonexpl}).  Thus if $\Delta_t= x_t(x) - x_t(y)$ and $\Delta_{t,k}=x_{t,k}(x)-x_{t,k}(y)$, then for any $\sigma \leq t\wedge \tau_k(x) \wedge \tau_k(y) $ with $t\geq 0$ deterministic, we have the estimate
\begin{align*}
\E \sup_{s\in [0, \sigma]} |\Delta_s|^2 \leq \E \sup_{s\in [0, t]} |\Delta_{s,k}|^2.  
\end{align*} 
Hence, to estimate $\E \sup_{s\in [0, \sigma]} |\Delta_s|^2$ it suffices, for a given deterministic time $t\geq 0$ to estimate $\E \sup_{s\in [0, t]} |\Delta_s|^2$, where $\Delta_t=x_t(x)-x_t(y)$ satisfies
\begin{equation}
\Delta_t =x-y+ \int_0^t [b(x_s(x)) - b(x_s(y)) ] \, ds + \int_0^t [\sigma(x_s(x)) - \sigma(x_s(y))]\, d\mathbf{W}_s
\end{equation} 
with  globally Lipschitz $b, \sigma$ on $\RR^\mathfrak{m}$ and Lipschitz constants depending on $k$.  
Hence, there exists a constant $Q_k>0$ such that 
\begin{align*}
|b(x)-b(y)| + |\sigma(x)-\sigma(y)| \leq Q_k|x-y| \,\,\, \text{ for all } \,\,\, x,y \in \RR^\mathfrak{m}.  
\end{align*}
Using  Doob's maximal inequality and It\^ o isometry, we have 
\begin{align*}
\E\sup_{s\in[0,t]}|\Delta_s|^2 &\leq 9|x-y|^2 + 9 Q_k^2 \E\sup_{s \in[0, t]} \left( \int_0^s |x_u(x)-x_u(y)| \, du\right)^2\\
&\qquad +9 \E\sup_{s\in[0, t]}  \left| \int_0^s  (\sigma(x_u(x)) - \sigma(x_u(y)) )d\mathbf{W}_u \right|^2 \\
&\leq 
9|x-y|^2+ C_k(  t) \int_0^t \E \sup_{s \in [0, u]}|\Delta_s|^2 du  
\end{align*}
for some constant $C_k(t)$ depending only on $k,t$.  
Consequently by Gr\"{o}nwall's inequality, we have 
\begin{equation}\label{eqn:swb}
\E \sup_{s\in[0, t]}|\Delta_s|^2 \leq  9|x-y|^2 e^{C_k(t) t} =: |x-y|^2 D_k(t).\end{equation}

Now if $b, \sigma$ are no longer globally Lipschitz and $\delta_k = \text{dist}(\X_k, \partial \X_{k+1})>0$, 
we have by Chebyshev's inequality and \eqref{eqn:swb}
\begin{align*}
&\E \sup_{s\in[0, t]}|\Delta_s|^2\wedge M\\
 &= \E\sup_{s\in[0, t]}|\Delta_s|^2\wedge M(\textbf{1}\{\tau_k(x) < t\} + \textbf{1}\{\tau_k(x) \geq t\})\\
&\leq M \PP_x\{\tau_k < t\} +  \E_x\sup_{s\in[0, t ]}|\Delta_{s,k}|^2 \mathbf{1} \{\tau_k(x) \geq t, \tau_{k+1}(y) \geq t\} + M \PP \{ \sup_{s\in[0, t ] } |\Delta_{s,k}|^2 > \delta_k^2 \} \\
&\leq  M \PP_x\{\tau_k < t\} +  D_k(t) |x-y|^2 \,  + \frac{|x-y|^2 D_k(t)}{\delta_k^2} M.
\end{align*}
Passing $y\to x$ and then $k \to \infty$ using nonexplosivity of $x_t(x)$, we obtain (i). 

To obtain (ii), suppose there exist $\epsilon_0, \delta_0 > 0$ and
 a sequence $\{x_n\}\subset U$ with $x_n \rightarrow x_*$ such that 
\begin{align*}
\PP_{x_n} \{ \tau \geq \delta_0 \} \geq \epsilon_0
\end{align*}
for all $n$.  For any $\epsilon >0$, let $U^+_\epsilon$ be given by
\begin{align}
\label{eqn:dfup}
U_{\epsilon}^{+}= \{ x\in \X \, : \, \text{dist}(x, \overline{U}) < \epsilon\}.
\end{align}
By path continuity, $\overline{\tau}= \lim_{\epsilon \rightarrow 0} \tau_{U_\epsilon^+}$ almost surely.  Since $x_*$ is regular, there is $\epsilon \in (0, 1)$ so that 
\begin{align*}
\PP_{x_*} \{ \tau_{U_\epsilon^+}< \delta_0/2 \} \geq 1- \epsilon_0/2 \,,
\end{align*} 
and consequently for any $n$
\begin{align*}
\PP \{ \tau(x_n) \geq \delta_0, \tau_{U_{\epsilon}^+}(x_*) < \delta_0/2 \} \geq \epsilon_0/2  \,,
\end{align*}
where the initial condition is indicated in the stopping times above.  By conclusion (i) with $M = 1$, and $\epsilon <1$
for any  large enough $n$ one has $\PP\{ \sup_{s\in [0, \delta_0/2]} |x_s(x_n) -x_s(x_*)| > \epsilon/2\}\leq \epsilon_0/4$.  But then 
\begin{align*}
\frac{\epsilon_0}{2} &\leq  \PP \left\{ \tau(x_n) \geq \delta_0, \tau_{U_{\epsilon}^+}(x_*) < \frac{\delta_0}{2}, \sup_{s\in [0, \delta_0/2]}|x_s(x_n) -x_s(x_*)| \leq \frac{\epsilon}{2} \right\}\\
& \qquad + \PP \left\{\sup_{s\in [0, \delta_0/2]} |x_s(x_n) -x_s(x_*)| > \frac{\epsilon}{2} \right\} \leq 0 + \frac{\epsilon_0}{4}, 
\end{align*}  
a contradiction.  
\end{proof}

In addition to \cref{lem:approx}, we need one more auxiliary result.  Although this result is basic, it is used repeatedly throughout the paper.    
\begin{Lemma}
\label{lem:red}
Let $V \subset \RR^{k}$ be open and let $\LL$ and $\LL_n$, $n\in \N$, be linear second-order differential operators with $C^\infty(V)$ coefficients.  Suppose that $\{v_n\}$ be a uniformly bounded sequence of measurable functions $v_n:V\to \RR$ and that the following conditions are satisfied:\begin{itemize}  
\item[-] For every $\phi \in C_0^\infty(V)$, $(\LL_n -\LL)^* \phi \to 0 $ as $n\to \infty$ in $L^1(V, dx)$ where the $*$ denotes the formal adjoint with respect to $L^2(V, dx)$-inner product;
\item[-] For some measurable function $f$ on $V$ which is bounded on compact subsets of $V$, $\LL_n v_n \to -f$ as $n\rightarrow \infty$ on $V$ in the sense of distributions.
\item[-]  $v_n$ converges to $v_\infty$ in the sense of distributions on $V$ as $n \rightarrow \infty$.   
 \end{itemize}
 Then, $\LL v_{\infty}=-f$ on $V$ in the sense of distributions.  Furthermore, if $\LL$ is hypoelliptic on $V$ and $f\in C^\infty(V)$, then $v_\infty \in C^\infty(V)$ and $\LL v_{\infty} = -f$ on $V$ in the classical sense.            
\end{Lemma}

\begin{proof}
For fixed $\varphi \in C_0^\infty(V)$,  conditions of the statement and the Dominated Convergence Theorem imply
\begin{align*}
\int_{V}  v_{\infty} \LL^* \varphi \, dx =\lim_{n\rightarrow \infty} \int_V v_n \LL^* \varphi \, dx &= \lim_{n \to \infty} \int_{V} v_n \LL_n^* \varphi - \lim_{n \to \infty} \int_{U} v_n (\LL_n - \LL)^* \varphi \, dx \\
&=\int_V -f \varphi \, dx.  
\end{align*} 
Thus $\LL v_{\infty} = -f$ on $V$ in the sense of distributions.  The remaining assertion is an immediate consequence of hypoellipticity.  
\end{proof}

\subsection{The forward and backward Kolmogorov equations for the stopped process}
\label{sec:kol}
Recall that if $x_t$ is nonexplosive, then the stopped process $\tilde{x}_t$, see \eqref{eqn:dfsp}, is a well-defined, continuous-time Markov process distributed on $\overline{U}$.  Moreover, the associated semigroup and transitions are respectively denoted by $\tilde{\mathcal{P}}_t$ and $\tilde{\mathcal{P}}_t(x, \, \cdot \,)$.

Our main result in this section is the following.  
\begin{Theorem}
\label{thm:kol}
Suppose that \hyperlink{(U0)}{{\bf(U0)}}, \hyperlink{(NE)}{{\bf(NE)}}, and \hyperlink{(L2)}{{\bf(L2)}} are satisfied.  
\begin{itemize}
\item[(i)]  For all $x\in U$ and $t>0$, the restriction of the measure $\tilde{\mathcal{P}}_t(x, \, \cdot \,)$ to Borel subsets of $U$ is absolutely continuous with respect to Lebesgue measure on $U$ with density $\tilde{p}_t(x,y)$, and for fixed $x\in U$ the mapping $(t,y) \mapsto \tilde{p}_t(x,y)\in C^\infty((0, \infty) \times U) $.  Furthermore, the forward Kolmogorov equation~\eqref{eqn:forkol} is satisfied in the classical sense. 
\item[(ii)]  If \hyperlink{(L3)}{{\bf(L3)}} is furthermore satisfied, then $(t,x,y) \mapsto \tilde{p}_t(x,y)$ belongs to $C^\infty((0, \infty) \times U\times U)$ and the backward Kolmogorov equation~\eqref{eqn:backkol} is satisfied in the classical sense.
\end{itemize}
\end{Theorem}
 
 \begin{Remark}
The above result is natural and understood by experts in the field.  However, we found it difficult to locate a complete proof of part (ii) as it is much more subtle than part (i).  While part (i) follows almost immediately from Dynkin's formula, part (ii) requires several nontrivial approximations and steps.  
\end{Remark}
 
\begin{Remark}
Intuitively the result above holds because, while the measure $\tilde{\mathcal{P}}_t(x, \, \cdot \,)$ has a singular component on the boundary if the process exits $U$ by time $t$ with positive probability, 
when restricted to subsets of $U$ this singularity is not seen.    
 \end{Remark}
 
 \begin{Remark}
 \label{rem:extn}
 Note that in the statement above we may take $U=\X$, in which case $\tilde{x}_t$ coincides with the original process $x_t$ on $\X$, and thus the statement is about the law of $x_t$ in $\X$.    
 \end{Remark}

 \begin{proof}[Proof of \cref{thm:kol}]
To prove part (i), fix $t>0$, $x\in U$, and take any $\psi \in C^\infty([0, t] \times U)\cap C^\infty_0((0, t) \times U)$.  Thus, in particular, $\psi$ is compactly supported in $(0, t) \times U$.  
Then, Dynkin's formula yields
\begin{align*}
0=\E_x \psi(t, \tilde{x}_t)= \E_x \psi(t, x_{t\wedge \tau_0}) &= \psi(0, x) + \E_x \int_0^{t\wedge \tau_0} (L+ \partial_s) \psi(s,x_{s}) \, ds\\
&= 0+ \E_x \int_0^t (L+ \partial_s) \psi(s,\tilde{x}_s) \, ds \,,
\end{align*}   
where $\tau_0$ from \eqref{eqn:dotz} is the first exit time from $U$. 
Hence, 
$$
\E_x \int_0^t (L+\partial_s) \psi(s, \tilde{x}_s) =\int_0^t \int_U(L+\partial_s)  \psi(s, y) \tilde{p}_s (x, dy) ds =0
$$
 for any $\psi \in C_0^\infty((0, t) \times U)$.  That is, $(L^*_y-\partial_t)(\tilde{p}_t(x,dy))=0$ on $(0, \infty) \times U$ in the sense of distributions. By hypoellipticity of $\partial_t-L^*$ on $(0, \infty) \times U$ as in \hyperlink{(L2)}{{\bf(L2)}}, part (i) follows. 

Part (ii) is more involved as there are several layers of approximations.  We divide the proof into two main steps. 

\underline{\emph{Step 1}}.  Let $\psi \in C_0^\infty(U)$ and consider the mapping $(t,x)\mapsto u(t,x):= \tilde{\mathcal{P}}_t \psi(x)$.  We claim that 
$u\in C^\infty((0, \infty) \times U)$ and $\partial_t  u= L u$ in the classical sense on $(0, \infty) \times U$.  

\underline{\emph{Step 2}} .We prove part (ii) by taking a sequence of approximations of $\tilde{p}_t(x,y)$ of the form $\tilde{\mathcal{P}}_t \psi_n(x)$, $\psi_n \in C_0^\infty(U)$.  

\emph{Proof of Step 1}.  First observe that since $x_t$ is nonexplosive, the stopped process $\tilde{x}_t= x_{t\wedge \tau_0}$ with initial condition $\tilde{x}_0 =x\in U$ satisfies the It\^{o} SDE  
\begin{align*}
d \tilde{x}_t= b(\tilde{x}_t) \mathbf{1}_U(\tilde{x}_t) \, dt + \sigma(\tilde{x}_t) \mathbf{1}_U(\tilde{x}_t)\, d\mathbf{W}_t .
\end{align*}
We first show $u\in C^\infty((0, \infty) \times U)$ and $\partial_t  u= L u$ in the classical sense on $(0, \infty) \times U$ in the special case when $b, \sigma \in C^\infty(\X)$ are bounded with bounded derivatives of all orders.

Recall that $U_n$ is a sequence of bounded open sets with $\overline{U}_n \subset U$ and $U_n\uparrow U$ as $n\to \infty$.  Suppose that $\varphi_n \in C^\infty(\overline{U};[0,1])$ satisfies $\varphi_n(x)= 1$ on $U_n$ and $\varphi_n(x) =0$ for $x\in U_{n+1}^c$.  Consider a sequence of approximating processes $x_t^n$ with $x_0^n=x$ and 
\begin{align}
dx_t^n&= b(x_t^n) \varphi_n(x_t^n) \, dt + \sigma (x_t^n) \varphi_n(x_t^n) \, d\mathbf{W}_t.  
\end{align}        
\textbf{Claim.}
If $b, \sigma \in C^\infty(\X)$ are bounded with bounded derivatives of all orders, then
for any $t>0$ fixed, 
\begin{equation}\label{eqn:efcc}
\E_x \sup_{r\in [0, t]} |\tilde{x}_r- x_r^n|^2\to 0 \,\,\text{ as } \,\, n\to \infty.
\end{equation}

\begin{proof}[Proof of \eqref{eqn:efcc}.]
Observe that 
\begin{align*}
\tilde{x}_t- x_t^n &= \int_0^t \varphi_n(x_s^n)[ b(\tilde{x}_s)- b(x_s^n) ] \, ds  + \int_0^t \varphi_n(x_s^n)[ \sigma(\tilde{x}_s)- \sigma(x_s^n) ] \, d\mathbf{W}_s\\
&\qquad +
\int_0^t (1- \varphi_n(x_s^n)) \mathbf{1}_U (\tilde{x}_s) b(\tilde{x}_s) \, ds + \int_0^t (1- \varphi_n(x_s^n)) \mathbf{1}_U (\tilde{x}_s) \sigma(\tilde{x}_s) \, d\mathbf{W}_s .\end{align*}  
Hence, squaring both sides, taking the supremum and using Doob's maximal inequality we find that for all $t \leq T$
\begin{align*}
\E \sup_{r\in [0, t]} |\tilde{x}_r - x_r^n|^2 & \leq C_1 \int_0^t \E \sup_{r\in [0, s]} |\tilde{x}_r- x_r^n |^2 \, ds + C_2 \int_0^t \E   [(1- \varphi_n(x_s^n)) \mathbf{1}_U (\tilde{x}_s)]^2 \, ds \,,
\end{align*}
where the constants $C_1=C_1(T, \|b\|_{\text{Lip}}, \|\sigma \|_{\text{Lip}})>0$ and $C_2=C_2(T, \|b\|_{L^\infty} , \|\sigma \|_{L^\infty})>0$ do \emph{not} depend on $n$.  Gr\"{o}nwall's inequality then implies that for all $t\leq T$
\begin{align*}
\E \sup_{r\in [0, t]} |\tilde{x}_r - x_r^n|^2 & \leq C_2  e^{t C_1} \int_0^t \E   [(1- \varphi_n(x_s^n)) \mathbf{1}_U (\tilde{x}_s)]^2 \, ds  \,.
\end{align*}
Next, observe that if $\tilde{x}_s(\omega) \in U$ for some $s \geq 0$, then by path continuity and the definition of the stopped process, $\tilde{x}_u(\omega) \in U_{n_0}$ for all $u\in [0, s]$ for some 
$n_0=n_0(\omega, s) \in \N$.  Hence, by definition, $x_u^n(\omega) \equiv\tilde{x}_u(\omega)$ for all $u\in [0, s]$ and $n\geq n_0$, and consequently by the Bounded Convergence Theorem 
\begin{align*}
\lim_{n \to \infty} \int_0^t \E   [(1- \varphi_n(x_s^n)) \mathbf{1}_U (\tilde{x}_s)]^2 \, ds = 0 \,.
\end{align*}
Hence, $ \E \sup_{r\in [0, t]} |\tilde{x}_r - x_r^n|^2 \to 0$ as $n\to \infty$, establishing \eqref{eqn:efcc}.    
\end{proof}

Let $u_n(t,x) := \E_x \psi(x_t^n)$.  It follows that for fixed $t>0$, the mapping $x\mapsto u_n(t,x)$ belongs to the space $C^2(\RR^\mathfrak{m})$~\cite[Theorem 5.5]{Friedman_75}.  Define an extended version $A_n$ of the generator of $x_t^n$ by
\begin{align}\label{eqn:gld}
A_n g(x):= \lim_{t\rightarrow 0} \frac{\E_x g(x_t^n)- g(x) }{t}\,,
\end{align}
where the domain of $A_n$ consists of all measurable functions $g: \RR^\mathfrak{m}\to \RR$  such that the limit in \eqref{eqn:gld} exists for each $x$.  
Next, by~\cite[Theorem 8.1.1]{Oksen_13}, for fixed $x\in \RR^\mathfrak{m}$, $t\mapsto u_n(t,x) \in C^1(\RR^\mathfrak{m})$ and by Markov property $u_n$ satisfies on $(0, \infty)\times \RR^\mathfrak{m}$ the generalized backward equation
\begin{align*}
\partial_t u_n= A_n u_n .    
\end{align*} 
Since $x\mapsto u_n(t,x)$ is globally $C^2$, Dynkin's formula implies that $A_n u= L_n u$ where $L_n$ is the \emph{classical} generator of $x_t^n$, which is a second-order differential operator with $C^\infty(\RR^\mathfrak{m})$ coefficients.  Moreover, these coefficients agree with the coefficients of $L$ on $U_n$.  
In particular, by \hyperlink{(L3)}{{\bf(L3)}},  $L- \partial_t$ is hypoelliptic on $(0, \infty) \times U_n$, and therefore $u_n\in C^\infty((0, \infty) \times U_n)$ and $\partial_t u_n= L u_n$ on $(0, \infty) \times U_n$ in the classical sense.  
To obtain a similar result for $u(t,x)= \tilde{\mathcal{P}}_t \psi(x)=\E_x \psi(\tilde{x}_t)$, we note that $u_n$ is uniformly bounded.  Furthermore, $u_n \to u$ pointwise as $n\to \infty$ since by Claim
\begin{align}\label{eqn:lcp}
|u_n(t,x)- u(t,x)|^2 \leq \| \psi \|_{\text{Lip}}^2 \E_x |\tilde{x_t} - x_t^n|^2 \to 0
\end{align}
as $n\to \infty$.  By \cref{lem:red},  $u\in C^\infty((0, \infty) \times U)$ and $\partial_t u = L u$ in the classical sense on $(0, \infty) \times U$.   Thus Step 1 is completed in the case when $b, \sigma$ have bounded derivatives of all orders. 

To complete Step 1, assume that $b \in C^\infty(\X; \RR^\mathfrak{m})$ and $\sigma \in C^\infty(\X; M_{ \mathfrak{m} \times r})$ are not necessarily bounded and $x_t$ is nonexplosive.  For the sequence
 $\{U_n\}$ above, note that the restrictions of $b, \sigma$ to $U_n$ can be extended to functions on $\RR^\mathfrak{m}$ which are bounded with bounded derivatives all orders.  Setting $\tau_n= \inf \{ t> 0 \, : \, \tilde{x}_t \notin U_n\}$, we have shown above that if $\tilde{x}_t^n:= \tilde{x}_{t\wedge \tau_{n}}$, then $(t,x) \mapsto \tilde{u}_n(t,x):= \E_x \psi(\tilde{x}_t^n) $ satisfies $\tilde{u}_n \in C^\infty((0, \infty) \times U_n)$ and  $\partial_t  \tilde{u}_n= L \tilde{u}_n$ on $(0, \infty) \times U_n$ in the classical sense.  However, since 
$\psi$ is bounded, then $\tilde{u}_n$ is  also uniformly bounded in $n$ and by \eqref{eqn:lcp}, 
$\tilde{u}_n\to u$ pointwise on $(0, \infty) \times U$.  Step 1 now follows after applying \cref{lem:red}.  

\emph{Proof of Step 2}.  Let $\chi \in C_0^\infty(\RR^\mathfrak{m};[0, \infty))$ with $\int_{\RR^\mathfrak{m}}\chi(x) \, dx =1$.  Let $\chi_\epsilon(x)= \epsilon^{-\mathfrak{m}} \chi(\epsilon^{-1} x)$ and extend $\tilde{p}_t(x,y)$ to be zero for all $(t,x,y) \notin (0, \infty) \times U\times U$. For any $(t,x,y) \in (0, \infty) \times \RR^\mathfrak{m}\times \RR^\mathfrak{m}$, define
\begin{align}
\label{eqn:formueps}
u_\epsilon(t,x,y):= \int_{\RR^\mathfrak{m}} \tilde{p}_t(x,z) \chi_\epsilon( y-z) \, dz.
\end{align}
Consider any bounded open $V\subset U$ with $\overline{V} \subset U$.  
For each fixed $y \in \overline{V}$, there is $\epsilon_0(y) > 0$ small enough such that $z\mapsto \chi_\epsilon(y-z)\in C_0^\infty$ is compactly supported in $U$ for all $\epsilon \in (0, \epsilon_0(y)]$.  It follows by compactness of $\overline{V}$, that we may choose $\epsilon_0>0$ such that $z\mapsto \chi_\epsilon(y-z)$ is compactly supported in $U$ for all $y\in \overline{V}$ and all $\epsilon \in (0, \epsilon_0]$.  Thus, for any $\epsilon \in (0, \epsilon_0]$ and $(t,x) \in (0, \infty) \times U$, the formula~\eqref{eqn:formueps} makes sense and $y\mapsto u_\epsilon(t,x,y)\in C^\infty(V)$.   Also, by Step 1, for each $y \in V$ and $\epsilon \in (0, \epsilon_0]$,
$(t,x) \mapsto u_\epsilon(t, x,y) \in C^\infty((0, \infty) \times U)$ and $\partial_t u_\epsilon (t,x,y) = L_x u_\epsilon (t,x,y)$ on $(0, \infty) \times U$ in the classical sense.  
To conclude the result, it suffices to show
for any $\varphi \in C_0((0, \infty) \times U \times V)$ with $\varphi$ supported in $K\subset (0, \infty) \times U \times V$ compact, that
\begin{align*}
\lim_{\epsilon \to 0} \int_K \varphi(t,x,y)[ u_\epsilon(t,x,y)- \tilde{p}_t(x,y)] dx \, dy \, dt = 0 \,.
\end{align*}
  Indeed, simply replace $\varphi$ in the above formula by $(\partial_t- L_x)^* \varphi$, where $\varphi$ is smooth with compact support.  

First, since $\int_{\RR^\mathfrak{m}} \chi (z) \, dz =1$, a standard substitution gives for all $\epsilon \in (0, \epsilon_0]$ and $y\in V$
\begin{align*}
u_\epsilon(t,x,y)- \tilde{p}_t(x,y)= \int_{\RR^\mathfrak{m}} [ \tilde{p}_t(x, y-\epsilon z) - \tilde{p}_t(x, y)] \chi(z) \, dz.  
\end{align*}  
Next, let $\hat{p}_t(x,y)=\tilde{p}_t(x,y)$ on $K$ and $\hat{p}_t(x,y)=0$ otherwise.  
By Tonelli's theorem, $(t,x,y) \mapsto \hat{p}_t(x,y) \in L^1(\RR \times (\RR^\mathfrak{m})^2)$.  Moreover, 
\begin{align*}
&\bigg|\int_K \varphi(t,x,y)[u_\epsilon(t,x,y) - \tilde{p}_t(x,y)] \, dx \, dy \, dt\bigg| \\
&\leq \| \varphi\|_{L^\infty} \int_{\RR\times (\RR^\mathfrak{m})^3}  | \hat{p}_t(x, y-\epsilon z) - \hat{p}_t(x, y)|\chi(z) dz \, dx \, dy \, dt.\end{align*}
Since $\hat{p}_t \in L^1(\RR\times (\RR^\mathfrak{m})^2)$, for any $\delta >0$ there is $\psi \in C_0^\infty( (\RR^\mathfrak{m})^3)$ within $L^1$ distance $\delta$ of $\hat{p}_t$.  Thus, 
by translation invariance and Fubini's theorem   
\begin{align*}
&\bigg|\int_K \varphi(t,x,y)[u_\epsilon(t,x,y) - \tilde{p}_t(x,y)] \, dx \, dy \, dt\bigg| \\
&\leq 2\delta \| \varphi\|_{L^\infty} + \| \varphi\|_{L^\infty}\int_{(\RR^\mathfrak{m})^4}  |\psi(t,x, y-\epsilon z) - \psi(t,x, y)|\chi(z) dz \, dx \, dy \, dt.  
\end{align*}
Since $\psi$ is compactly supported, passing $\epsilon \to 0$, using Dominated Convergence theorem, and then $\delta \to 0$ finishes the proof of \cref{thm:kol}(ii). 
\end{proof}

\subsection{Interior smoothness of functionals of the first positive exit time}
\label{sec:exittimes}

In this section we analyze quantities of the form
\begin{align}
\label{eqn:phitau}
 \E_x \varphi(\tau), \,\, x\in U,
\end{align}
 for various functions $\varphi$ and $\tau$ as in \eqref{eqn:exitpos}.  The first two choices
  $$\varphi(x) = \mathbf{1}\{ |x| < \infty\} \,\,\, \text{ and } \,\,\, \varphi(x)=x,$$ 
  that respectively yield  in \eqref{eqn:phitau}
  \begin{align}
\label{eqn:exittimes}
\PP_x \{ \tau < \infty \} \qquad \text{ and } \qquad \E_x \tau 
\end{align}     
 are of significant importance for analyzing recurrence and transience in \cref{sec:recurrence} below.     
Throughout, we assume the parabolic H\"{o}rmander condition \hyperlink{(PH)}{{\bf(PH)}} for simplicity, which in particular implies  that \hyperlink{(L1)}{{\bf(L1)}}--\hyperlink{(L4)}{{\bf(L4)}} are satisfied.

\begin{Proposition}
\label{prop:exitprob}
Suppose that  \hyperlink{(U00)}{{\bf(U00)}}, \hyperlink{(NE)}{{\bf(NE)}}, and \hyperlink{(PH)}{{\bf(PH)}} are satisfied.  For any $(t,x) \in (0, \infty) \times U$, define 
\begin{align}
u_1(t,x)= \PP_x \{ \tau > t\} \qquad \text{ and } \qquad u_2(x)= \PP_x \{ \tau = \infty \}.
\end{align}
Then:
\begin{itemize}
\item[(i)]  $u_1 \in C^\infty((0, \infty) \times U)$ and $\partial_t u_1= L u_1 $ on $(0, \infty) \times U$ in the classical sense;
\item[(ii)]  $u_2 \in C^\infty(U)$ and $Lu_2=0$ on $U$ in the classical sense.  
\end{itemize}
\end{Proposition}

\begin{proof}
For (i), observe that for any $x\in U$, $u_1(t,x) = \PP_x \{ \tau >t \}= \tilde{\mathcal{P}}_t(x, U)$ where $\tilde{\mathcal{P}}_t(x, \, \cdot \,)$ is the transition kernel of the stopped process $\tilde{x}_t= x_{t\wedge \tau_0}$.  Let $K\subset U$ be compact and  by \cref{thm:kol}
$\tilde{\mathcal{P}}_t(x, K)= \int_K \tilde{p}_t(x,y) \, dy$.  Furthermore, by \cref{thm:kol} and the Dominated Convergence Theorem, 
it follows that $(t, x) \mapsto \tilde{\mathcal{P}}_t(x,K)$ is $C^\infty((0, \infty) \times U)$ and $\partial_t \tilde{\mathcal{P}}_t(x,K) = L \tilde{\mathcal{P}}_t(x,K)$ on $(0, \infty) \times U$ in the classical sense.   To obtain the same result for $\tilde{\mathcal{P}}_t(x, U)$, let $K_n \subset U$ be a sequence of compact sets with $K_n \uparrow U$ as $n\rightarrow \infty$.  Then, by the monotone convergence theorem, 
the sequence $(\tilde{\mathcal{P}}_t(x,K_n))_n$ is uniformly bounded and converges pointwise  to $\tilde{\mathcal{P}}_t(x, U)$.  The assertion (i) follows after applying \cref{lem:red}.  

To obtain (ii), first observe that $$u_2(x) = \PP_x \{ \tau = \infty\} = \lim_{t\rightarrow \infty} \tilde{\mathcal{P}}_t(x, U).$$  Thus we seek to apply an argument similar to that in \cref{lem:red}.  To this end, note that
\begin{align*}
\bar{u}(t,x):= \int_{t}^{t+1} \tilde{\mathcal{P}}_s(x, U) \, ds \rightarrow u_2(x) \,\,\,\text{ pointwise as } \,\,\, t\to \infty
\end{align*}
and that $\bar{u}(t,x)$ is uniformly bounded on $[0, \infty) \times U$.  Thus, for any function $\varphi \in C_0^\infty(U)$ and any $t>0$, using integration by parts and the Fubini-Tonelli Theorem we have
\begin{align*}
0 &= \int_t^{t+1} \int_U (L- \partial_s) \tilde{\mathcal{P}}_s(x, U) \varphi(x) \, dx \,ds \\
&= \int_U \bar{u}(t,x) L^* \varphi(x) \, dx - \int_U [\tilde{\mathcal{P}}_{t+1}(x, U) - \tilde{\mathcal{P}}_t(x, U)] \varphi(x) \, dx.  
\end{align*}
By passing $t\rightarrow \infty$ and 
using the Dominated Convergence Theorem,  $u_2$ solves $L u_2=0$ in the sense of distributions on $U$.  Part (ii) follows by hypoellipticity of $L$ on $U$.   
\end{proof}

We next state and prove a regularity result for $\E_x \tau$; that is, for $\varphi(x)=x$ in~\eqref{eqn:phitau}.  
\begin{Proposition}
\label{prop:expexit}
Suppose that  \hyperlink{(U00)}{{\bf(U00)}}, \hyperlink{(NE)}{{\bf(NE)}}, and \hyperlink{(PH)}{{\bf(PH)}} are satisfied and $\PP_x \{ \tau < \infty \} =1$ for all $x\in U$. 
\begin{itemize}
\item[(i)]  If $U_0\subset U$ is bounded open with $\overline{U_0}\subset U$, then $v_0(x):=\E_x \tau_{U_0}$ is bounded on $\overline{U_0}$.  Furthermore, $v_0 \in C^\infty(U_0)$ and $ Lv_0=-1$ on $U_0$ in the classical sense.  
\item[(ii)]  Let $U_0 \subset U$ be an open set and suppose that $v(x):=\E_x \tau $ is finite on a dense set $D\subset U_0$.  Then, $v \in C^\infty(U_0)$ and $v$ solves $Lv=-1$ on $U_0$ in the classical sense.  \end{itemize}   
\end{Proposition}

\begin{Remark}
The technical hypothesis in (ii) that $v$ is finite on a dense set and not  everywhere 
is needed below in \cref{sec:recurrence}.  We defer the proof of \cref{prop:expexit}(ii) until after we prove a hypoelliptic Harnack inequality (\cref{thm:harnack}).   
\end{Remark}

\begin{proof}[Proof of \cref{prop:expexit}(i)]
Let $\sigma= \tau_{U_0}$ and note that $\PP_x\{ \tau< \infty \} =1$ imply that for all $x\in \overline{U_0}$, there exists $t>0, \alpha\in (0,1/2)$ such that $\PP_x\{ \sigma > t \} \leq (1-2\alpha)$.  By \cref{prop:exitprob}(i), there exists $\epsilon =\epsilon(x)>0$ such that $\PP_y\{ \sigma > t \} \leq(1- \alpha)$ for all $y\in B_\epsilon(x)$.  Using compactness of $\overline{U_0}$ and taking a finite subcover, it follows that there exists $t_*>0, \alpha_*\in (0,1)$ such that $\PP_y \{ \sigma > t_* \} \leq (1-\alpha_*)$ for all $y\in \overline{U_0}$.  Since $\tau $ is almost surely finite, so is $\sigma$ by 
$U_0 \subset U$ and
path continuity.  Thus, for any $x\in \overline{U_0}$ we have
\begin{align}
\label{eqn:expsig}
\E_x \sigma = \sum_{m=1}^\infty \E_x \sigma \mathbf{1} \{ \sigma \in [(m-1)t_*, m t_*)\} \leq 3t_*+ \sum_{m=3}^\infty m t_* \PP_x\{ \sigma > (m-2) t_* \}.   
\end{align}
By the Markov property, for $x\in \overline{U_0}$ and $m\in \N$ we have 
\begin{align*}
\PP_x\{ \sigma > m t_*\} = \E_x [\E_x \mathbf{1} \{ \sigma > m t_* \} | \mathcal{F}_{(m-1)t_*} ] &= \E_x \mathbf{1} \{ \sigma > (m-1) t_* \} \PP_{x_{(m-1) t_*}} \{ \sigma > t_* \}\\
& \leq (1- \alpha_*) \PP_x\{ \sigma > (m-1) t_*\}\leq (1- \alpha_*)^m
\end{align*}    
and~\eqref{eqn:expsig} implies for all $x\in \overline{U_0}$
\begin{align*}
\E_x \sigma \leq 3 t_* + \sum_{m=3}^\infty m t_* (1- \alpha_*)^{m-2}, 
\end{align*}
which is finite and bounded independently of $x\in \overline{U_0}$.   

Next,  in order to show that $v_0\in C^\infty(U_0)$ and satisfies $Lv_0=-1$ on $U_0$ in the classical sense, observe that 
\begin{align*}
v_0(x)= \E_x \sigma = \int_0^\infty \PP_x \{ \sigma > t \} \, dt =: \int_0^\infty u(t,x) \, dt.   
\end{align*} 
\cref{prop:exitprob} yields $u\in C^\infty((0, \infty) \times U_0)$ and $L u =\partial_t u$ on $U_0$ in the classical sense.  Thus if $v_{n}(x) = \int_{1/n}^n u(t,x) \, dt$ for $n\in \N$, it follows that $v_{n} \in C^\infty(U_0)$ and for all $x\in U_0$
\begin{align*}
L v_{n}(x) = \int_{\frac{1}{n}}^n L u(t,x) \, dt &= \int_{\frac{1}{n}}^n \partial_t u(t, x) \, dt \\
&= u(n, x) - u(1/n, x) = \PP_x \{ \sigma > n \} - \PP_x \{ \sigma > 1/n \} \to -1  
\end{align*}
as $n\to \infty$ in the sense of distributions on $U_0$.  
Since the sequence $v_n$ is uniformly bounded by bounded $v_0$ and $v_n\rightarrow v_0$ pointwise on $U_0$, the result follows from \cref{lem:red}.  
\end{proof}

We next consider the expression~\eqref{eqn:phitau} for more general $\varphi:[0, \infty) \rightarrow \RR$.  First we need an auxiliary lemma, which is a representation result for the expected value of certain random variables. 
\begin{Lemma}
\label{lem:exprep}
Suppose that \hyperlink{(U00)}{{\bf(U00)}} and \hyperlink{(NE)}{{\bf(NE)}} are met and let $\varphi:[0, \infty) \to \RR$ be $C^1$, strictly increasing and satisfy $\varphi(t) \to \infty$ as $t\to \infty$.  Then 
\begin{align}
\label{eqn:exprep}
\E_x \varphi(\tau) = \varphi(0)+ \int_0^\infty \varphi'(t) \PP_x \{ \tau > t \} \, dt.  
\end{align}    
\end{Lemma}

\begin{proof}
Note that by shifting the formula~\eqref{eqn:exprep}, it suffices to assume that $\varphi(0)= 0$.  Since $\varphi$ is $C^1$, strictly increasing with $\varphi(t) \to \infty$ as $t\to \infty$, $\varphi$ has a $C^1$ inverse $\varphi^{-1}$ mapping $[0, \infty)$ onto $[0, \infty)$.  Moreover, $\varphi^{-1}$ is strictly increasing on $[ 0, \infty)$.  It thus follows that $ \varphi^{-1}(t)\to \infty$ as $t\to \infty$.  Then, 
\begin{align*}
\E_x (\varphi(\tau)\wedge N) = \int_0^N \PP_x \{ \varphi(\tau) > t \} \, dt &=\int_{0}^{N}\PP_x \{ \tau > \varphi^{-1}(t) \} \, dt \\
&=  \int_{0}^{\varphi^{-1}(N)}\varphi'(t) \PP_x \{ \tau > t \} \, dt.\end{align*}  
Passing $N\to \infty$ and using the Monotone Convergence Theorem finishes the proof.  
\end{proof}

Next, we state and prove our main result for general $\varphi$.  

\begin{Proposition}
\label{prop:generalphi}
Suppose that \hyperlink{(U00)}{{\bf(U00)}}, \hyperlink{(NE)}{{\bf(NE)}}, and \hyperlink{(PH)}{{\bf(PH)}} are satisfied.  Let $\varphi\in C^2([0, \infty); \RR)$ be strictly increasing with $\varphi(t) \to \infty$ as $t\to \infty$ with derivative $\varphi'\in C^1([0, \infty); \RR)$ which is strictly increasing with $\varphi'(t) \to \infty$ as $t\to \infty$.  Suppose furthermore that  
$$v_1(x):=\E_x \varphi(\tau) \qquad \text{ and } \qquad v_2(x)= \E_x \varphi'(\tau)$$ are bounded on compact subsets of $U$.  Then $Lv_1= -v_2$ on $U$ in the sense of distributions.    
\end{Proposition}

\begin{Remark}
Compare the equation $Lv_1 = - v_2$ with $Lv=-1$ when $\varphi(x)=x$.  Although such a $\varphi$ does not satisfy the hypotheses above, we still have the analogous conclusion.    
\end{Remark}

\begin{proof}[Proof of \cref{prop:generalphi}]
By \cref{lem:exprep}, we have 
\begin{align*}
v_1(x)= \varphi(0)+ \int_0^\infty \varphi'(t) \PP_x\{ \tau > t \} \, dt \qquad \text{ and } \qquad v_2(x)=\varphi'(0)+ \int_0^\infty \varphi''(t) \PP_x\{ \tau > t \} \, dt. 
\end{align*}
In the latter formula, simply apply \cref{lem:exprep} to $\varphi'$ instead of $\varphi$. 
Consider the sequence of functions on $U$ given by
\begin{align*}
v_{1,n}(x):= \varphi(0) + \int_{1/n}^n \varphi'(t) \PP_x\{ \tau > t \} \, dt \,.
\end{align*}
Observe that by \cref{prop:exitprob}, $v_{1,n} \in C^\infty(U)$ and 
\begin{align*}
L v_{1,n}(x) = \int_{1/n}^n \varphi'(t) L_x (\PP_x\{ \tau > t \}) \, dt&= \int_{1/n}^n \varphi'(t) \partial_t (\PP_x\{ \tau > t \}) \, dt\\
&= \varphi'(n) \PP_x \{ \tau > n \} - \varphi'(1/n) \PP_x \{ \tau > 1/n\} \\
&\qquad - \int_{1/n}^n \varphi''(t) \PP_x \{ \tau > t \} \, dt \,.
\end{align*}
Since for any $x\in U$ we have $\E_x \varphi'(\tau) < \infty$, it follows that 
\begin{equation}
\lim_{n \to \infty} \varphi'(n) \PP_x \{ \tau > n \} \leq  \lim_{n\to \infty} \E_x \varphi'(\tau)  \mathbf{1}\{ \tau \geq n\} = 0
\end{equation}
on $U$.  Furthermore, since $ \PP_x \{ \tau = 0 \} = 0$ for any $x\in U$, it also follows that on $U$
\begin{equation}
\lim_{n \to \infty}  \varphi'(1/n) \PP_x \{ \tau > 1/n\} = \varphi'(0) \PP_x \{ \tau > 0\} = \varphi'(0). \,
\end{equation}
In addition, $\varphi'$ is strictly increasing, and therefore $\varphi'' \geq 0$ and by the monotone convergence theorem,  $L v_{1,n} \to -v_2$ pointwise as $n\to \infty$ on $U$.  
Since $v_1, v_2$ are bounded on compact subsets of $U$ and $v_{1, n} \to v_1$ as $n \to \infty$, \cref{lem:red} implies $L v_1=- v_2$ on $U$ in the sense of distributions. 
\end{proof}

Let us note some more specific examples and results that follow as an immediate consequence of the previous results.

\begin{Corollary}
Suppose that \hyperlink{(U00)}{{\bf(U00)}}, \hyperlink{(NE)}{{\bf(NE)}}, and \hyperlink{(PH)}{{\bf(PH)}} are satisfied.  Then:  
\begin{itemize}
\item[(i)]  If for some $k\in \N$, $k\geq 2$, $v_1^k(x):= \E_x \tau^k$ is bounded on compact subsets of $U$, then $v_1^k \in C^\infty(U)$ and $L v_1^k =- k \E_x \tau^{k-1}$ on $U$ in the classical sense. 
\item[(ii)]  Suppose that for some $\delta >0$, $w_\delta(x):= \E_x e^{\delta \tau}$ is bounded on compact subsets of $U$.  Then $w_\delta \in C^\infty(U)$ and $Lw_\delta =-\delta w_\delta$.  
\end{itemize}  
\end{Corollary}
\begin{proof}
For (i), set $\varphi(t)= t^k$ and note that $v_1^k(x)$ and $v_2^k(x): = k \E_x \tau^{k-1}$ are both bounded on compact subsets of $U$. By \cref{prop:generalphi},  $L v_1^k =- v_2^k$ on $U$ in the sense of distributions.  By induction and \cref{prop:expexit}, $v_2^k\in C^\infty(U)$ so that $v_1^k \in C^\infty(U)$ and $Lv_1^k=-v_2^k$ in the classical sense.    

For (ii), set $\varphi(t) =e^{\delta t}$.  Then, \cref{prop:generalphi} implies $(L+\delta)w_\delta =0$ in the sense of distributions.  However, $L+\delta$ is hypoelliptic on $U$, so that $w_\delta \in C^\infty(U)$ and $L w_\delta = - \delta w_\delta$ in the classical sense. 
\end{proof}

\section{Green's functions and Bony's Harnack inequality}
\label{sec:green}

In this section, we  explore \emph{Green's functions} in the hypoelliptic setting for an open set $U$.  As a consequence, we  generalize Bony's form of the Harnack inequality~\cite{Bony_69}.          

Let $\beta> 0$ be a constant and suppose $f\in C^\infty(U)\cap B(\overline{U})$.  Often (see~\cite{Bony_69}) one refers to the \emph{Green's operator} $G_\beta$ as the `mapping' $f\mapsto v$, where $v$
 is the `unique' solution of the Poisson problem
\begin{equation}
\label{eqn:poissonB}
\left\{
\begin{aligned}
(L-\beta) v &= -f \qquad && \text{ on } U \,,  \\
u &= 0 \qquad && \text{ on } \partial U.  
\end{aligned}
\right.
\end{equation} 
When $\beta=0$  and $L$ is hypoelliptic on $U$, the uniqueness of solutions of \eqref{eqn:poissonB} heavily depends on the structure of the diffusion $x_t$ driven by $L$ near $\partial U$.  This is also the case when~\eqref{eqn:poissonB} and $\beta > 0$. Here, we employ stochastic methods to  define and deduce properties of $G_\beta f$ without the uniqueness.

Formally, our `best guess' of solution of~\eqref{eqn:poissonB} would be 
\begin{align}
\label{eqn:Green}
v(x)= \E_x \int_0^\tau f(x_s) e^{-\beta s} \, ds   \,,
\end{align}
where $\tau= \inf \{ t>0 \,: \, x_t \notin U \}$.  The expression~\eqref{eqn:Green} is motivated by a formal application of 
 It\^{o}'s formula to the function $e^{-\beta t} v(x_t)$, where $v$ is a presumed classical solution of~\eqref{eqn:poissonB}. Even though it is not clear nor necessarily true that $v$ solves~\eqref{eqn:poissonB}, 
 $v$ in \eqref{eqn:Green} is well-defined 
 for any function $f\in B(\overline{U})$.  Note that $v$ in \eqref{eqn:Green} is well-defined even if $\tau = \infty$, 
 due to the presence of the exponentially decaying factor $e^{-\beta s}$.    

\begin{Remark}
Under further conditions on the stopping time $\tau$, one can define $G_0$ or even sometimes $G_{-\beta}$ for $\beta >0$ small enough.  Here, for simplicity, we keep  $\beta$ 
positive to avoid any further complexities of assumptions on $\tau$.  Additionally, under similar assumptions, one can replace $\beta$ by a spatially dependent function $\beta=\beta(x)\in C(\overline{U})$ to arrive at the Feynman-Kac formula
\begin{align}\label{eqn:fkf}
\E_x \int_0^\tau f(x_s) e^{-\int_0^s \beta(x_v) \, dv } \, ds,  
\end{align}      
provided \eqref{eqn:fkf} makes sense.  
\end{Remark}

\subsection{Definition and properties of $G_\beta$}
Given the preliminary remarks above, we now define $G_\beta$. 
\begin{Definition}
Let \hyperlink{(U00)}{{\bf(U00)}} and \hyperlink{(NE)}{{\bf(NE)}} be satisfied.  For any $\beta >0$ and any $f:\overline{U}\to [0, \infty)$ measurable, define $G_\beta f :U \to [0, +\infty]$ by
\begin{align}
\label{eqn:Gbetapos}
G_\beta f(x) := \E_x \int_0^\tau f(x_s) e^{-\beta s} \, ds. 
\end{align}
If $f:\overline{U}\to \RR$ is measurable and $G_\beta |f|(x)< \infty$ for all $x\in U$, we define 
\begin{align}
\label{eqn:Gbeta}
G_\beta f(x):= G_\beta f_+(x)- G_\beta f_-(x) = \E_x \int_0^\tau f(x_s) e^{-\beta s} \, ds
\end{align}
 where $f_+$, $f_-$ denote the positive and negative parts of $f$, respectively.  
 We call $G_\beta$ the \emph{Green's operator of order} $\beta >0$.  Let $$\mathcal{D}_{\beta}:= \{ f: \overline{U}\to \RR \text{ measurable}\,: \, x\mapsto G_\beta |f|(x)  \text{ is bounded on compacts in } U \}$$ be the \emph{domain} of $G_\beta$.       
\end{Definition}

\begin{Remark}
Observe that the expression \eqref{eqn:Gbeta} is better behaved compared with the formal stochastic representation corresponding to  
 the usual Poisson problem (i.e. with $\beta =0$) due to the presence of the exponentially decaying factor $e^{-\beta s}$.  \end{Remark}

Often it is convenient to  express the operator $G_\beta f$ as a traditional integral operator, meaning that for any 
\emph{reasonable} function $f$ on $\overline{U}$ and any $x\in U$:
\begin{align}
\label{eqn:greenfformal}
G_\beta f(x) = \int_{\overline{U}} f(y) g_\beta(x,y) \, dy
\end{align}
for some kernel $g_\beta$ to be referred to as the \emph{Green's function} associated to $G_\beta$.  However, we must be careful as there are subtleties in defining $g_\beta (x, y)$ for $y\in \partial U$ and for $x=y$ in $U$.

To obtain a workable expression for $g_\beta$, under the assumptions \hyperlink{(U00)}{{\bf(U00)}} and \hyperlink{(NE)}{{\bf(NE)}} we recall the stopped process $\tilde{x}_t=x_{t\wedge \tau_0}$ from 
\eqref{eqn:dfsp} distributed on $\overline{U}$,
 where $\tau_0=\inf\{ t\geq 0 \, : \, x_t \notin U\}$.  We also recall its transition kernel $\tilde{\mathcal{P}}_t(x, \, \cdot \,)$ defined on Borel subsets of $\overline{U}$. Note that, for any $x\in U$ and any $f\in \mathcal{D}_\beta $, the Dominated Convergence theorem implies 
\begin{align*}
\E_x \int_0^{\tau_0} f( \tilde{x}_s)\mathbf{1}_U(\tilde{x}_s) e^{-\beta s}\, ds&=\lim_{\delta \downarrow 0} \E_x \int_0^{(\tau_0- \delta)\vee 0} f(\tilde{x}_s ) \mathbf{1}_U(\tilde{x}_s) e^{-\beta s}\, ds \\
&=\lim_{\delta \downarrow 0} \E_x \int_0^{(\tau_0- \delta)\vee 0} f(\tilde{x}_s ) e^{-\beta s}\, ds\\
&
=\E_x \int_0^{\tau_0} f( \tilde{x}_s) e^{-\beta s}\, ds.  
\end{align*} 
Hence, for any $f\in \mathcal{D}_\beta$ and $x\in U$, we have by Fubini's theorem
\begin{equation}
\label{eqn:fubiningreen}
\begin{aligned}
G_\beta f(x) &= \E_x \int_0^{\tau_0} f(\tilde{x}_s) e^{-\beta s} \, ds = \E_x \int_0^\infty  \mathbf{1}_{U}( \tilde{x}_s) f(\tilde{x}_s) e^{-\beta s} \, ds \\
&= \int_0^\infty \int_{U} f(y) \tilde{\mathcal{P}}_s(x, dy) e^{-\beta s} \, ds 
= \int_{U} f(y) \int_0^\infty \tilde{\mathcal{P}}_s(x, dy) e^{-\beta s} \, ds \\
&=: \int_{U} f(y) g_{\beta, x}(dy) \,,
\end{aligned}
\end{equation}       
where for each $x\in U$, $g_{\beta, x}$ is a finite Borel measure on  $U$ given by
\begin{align}
\label{eqn:gbdef}
g_{\beta ,x }(A) = \int_0^\infty \tilde{\mathcal{P}}_s(x, A\cap U) e^{-\beta s} \, ds.
\end{align}  
In particular, the purported Green's `function' $g_{\beta,x}$, $x\in U$, does not charge the boundary $\partial U$. In order to show that $g_{\beta, x}$ has a 
density with respect to Lebesgue measure on $U$, we prove the following result:

\begin{Theorem}
\label{thm:smoothgreen}
Let $\beta >0$, and suppose that \hyperlink{(U00)}{{\bf(U00)}}, \hyperlink{(NE)}{{\bf(NE)}}, and \hyperlink{(PH)}{{\bf(PH)}} are satisfied.  Then:  
\begin{itemize}
\item[(i)] For every $x\in U$, the finite measure $g_{\beta, x}$ on Borel subsets of $U$ given by~\eqref{eqn:gbdef} is absolutely continuous with respect to Lebesgue measure on $U$ with density $g_\beta(x,y)$.  
\item[(ii)] The mapping $(x,y) \mapsto g_\beta(x,y)\in C^\infty(U\times U\setminus \text{\emph{Diag}})$, where $\text{\emph{Diag}}= \{ (x,x) \, : \, x\in U \}$.  Furthermore, for fixed $x\in U$, $(L_{y}^*-\beta) g_\beta(x,y)= -\delta_x$ in the sense of distributions on $U$ and, for fixed $y\in U$, $(L_x-\beta) g_\beta(x,y)= -\delta_y $ in the sense of distributions on $U$.  
\item[(iii)]  For any $f\in \mathcal{D}_\beta \cap C(U)$, $(L_x-\beta) G_\beta f= -f$ on $U$ in the sense of distributions.   
\item[(iv)]  If $u\in C^2(U;[0, \infty))$ satisfies $L u=0$ on $U$,  then $u \geq \beta G_\beta u  \text{ on } U.$
\end{itemize}
\end{Theorem}

\begin{proof}
First, we show (i). By \cref{thm:kol}, for $s>0$ and $x\in U$, the measure $\tilde{\mathcal{P}}_s(x, \, \cdot \, \cap U)$ is absolutely continuous with respect to Lebesgue measure on $U$ with density $\tilde{p}_t(x,y)$.  Furthermore, $(t,x,y) \mapsto \tilde{p}_t(x,y): (0,\infty) \times U \times U\mapsto \RR$ is $C^\infty$ and $\tilde{p}_t$ satisfies both the forward and backward Kolmogorov equations in the classical sense as in \cref{thm:kol}(i), (ii).  Thus, if $|A \cap U| = 0$, then $\tilde{\mathcal{P}}_s(x, A\cap U)=0$ for any $s >0$ and by~\eqref{eqn:gbdef} one has $g_{\beta, x}(A) = 0$.  
Hence,  by the Radon-Nikodym theorem, for each $x\in U$, $g_{\beta, x}$ has density $y \mapsto g_\beta(x,y)$ which belongs to $L^1(U)$.  This concludes (i).  

Next, we establish parts (ii) and (iii) simultaneously.  Note that for any $f\in C^\infty_0(U)$, we have for any $x\in U$, 
\begin{align}\label{eqn:feq}
G_\beta ((L-\beta) f)(x) = \int_U (L-\beta) f(y) g_{\beta}(x, y) dy.  
\end{align} 
On the other hand, because $f$ vanishes outside of a compact set in $U$, Dynkin's formula \eqref{eqn:Dynkin} applied to $f(x_{t\wedge \tau_0} )e^{-\beta (t\wedge \tau_0)}$ gives
\begin{align}\label{eqn:smrf}
\E_x f(x_{t\wedge \tau_0} )e^{-\beta (t\wedge \tau_0)} &=  f(x) + \E_x \int_0^{t\wedge \tau_0} (L - \beta) f(x_s)e^{-\beta s} ds.
\end{align}
Since $f$ is compactly supported in $U$ and $\beta >0$, after passing $t\to \infty$ we have 
$$
\E_x f(x_{t\wedge \tau_0} )e^{-\beta (t\wedge \tau_0)} \to 0.
$$  
By \eqref{eqn:Gbeta} and \eqref{eqn:smrf}, 
\begin{align}\label{eqn:seq}
G_\beta (L-\beta)f(x) = - f(x)
\end{align}
on $U$ for all $f\in C_0^\infty(U)$.  By combining \eqref{eqn:feq} and \eqref{eqn:seq} we obtain 
\begin{align}
(L^*_y - \beta) g_{\beta}(x, \cdot) = -\delta_x \,\,\, \text{ on } U  
\end{align}
 in the sense of distributions.  Hence, by hypoellipticity of $L^*-\beta$, for each $x\in U$, the mapping $y\mapsto g_\beta(x,y) \in C^\infty(U \setminus \{ x\})$.  

To obtain regularity in the $x$ variable, for any $f\in C(U) \cap B(U) \subset \mathcal{D}_\beta$, consider the sequence of approximations
\begin{align*}
G_\beta^n f(x) := \int_{1/n}^n \int_U f(y) \tilde{p}_s(x,y) \, dy \, e^{-\beta s} \, ds = :\int_U f(y) g_\beta^n(x,y) \, dy  
\end{align*} 
and let $h\in C^\infty_0(U)$.  Employing \cref{thm:kol} and the Dominated Convergence Theorem (in order to interchange the integral and derivatives) it follows that  
\begin{align*}
 \int_U h(x) (L-\beta)_x G_\beta^n f(x) \, dx& = \int_U \int_U h(x)  f(y) \int_{1/n}^n \partial_t(\tilde{p}_t(x,y) e^{-\beta t}) \, dt \, dy \, dx \\
&= \int_U  h(x) [e^{-\beta n} \E_x f(\tilde{x}_n) - e^{-\beta/n} \E_x f(\tilde{x}_{1/n})] \, dx.
  \end{align*}
Passing $n\to \infty$ and using  the Dominated Convergence Theorem, we find that $G_\beta$ satisfies 
\begin{equation}\label{eqn:rtbf}
(L-\beta)_x G_\beta f(x)= -f(x)
\end{equation}
in the sense of distributions on $U$.  More generally, if $f\in \mathcal{D}_\beta \cap C(U)$, then for every $N\in \N$ define $f_N:=f\wedge N\in C(U) \cap B(U)$.  Then $(L-\beta)_x G_\beta f_N = -f_N$ in the sense of distributions on $U$.  Passing $N\to \infty$, it follows by \cref{lem:red} and  definition of $\mathcal{D}_\beta$,  that $(L-\beta)_x G_\beta f =- f$ in the sense of distributions on $U$.  This proves (iii).

 To finish the rest of (ii), let $f, h\in C_0^\infty(U)$ and note that by \eqref{eqn:rtbf}
 \begin{equation}
 \int_{U\times U} [(L-\beta)_x^* h](x)  g_\beta(x, y) f(y) \, dy dx = - \int_U h(x) f(x) \, dx.  
 \end{equation}
 Since $f$ and $h$ are arbitrary, 
 \begin{align}
 \label{eqn:backward}
 (L-\beta)_x g_\beta(x,y) =-\delta_y(x)
 \end{align}
 on $U$ in the sense of distributions, and therefore the hypoellipticity of $L-\beta$ implies that  $x \mapsto g_{\beta}(x, y)$  is $C^{\infty} (U \setminus \{y\})$, as desired.

 Finally, for part (iv), for $n\geq 1$ let $U_n$ be a sequence of bounded open sets with $U_n\uparrow U$ and $\overline{U}_n \subset U$.  Set $\tau_n= \inf\{ t\geq 0 \, : \, x_t \notin U_n \}$.  Then,
  Dynkin's formula \eqref{eqn:Dynkin}, $Lu = 0$, and $u\geq 0$ imply for any $x\in U$, 
 \begin{align*}
 0 \leq \E_x e^{- \beta t \wedge \tau_0 \wedge \tau_n} u(x_{ t \wedge \tau_0 \wedge \tau_n}) =u(x) - \beta \E_x \int_0^{ t \wedge \tau_0 \wedge \tau_n} u(x_s) e^{-\beta s} \,  ds.  
 \end{align*}  
 Rearranging  and passing $n\to \infty$ and $t\to \infty$ gives the result by the Monotone Convergence Theorem.  
\end{proof}

We now apply \cref{thm:smoothgreen} to obtain a version of Harnack's inequality originally due to Bony~\cite{Bony_69}.  
Our assumptions are weaker as we do not assume the uniqueness of the solution of~\eqref{eqn:poissonB}. 
Before we proceed, let 
\begin{align}
\mathcal{H}_U= \{ u\in C^2(U; [0, \infty))  \,: \,   L u =0  \text{ on } U \}.  
\end{align}
Of course, if $L$ is hypoelliptic on $U$, then any distribution $u$ with $Lu=0$ on $U$ must belong to $C^\infty(U)$ and satisfy $Lu=0$ on $U$ in the classical sense.

\begin{Theorem}[Harnack inequality]
\label{thm:harnack}
Suppose that  \hyperlink{(U00)}{{\bf(U00)}}, \hyperlink{(NE)}{{\bf(NE)}}, and \hyperlink{(PH)}{{\bf(PH)}} are satisfied.  Consider any compact set $K\subset U$, any set $D\subset U$ which is dense in $U$ and any multiindex $\alpha$.  Then, there exist points $y_1, y_2, \ldots, y_k\in D$ and a constant $c>0$ such that the following inequality 
\begin{align}
\sup_{x\in K} |D^\alpha u(x)| \leq c[u(y_1)+u(y_2)+\cdots + u(y_k)]
\end{align}    
is satisfied for all $u \in \mathcal{H}_U$.    
\end{Theorem}

\begin{proof}From this point in the paper, the argument is a slight modification of the proof of~\cite[Lemma~7.1]{Bony_69}.  Fix any multi-index $\alpha$.  Let $K\subset U$ be compact and $x_0 \in K$.  
We show that there exist an open neighborhood $V \subset U$ of $x_0$, an element $y\in D$, and a constant $c>0$ such that  
\begin{align}
\label{eqn:firstharnack}
u(y) \geq c \sup_{x\in V} |D^\alpha u(x)|  \,\, \text{ for all } u \in \mathcal{H}_U.  
\end{align}     
The result then follows from \eqref{eqn:firstharnack} using compactness of $K$.  
By \cref{thm:smoothgreen}, fixing $\beta >0$,  we obtain for all $y\in U$ and all $u\in \mathcal{H}_U$
\begin{align*}
u(y) \geq \beta \int_{U} u(z) g_\beta(y, z) \, dz.  
\end{align*} 
We next claim that there exists $y\in D\setminus \{ x_0\}$ such that $g_\beta(y, x_0) >0$.  If not, by \cref{thm:smoothgreen}, $ g_\beta(\cdot, x_0) = 0$ 
on $U \setminus \{ x_0 \}$, a contradiction to $(L-\beta)_y g_\beta(y,x_0) = -\delta_{x_0}(x)$ on $U$ in the sense of distributions (see \eqref{eqn:backward}). 
Thus fix $y\in D\setminus \{ x_0 \}$ so that $g_\beta(y, x_0) >0$ and by continuity, choose disjoint neighborhoods $W$ of $y$ and $X$ of $x_0$ and a constant $c>0$ such that 
\begin{align*}
g_\beta (w, x) \geq c  \,\,\text{ for all } (w, x) \in W \times X.  
\end{align*}
Nonnegativity of $g_\beta$ then  implies for fixed $y\in D$ 
 \begin{align*}
 u(y) \geq \beta c \int_X u(x) \, dx  
 \end{align*}
 for all $u\in \mathcal{H}_U$.  
In order to bound the integral $c\int_X u(x) \, dx= c\| u\|_{L^1(X)}$ from below, we bootstrap \eqref{eqn:Horest} to obtain 
for any $s>0$ and $t <0$ and any open neighborhood $V$ of $x_0$ with $\overline{V}\subseteq X$, the existence of a constant $C_{s,t}$ depending only on $s,t$ and $V, X$ 
such that 
\begin{align*}
\|u\|_{H^s(V)} \leq C_{s,t} \|u \|_{H^t(X)} \,\, \text{ for all } u\in \mathcal{H}_U.    
\end{align*}      
For sufficiently negative $t$ we have $L^1 \hookrightarrow H^t$, and 
  there exists $C>0$ independent of $u$ such that 
\begin{align*}
\sup_{x\in V} |D^\alpha u(x)| \leq C \int_X u(x) \, dx \leq \frac{C}{\beta c} u(y).  
\end{align*}
\end{proof}

Given the previous result, we return to the proof of \cref{prop:expexit}(ii). 

\begin{proof}[Proof of \cref{prop:expexit}(ii)]
 Take a sequence of bounded open sets $U_n$, $n\geq 1$, with $\overline{U_n} \subset U$ and  $U_n \uparrow U$ as $n\to \infty$.  By \cref{prop:expexit}(i), for any $n \geq 1$,
  if $\tau_n:= \tau_{U_n}$, then $w_n(x):= \E_x \tau_n$ is bounded on $\overline{U_n}$ and $C^\infty(U_n)$ with $L w_n=-1$ on $U_n$ in the classical sense.  
 In addition, since $U_n \subset U$, we have $w_n \leq v$.  Let $x\in U_0$ and fix $\delta >0$ such that $\overline{B_{\delta}(x)} \subset U_0$.   
 Fix $m\in \N$ such that $U_m \supset \overline{B_\delta(x)}$.  Then, for each $n \geq m$, the function $z_n = w_n - w_m$ satisfies $Lz_n = 0$ on $B_\delta(x)$ and for any $y \in D$, $0 \leq z_n(y) \leq v(y) < \infty$. 
  By \cref{thm:harnack}, $z_n$ has bounded derivatives of all orders independently of $n$ on $\overline{B_\delta(x)}$. Since $m$ is fixed, $w_n = z_n + w_m$ has 
 has bounded derivatives of all orders independently of $n$ on $\overline{B_\delta(x)}$. Then, $w_n$ converges to a function $\overline{v}$, uniformly on compact subsets of $U_0$.  

The assertion (ii) follows from~\cref{lem:red} once we show that $v = \overline{v}$ on $U_0$. Since $U_n \uparrow U$,
then $\tau_n \uparrow \tau$ and $v = \overline{v}$ follows from the monotone convergence theorem. 
This finishes the proof of (ii).
\end{proof}

\section{General results on the formal stochastic solution}
\label{sec:poisson2}

\subsection{Preliminary remarks}
Throughout this section, we assume that $U\subset \RR^\mathfrak{m}$ is non-empty, open set with non-empty complement $U^c$.  
We return to one of our main goals:  satisfaction of~\eqref{eqn:pp} by the formal stochastic representation, which for the  Poisson equation 
is given by (provided it makes sense) 
\begin{align}
\label{eqn:vstochrev}
v_{\s, 0}(x)= \E_x \int_0^{\tau} f(x_s ) \, ds= \int_0^{\tau_0}f(\tilde{x}_s) \, ds, \,\, x\in U, 
\end{align}
where $\tilde{x}_t = x_{t\wedge \tau_0}$ is the process stopped at the boundary $\partial U$.   
For example if $f \equiv 1$, $v(x):=\E_x \tau=\E_x \tau_0 = \E_x \int_0^{\tau_0} 1 \, ds$, $x\in U$ and as long as the hypotheses of \cref{prop:expexit}(ii) are satisfied, 
$v\in C^\infty(U)$ and $Lv=-1$ on $U$ in the classical sense.    Thus, we expect that under analogous assumptions, $v_{\s, 0} \in C^\infty(U)$ and $Lv_{\s, 0}=-f$ in the classical sense on $U$.  

Additionally, if $g\in C^\infty_0(\X)$ for the formal stochastic solution of the Dirichlet problem
\begin{align}
\label{eqn:ustochrev}
u_{\s, 0}(x) = \E_x g(x_\tau),
\end{align}     
Dynkin's formula applied to $g$ yields for any $x\in U$, $t \geq 0$
\begin{align}
\label{eqn:dirichletpoisson}
u_{\s, 0}(x,t) := \E_x g(x_{\tau \wedge t})&= g(x) + \E_x \int_0^{\tau\wedge t} L g(x_s) \, ds.  \end{align}
Thus by formally passing $t\to \infty$, rearranging and using ~\eqref{eqn:vstochrev}, one has that $u_{\s, 0} - g$ satisfies $L(u - g) = Lg$, and interior smoothness and  
$Lu_{\s,0}=0$ on $U$ in the classical sense follows. 
Then, by an approximation argument and \cref{lem:red}, the same properties follow 
for more general $g$.

\subsection{Interior smoothness for $v_{\s, 0}$ and $u_{\s,0}$}
We begin with the interior smoothness for $v_{\s, 0}$.  Suppose  \hyperlink{(U00)}{{\bf(U00)}} and \hyperlink{(NE)}{{\bf(NE)}} hold and  
for any measurable $f:\overline{U}\to [0, \infty)$, define $G_0 f : U\to [0, \infty]$  by (cf.  \eqref{eqn:Gbetapos})
\begin{align}
G_0 f(x) = \E_x \int_0^{\tau} f(x_s) \, ds \,,  
\end{align}
where $\tau = \tau_0$ is as in \eqref{eqn:dotz}. 
For any $f:\overline{U}\to \RR$ measurable with $G_0 |f|(x)< \infty$ for all $x\in U$, we let  (cf.  \eqref{eqn:Gbeta})
\begin{align}
G_0 f(x):= G_0 f_+(x) - G_0 f_- (x) = \E_x \int_0^\tau f(x_s) \, ds
\end{align}  
and 
\begin{align}
\mathcal{D}_0 (U)=\{ f: \overline{U} \to \RR \text{ measurable} \, : \, x\mapsto G_0|f|(x)  \text{ is bounded on compact sets in } U \}.  
\end{align}

\begin{Theorem}[Interior Smoothness for $v_{\s,0}$]
\label{thm:poissonrev}
Let $f\in C(U)\cap \mathcal{D}_0(U)$ and assume that  \hyperlink{(U00)}{{\bf(U00)}}, \hyperlink{(NE)}{{\bf(NE)}}, and \hyperlink{(PH)}{{\bf(PH)}} are satisfied.  Suppose, moreover 
$\PP_x\{ \tau < \infty\}=1$ for all $x\in U$.  Then    
\begin{align}
L v_{\s,0}=-f \,\, \text{ on } \,\, U 
\end{align}
in the sense of distributions.  Also,  if $f\in C^\infty(U)\cap \mathcal{D}_0(U)$, then $v_{\s,0} \in C^\infty(U)$ and $Lv_{\s, 0}=-f$ on $U$ in the classical sense.        
\end{Theorem}

\begin{proof}
Suppose $f\in C(U)\cap \mathcal{D}_0(U)$.  Let $U_n$ be bounded, open set with $\overline{U}_n \subset U$ and $U_n\uparrow U$ as $n\to \infty$.  For $n,k \in \N$ define
\begin{align*}
v_{n,k}(x):= \int_{1/k}^k \int_{U_n} \tilde{p}_t(x,y) f(y) \, dy \, dt.  
\end{align*} 
By \cref{thm:kol}, $v_{n,k} \in C^\infty(U)$.  Fubini's Theorem then gives
\begin{align*}
L v_{n,k}(x) =  \int_{\frac{1}{k}}^{k} \int_{U_n} L_x \tilde{p}_t(x,y) f(y) \, dy \, dt &=  \int_{\frac{1}{k}}^{k}  \int_{U_n}  \partial_t \tilde{p}_t(x,y) f(y) \, dy \, dt \\
&=\E_x  \mathbf{1}_{U_n} (\tilde{x}_k) f(\tilde{x}_k) - \E_x \mathbf{1}_{U_n} (\tilde{x}_{1/k}) f(\tilde{x}_{1/k}) \\
&\to -\mathbf{1}_{U_n}(x)f(x) 
\end{align*}
as $k\to \infty$ in the sense of distributions on $U$, where we used that $f$ is bounded, continuous on $U_n$ and  for all $x\in U$,  $\E_x  \mathbf{1}_{U_n} (\tilde{x}_k) \leq 1 - \PP \{\tilde{x}_k \in \partial U\} \to 0$ as $k \to \infty$ by 
$\PP_x\{ \tau < \infty \}=1$.  

Since $f$ is bounded on $U_n$, by the Dominated Convergence theorem $v_{n,k} \to v_n$ as $k \to \infty$, where
$$
v_{n}(x):= \int_0^\infty  \int_{U_n}  \tilde{p}_t(x,y) f(y) \, dy \, dt = \E_x \int_0^\tau \mathbf{1}_{U_n}(\tilde{x}_s) f(\tilde{x}_s) \, ds.
$$  
Then, by \cref{lem:red}, $Lv_n=-f$ on $U_n$ in the sense of distributions.  Splitting $f$ into positive and negative parts (using $f \in \mathcal{D}_0(U)$) 
and using the Monotone Convergence Theorem, we obtain that $v_n$, which is locally uniformly bounded on $U$,
converges pointwise to $v_{\s,0}$ as $n \to \infty$. 
By employing \cref{lem:red} again, we obtain the desired result.   
\end{proof}

Given the previous result, we next investigate the interior smoothness of $u_{\s, 0}$.  

\begin{Theorem}[Interior smoothness for $u_{\s,0}$]
\label{thm:dirichletrev}
Assume  \hyperlink{(U00)}{{\bf(U00)}}, \hyperlink{(NE)}{{\bf(NE)}}, and \hyperlink{(PH)}{{\bf(PH)}} are satisfied and 
$\PP_x\{ \tau < \infty\}=1$ for all $x\in U$.  Suppose $g\in C(\X)$ and $u_{\s, 0}(x)=\E_x g(x_\tau)$ is bounded on compact subsets of $U$.  Then, $u_{\s, 0} \in C^\infty(U)$ and $L u_{\s,0}=0$ on $U$ in the classical sense. 
\end{Theorem}

\begin{Remark}
Note that by the Tietze Extension Theorem we can replace the assumption $g\in C(\X)$ with $g$ being continuous only on  a neighborhood of $\partial U$.  
\end{Remark}

\begin{proof}
First let $U_0\subset U$ be bounded, open with $\overline{U_0} \subset U$.  By \cref{prop:expexit}(i), $v_0(x):= \E_x \tau_{U_0}$ is bounded on $U_0$.  Thus, for any
open $V_0$ with $V_0 \supset \overline{U}_0$ and any $g\in C^2(V_0)$, the boundedness of $U_0$ implies $Lg \in C(U_0) \cap \mathcal{D}_0(U_0)$ (with respect to $\tau_{U_0}$).  By \cref{thm:poissonrev}, the function 
\begin{align}\label{eqn:dvn}
v_0(x)= \E_x \int_0^{\tau_{U_0}} L g(x_s) \, ds
\end{align}   
satisfies $Lv_0=- Lg$ on $U_0$ in the sense of distributions.  For $u_0(x):= \E_x g(x_{\tau_{U_0}})$, Dynkin's formula, the Dominated Convergence Theorem, and $\E_x \tau_{U_0} < \infty$ imply
\begin{align*}
u_0(x)= \lim_{t\to \infty} \E_x g(x_{\tau_{U_0}\wedge t})= g(x)+ \lim_{t\to \infty} \E_x \int_0^{\tau_{U_0}\wedge t} Lg(x_s) \, ds=g(x)+\E_x \int_0^{\tau_{U_0}} Lg(x_s) \, ds .
\end{align*}  
Hence, by \eqref{eqn:dvn} and $Lv_0=- Lg$, 
$$
L u_0 = Lg - Lg =0 \,\,\, \text{ on } \,\, U_0
$$ 
in the sense of distributions.  Since $L$ is hypoelliptic, $u_0 \in C^\infty(U_0)$ and $Lu_0 =0$ on $U_0$ in the classical sense. 

Next assume $g$ is merely continuous and supported on a bounded neighborhood $W_0$ of the boundary $\partial U_0$, and without loss of generality assume $g\equiv 0$ on $W_0^c$.  
Fix any $\psi \in C_0^\infty(\RR^\mathfrak{m};[0, \infty))$ with $\int_{\RR^\mathfrak{m}} \psi \, dx =1$, denote $\psi_\epsilon(x)= \epsilon^{-\mathfrak{m}} \psi(\epsilon^{-1} x)$ and set
\begin{align}
g_\epsilon(x)= \int_{\RR^\mathfrak{m}} g(y) \psi_\epsilon(x-y) \, dy \in C_0^\infty(\RR^\mathfrak{m}) \qquad \text{ and } \qquad u_\epsilon(x) = \E_x g_\epsilon (x_{\tau_{U_0}}). 
\end{align} 
Since $g_\epsilon$ is smooth,  by the first part of the proof, $u_\epsilon \in C^\infty(U_0)$ and $L u_\epsilon =0$ on $U_0$ in the classical sense. 
Furthermore, $g$ is bounded, and therefore $u_\epsilon$ is uniformly bounded and converges pointwise to $u_0(x)=\E_x g(x_{\tau_{U_0}})$ on $U_0$ as $\epsilon \to 0$.  It follows by \cref{lem:red} and hypoellipticity of $L$ that $u_0 \in C^\infty(U_0)$ and $Lu_0=0$ on $U_0$ in the classical sense.  

Finally, let $U_n$, $n\in \N$, be a sequence of bounded open sets with $\overline{U_n} \subset U$ and $U_n\uparrow U$.  Suppose that $g\in C(\X; [0, \infty))$ is nonnegative 
and by assumption $u_+(x):= \E_x g(x_\tau)$ is bounded on compact subsets of $U$.  Set $\tau_n = \inf\{ t>0 \, : \, x_t \notin U_n\}$ and note that we already proved that 
$u_{n, N}(x):= \E_x [g(x_{\tau_n}) \wedge N] \in C^\infty(U_n)$ with $L u_{n,N} =0$ on $U_n$ in the classical sense.  After passing $n\to \infty$,  \cref{lem:red},
the Dominated Convergence Theorem, and $\tau_n \to \tau$ imply  that $u_{\infty, N}(x):=  \E_x [g(x_{\tau}) \wedge N]$  satisfies $u_{\infty, N}\in C^\infty(U)$ with $L u_{\infty, N}=0$ on $U$ in the classical sense.  
 Passing  $N\to \infty$ and again applying \cref{lem:red} and Monotone convergence theorem, we find that $u_+\in C^\infty(U)$ with $Lu_+=0$ on $U$ in the classical sense.  
 The result follows after decomposing $g$ into positive and negative parts, $g=g_+- g_{-}$.    
\end{proof}

\subsection{Boundary behavior}
Note that under the hypotheses of~\cref{thm:poissonrev} and~\cref{thm:dirichletrev},  the formal stochastic solution $u_\s$ belongs to~\eqref{eqn:ustoch} is $C^\infty(U)$ and $u_\s$ is also classical solution of $L u_\s=-f$ on $U$, provided $f\in C^\infty(U)$.    
The next natural problem is to determine satisfaction of the boundary condition in~\eqref{eqn:pp}.  We explore it first for the Dirichlet part $u_{\s, 0}$.  

\begin{Theorem}
\label{thm:boundaryd}
Fix $x_* \in \partial U$ and $g\in C(\X)$.  Assume  \hyperlink{(U00)}{{\bf(U00)}}, \hyperlink{(NE)}{{\bf(NE)}}, and $\PP_x\{ \tau < \infty \}=1$ for all $x\in U$. 
 If   \hyperlink{(UIDgx)}{{\bf (UID($g,x_*$))}} and \hyperlink{(CEx)}{{\bf (CE($x_*$))}} are satisfied and $x_*$ is regular as in \cref{def:regular}, then $u_{\s, 0}(x) \to g(x_*)$ as $x\to x_*, x\in U$.  
\end{Theorem}

\begin{Remark}
Note that if $U$ is bounded, by extending $b,\sigma$ to $\RR^\mathfrak{m}$ so that $x_t$ is nonexplosive as in \cref{rem:nonexpl}, conditions  \hyperlink{(UIDgx)}{{\bf (UID($g,x_*$))}} 
and \hyperlink{(CEx)}{{\bf (CE($x_*$))}}  are 
 satisfied for every $x_* \in \partial U$.  Indeed, $g$ is continuous on $\RR^\mathfrak{m}$ and bounded on $\overline{U}$.  Moreover, one can set
  $\X_n=B_n$ so that for all $n$ large enough $\PP_x\{ \tau_{\X_n}  < \tau \} =0$ for all $x\in U$.  
\end{Remark}

\begin{Remark}
As discussed above,  checking that a given point $x_* \in \partial U$ is regular can be  challenging.  See~\cite{Kog_17,CFH_21} for some criteria.  
However, if it makes sense to slightly modify the set $U$, then the regularity of $x_*$ can be easier to verify.  See \cref{thm:modification} for further information. 
\end{Remark}

\begin{proof}[Proof of \cref{thm:boundaryd}]
By  \hyperlink{(UIDgx)}{{\bf (UID($g,x_*$))}}, choose $\delta >0$ so that $\mathcal{G}_{g, \delta}(x_*)$ as in~\eqref{eqn:defGg} is uniformly integrable.  For $|x-x_*|< \delta$ and $n\in \N$, write 
\begin{align*}
u_{\s, 0}(x)- g(x_*) &= \E_x[(g(x_\tau)-g(x_*)) \mathbf{1}\{ x_\tau \notin \overline{\X_n} \}] + \E_x[(g(x_\tau)-g(x_*)) \mathbf{1}\{ x_\tau \in \overline{\X_n} \}]\\
&=:T_1(x, n) +T_2(x,n).
\end{align*}
Fix $\epsilon >0$ and we  claim that there exists $n \in \N$ and $\delta_1\in (0, \delta]$ such that $|T_1(x,n)| < \epsilon/2$ for all $|x-x_*|< \delta_1, x\in U$.  Since the family $\mathcal{G}_{g, \delta}(x_*)$ is uniformly integrable, there is $\epsilon'>0$ be such that 
$$
\epsilon'<\frac{\epsilon}{4(|g(x_*)|+1)}
$$ 
and  whenever $\PP\{ A \}< \epsilon'$ we have $\E |X| \mathbf{1}_{A}< \tfrac{\epsilon}{4}$ for all $X\in \mathcal{G}_{g, \delta}(x_*)$.  However, by \hyperlink{(CEx)}{{\bf (CE($x_*$))}}, there exists $n \in \N, \delta_1\in (0, \delta]$ such that $|x-x_*|< \delta_1$, $x\in U$ implies 
\begin{align*}
\PP_x \{ x_ \tau \notin \overline{\X_n} \} \leq \PP_x \{ \tau_{\X_n} < \tau \} < \epsilon'.  
\end{align*} 
Hence, for such  $n$ and $\delta_1$, uniformly integrability gives $|T_1(x,n) |< \tfrac{\epsilon}{2}$ for all $|x-x_*|< \delta_1, x\in U$, establishing the claim.  

We next claim that for this choice of $n$, there exists $\delta_2 \in (0, \delta_1]$ such that $|T_2(x,n)|< \tfrac{\epsilon}{2}$, thus finishing the proof of the result.  Indeed, since $g$ is continuous on the compact set $\overline{\X_{n}}$, there is a Lipschitz function $h \in C(\overline{\X_{n}})$ such that $\| g- h\|_{L^\infty(\overline{\X_n})}< \frac{\epsilon}{8}$. 
Hence, for any $\epsilon_*>0$ we have 
\begin{align*}
|T_2(x,n)|& \leq \frac{\epsilon}{4} + \E_x [(h(x_\tau) - h(x_*))\mathbf{1} \{ x_\tau \in \overline{\X_n}\} ]\\
& \leq \frac{\epsilon}{4} +  C_1 \E |x_\tau(x) -x_*| \wedge C_2 \mathbf{1}\{ \tau < \epsilon_* \} + 2\|h \|_{L^\infty} \PP_x \{ \tau \geq \epsilon_* \}\\
& \leq  \frac{\epsilon}{4} +  C_1 \E \sup_{s\in [0, \epsilon_*]} |x_s -x_*| \wedge C_2 + 2\|h \|_{L^\infty} \PP_x \{ \tau \geq \epsilon_* \}=: \frac{\epsilon}{4}+ T_2^1+ T_2^2 \,,
\end{align*}
where $C_1, C_2 >0$ are constants depending only on $n$ and $h$.  To estimate $T_2^1$, observe that 
\begin{align*}
T_2^1&=C_1 \E \sup_{s\in [0, \epsilon_*]} |x_s(x)- x_*| \wedge C_2 \\
& \leq C_1 \big(\E \sup_{s\in [0, \epsilon_*]} |x_s(x)-x_s(x_*) |^2 \wedge C_2^2\big)^{1/2} +  C_1 \big(\E \sup_{s\in [0, \epsilon_*]} |x_s(x_*)-x_* |^2 \wedge C_2^2\big)^{1/2}
\end{align*}
where the last inequality follows  by  triangle  and  Cauchy-Schwarz inequalities.  Using \cref{lem:approx}(i) 
and path continuity of $x_t$, there exits $\delta_2 \in (0, \delta_1]$ and $\epsilon_* >0$ small enough so that $T_2^1 \leq \epsilon/8$ for all $x$ with $|x-x_*|< \delta_2$.  For $T_2^2$, 
 \cref{lem:approx}(ii) and the regularity of $x_*$, ensure that $T_2^2< \epsilon/8$ if 
 $\delta_2 > 0$ is sufficiently small.  This finishes the proof of the result.   
\end{proof}

\begin{Remark}\label{rmk:cfcx}
Note that if $U$ us bounded, then  \hyperlink{(CEx)}{{\bf (CE($x_*$))}}  is trivially satisfied by choosing $n$ large such that $U \subset \X_n$.
Here, we verify \hyperlink{(CEx)}{{\bf (CE($x_*$))}} for unbounded $U$ if $x\mapsto \E_x \tau$ is bounded on $B_{\delta_2}(x_*) \cap U$  for some $\delta_2 > 0$
and hypotheses of \cref{lem:nonexpl} are satisfied.

For any $t>0$, we have
\begin{align*}
\PP_x\{ \tau_{\X_n} < \tau \} \leq \PP_x\{ \tau \geq t \}+ \PP_x \{ \tau_{\X_n} \leq t \}.  
\end{align*}
Fix any $\delta_1 > 0$.  Then, Chebychev's inequality gives for all $x\in B_\delta(x_*) \cap U$ that
\begin{align*}
\PP_x \{\tau \geq t \} \leq t^{-1} \E_x \tau \leq C/t  \,.
\end{align*}
Fix $t >0$ large enough independent of $x$, such that $\PP_x\{ \tau \geq t \} < \frac{\delta_1}{2}$. 
As in the proof of \cref{lem:nonexpl}, we have
\begin{align*}
w_n e^{-Ct } \PP_x \{ \tau_{\X_n } \leq t \} \leq w(x) +D \,.
\end{align*}
with $w_n \to \infty$ as $n \to \infty$. 
For already fixed $t > 0$, there is 
 $n\in \N$ large enough so that $\PP_x \{ \tau_{\X_n} \leq t \} < \frac{\delta_1}{2}$ for each $x \in B_{\delta_2}(x_*)$ and ~\eqref{eqn:condexit} is satisfied.
 
\end{Remark}

Next, we derive sufficient conditions  on the Poisson part $v_{\s, 0}$ of $u_\s$ that ensure $v_{\s, 0}(x) \to 0$ as $x\to x_* \in \partial U$, $x\in U$.

\begin{Theorem}
Suppose  \hyperlink{(U00)}{{\bf(U00)}} and \hyperlink{(NE)}{{\bf(NE)}} are satisfied and that $x_* \in \partial U$ is regular as in~\cref{def:regular}. Suppose, furthermore, that conditions 
 \hyperlink{(UIPfx)}{{\bf (UIP($f, x_*$))}} and \hyperlink{(CEx)}{{\bf (CE($x_*$))}} are satisfied for some $f: U\to \RR$ which is measurable and bounded on $\X_n \cap U$ for all $n$. If for every $x\in U$, $\PP_x\{ \tau< \infty \}= 1$, then $v_{\s, 0}(x) \to 0$ as $x\to x_*$, $x\in U$.     
\end{Theorem}

\begin{proof}
Let $\epsilon >0$ and by \hyperlink{(UIPfx)}{{\bf (UIP($f, x_*$))}}  there is $\delta_1 >0$ such that for all $x\in U$ with $|x-x_*|< \delta_1$ the family $\mathcal{G}_{\delta_1}^f(x_*)$ in~\eqref{eqn:defGf} is uniformly integrable.  For simplicity, set $Y_t = \textstyle{\int_0^t f(x_s) \, ds}$.  Then, since $f$ is bounded on $\X_n \cap U$, for any $\delta > 0$
\begin{align*}
|v_{\s,0}(x)| &\leq  |\E_x Y_\tau  \mathbf{1} \{ \tau_{\X_n} < \tau \}|+ |\E_x Y_\tau \mathbf{1} \{ \tau_{\X_n} \geq \tau \}|  \\
& \leq   |\E_x Y_\tau  \mathbf{1} \{ \tau_{\X_n} < \tau \}|+ |\E_x Y_{\tau\wedge \delta} \mathbf{1} \{ \tau_{\X_n} \geq \tau, \tau \leq \delta  \}|+|\E_x Y_{\tau} \mathbf{1} \{ \tau > \delta  \}|  \\
& \leq |\E_x Y_\tau  \mathbf{1} \{ \tau_{\X_n} < \tau \}|+ \delta \| f\|_{L^\infty(\X_n \cap U)}+|\E_x Y_{\tau} \mathbf{1} \{ \tau > \delta  \}|.
\end{align*}
By uniform integrability, there exists $\epsilon' >0$ such that $\PP\{ A\} < \epsilon'$ implies $\E_x Y \mathbf{1}_A< \epsilon/3$ for all $|x - x_*| < \delta_1$.  By  \hyperlink{(CEx)}{{\bf (CE($x_*$))}}, there are
 $n\in \N, \delta_2 \in (0, \delta_1]$ such that $\PP_x \{ \tau_{\X_n} < \tau \} < \epsilon' $ for all $|x-x_*|< \delta_2$, $x\in U$.  For this choice of $n$, let 
 $$
 \delta = \frac{\epsilon}{3(\| f\|_{L^\infty(\X_n \cap U)}+1)}.
 $$  
By making $\delta_2>0$ smaller if necessary, by \cref{lem:approx}(ii) we can ensure that $\PP_x\{ \tau > \delta\} < \epsilon'$ for all $|x-x_*|< \delta_2$, $x\in U$.  The result  
follows  since for $|x-x_*|< \delta_2$, $x\in U$, we have $|v_{\s, 0}(x) | < \epsilon$ and $\epsilon > 0$ is arbitrary.    
\end{proof}

We can now combine the previous results and relate them back to the original problem~\eqref{eqn:pp}.  

\begin{Corollary}
Assume  \hyperlink{(U00)}{{\bf(U00)}}, \hyperlink{(NE)}{{\bf(NE)}}, and $\PP_x \{ \tau< \infty \} =1$ for all $x\in U$.  Let  $g\in C(\X)$ and $f\in C^\infty(U) \cap \mathcal{D}_0(U)$ be such that $u_{\s, 0}$ is bounded on compact subsets of $U$ and $f\in B(\X_n \cap U)$ for all $n$.  If $U$ is boundary regular for $x_t$ and the conditions  \hyperlink{(UIDg)}{{\bf (UID($g$))}}, \hyperlink{(UIPf)}{{\bf (UIP($f$))}}, 
and \hyperlink{(CE)}{{\bf (CE)}} are satisfied, then $u_\s$ is a classical solution of~\eqref{eqn:pp}.  If $U$ is furthermore assumed to be bounded, then $u_\s$ is the unique classical solution of~\eqref{eqn:pp}.     
\end{Corollary}

\begin{proof}
The only assertion we have left to prove is the uniqueness, when $U$ is bounded. 

Let $\{U_n \}$ be a sequence of bounded open subsets of $U$ with 
$\overline{U}_n \subset U_{n+1}$ and $U= \cup_n U_n$.  If $u$ is a classical solution of~\eqref{eqn:pp}, then $u$ is smooth on $U_n$ for each $n$, and by 
Dynkin's formula~\eqref{eqn:Dynkin}
 we have for any $x\in U$ and $n\in \mathbf{N}$ large enough
 \begin{align*}
\E_x u(x_{t\wedge \tau_{U_n}}) = u(x)+ \E_x\int_{0}^{t\wedge \tau_{U_n}} Lu(x_s) \, ds =u(x)-  \E_x\int_{0}^{t\wedge \tau_{U_n}} f(x_s) \, ds, 
\end{align*} 
and therefore
\begin{align*}
u(x)= \E_x u(x_{t\wedge \tau_{U_n}})+ \E_x\int_{0}^{t\wedge \tau_{U_n}} f(x_s) \, ds.
\end{align*} 
Since $u$ is continuous on the compact set $\overline{U}$, $u$ is bounded on $U$. Also, since $g \in C(\X)$ and $f\in C^\infty(U) \cap \mathcal{D}_0(U)$, passing $t\rightarrow \infty$ and then $n\rightarrow \infty$, boundary regularity and the Dominated Convergence Theorem imply that 
\begin{align*}
u(x)= u_\s(x).  
\end{align*}
\end{proof}

\section{The Transience and Recurrence Dichotomy for Degenerate Diffusions}

\label{sec:recurrence}
The goal of this section is to carefully establish the dichotomy between transience and recurrence for degenerate diffusions by adapting the classic cycle constructions of Khasminskii~\cite{Khas_60, Khas_60-rus}, which was carried out in the setting of elliptic diffusions on Euclidean space.  Note that this has been done previously using the language of invariant control sets as in~\cite{Kliem_87}.  However, we found a gap in the arguments in~\cite{Kliem_87} that we could not fix in an obvious way (see \cref{rem:gaps} below).  
Moreover, some regularity claims in~\cite{Kliem_87} could not be verified without calculations analogous to ones 
in the previous sections.  Although it is known that there are alternative, probabilistic paths which circumnavigate these issues (we refer, in particular, to the work of Harris~\cite{Har_56}, the survey paper of Baxendale~\cite{Bax_05} and the work of Meyn-Tweedie~\cite{MT_12}), here we establish the dichotomy using  classical ideas of Khasminskii and regularity properties established above. 
 Note that this approach traces back to Maruyama and Tanaka~\cite{Mar_57} and Watanabe~\cite{Wat_58} 
in the case of a one-dimensional, elliptic diffusion.  We also refer to the works~\cite{AZ_66, GT_15} which we found helpful.

\subsection{Nice diffusions}
In this subsection, we briefly introduce the structural assumptions we make on the diffusion $x_t$ satisfying~\eqref{eqn:sde} in this section. Recall that by condition \hyperlink{(NE)}{{\bf(NE)}}, $x_t \in \X$ for all $t \geq 0$ and any
initial condition $x_0 = x \in \X$.

In order to formulate our results, we need a notion of \emph{irreducibility} of $x_t$ as introduced in the following definition.  
\begin{Definition}\label{def:trav}
Suppose that condition \hyperlink{(NE)}{{\bf(NE)}} is satisfied. We call $x_t$ \emph{irreducible} if
for any $x,y \in \X$ and $\delta >0$ we have $\PP_y\{ \tau_{B_{\delta}(x)^c}< \infty \} >0$.  
\end{Definition}

Note that irreducibility means that, for all $x,y\in \X$, the process started at $y\in \X$ enters an arbitrarily small neighborhood of $x\in \X$ with positive probability.  Thus, 
the process can transition between arbitrarily small neighborhoods of any two points in $\X$.

\begin{Remark}
Comparing terminology, if $x_t$ is irreducible in the sense of  \cref{def:trav}, then $\X$ is the unique invariant control set of $x_t$ as in~\cite{Kliem_87}.  
Certainly, the methods used below can be applied in more general settings, e.g. 
 if there is more than one invariant control set or if the process $x_t$ eventually enters an invariant control set from a larger set to not return to other parts of space.  For our purposes, one irreducible set $\X$ is sufficient.    
\end{Remark}

We are  ready to introduce the assumptions we impose on the diffusion $x_t$ in this section.  

\begin{Definition}\label{def:anice}
We say that the diffusion $x_t$ is \emph{nice} if the following conditions are met:
\begin{itemize}
\item[(i)]  Condition \hyperlink{(NE)}{{\bf(NE)}} is satisfied;
\item[(ii)]  The generator $L$ of $x_t$ is satisfies the parabolic H\"{o}rmander condition on $\X$ as in~\cref{def:parahor}; 
\item[(iii)] $x_t$ is irreducible as in~\cref{def:trav}. 
\end{itemize}
\end{Definition}

One key property of a nice diffusion employed below is that the process leaves bounded sets in $\X$ sufficiently fast.   
\begin{Proposition}
\label{prop:logexit}
Suppose that $x_t$ is a nice diffusion and $U\subset \X$ is nonempty, open and bounded with $\overline{U} \subset \X$.  Then there exists $\delta >0$ such that 
\begin{align}
\sup_{x \in \overline{U}} \E_x e^{\delta \tau_{\overline{U}}}< \infty. 
\end{align}  
\end{Proposition}

\begin{proof}
The proof of this result follows a similar  reasoning used in the proof of \cref{prop:expexit}(i).  Let $z\in \X \setminus \overline{U}$ and fix $\epsilon >0$ such that $B_\epsilon(z) \subset \X \setminus \overline{U}$.  By irreducibility of $x_t$, for all $x\in \overline{U}$ there exists $t=t(x)>0$ and $\alpha=\alpha(x)\in (0,1)$ such that  
\begin{align*}
\PP_x \{ \tau_{B_{\epsilon}(z)^c} > t \} \leq 1- \alpha.  
\end{align*}   
Applying \cref{prop:exitprob}(i) along with the parabolic H\"{o}rmander condition on $\X$ and compactness of $\overline{U}$, there exists $t_*>0$ and $\alpha_* \in (0,1)$ independent of $x$ so that 
\begin{align}
\PP_x \{ \tau_{\overline{U}} > t_* \} \leq \PP_x \{ \tau_{B_{\epsilon}(z)^c} > t_* \} \leq 1- \alpha_* \,\,\, \text{ for all } \,\,\, x\in \overline{U}.   
\end{align}
Following the proof of~\cref{prop:expexit}(i), the Markov property implies that 
\begin{align*}
\PP_x \{ \tau_{\overline{U}} \geq  m t_* \} \leq (1-\alpha_*)^{m-1} 
\end{align*}
for all $m\in \N$ and all $x\in \overline{U}$.  Hence $\tau_{\overline{U}}< \infty$, $\PP_x$-almost surely.  Furthermore, choosing $\delta= \delta(t_*, \alpha_*) >0$ small enough so that 
$e^{\delta t_*} (1-\alpha_*) < 1$, it follows that for any $x\in \overline{U}$:
\begin{align*}
\E_x e^{\delta \tau_{\overline{U}}} = \sum_{m=1}^\infty \E_x  e^{\delta \tau_{\overline{U}}} \mathbf{1} \{ (m-1)t_* \leq \tau_{\overline{U}} < m t_* \} &\leq  \sum_{m=1}^\infty  e^{\delta m t_*}\PP_x\{ \tau_{\overline{U}} \geq (m-1)t_* \}\\
& \leq (1-\alpha_*)^{-2} \sum_{m=1}^\infty ( e^{\delta t_*} (1-\alpha_*))^m < \infty.  
\end{align*}
\end{proof}

\begin{Remark}
Yet another way to rephrase the conclusion of~\cref{prop:logexit} is that the process $x_t$ exits any bounded domain in $\X$ in \emph{logarithmic time} or \emph{exponentially fast} on average.  
\end{Remark}

\subsection{Recurrence and transience for nice diffusions}

We start with the definition of transience and recurrence.  

\begin{Definition}
Suppose the diffusion $x_t$ is nice.  
We say that a point $x\in \X$ is \emph{recurrent} if for any $\delta >0$ and any $y\in \X \setminus  B_{\delta}(x)$
\begin{align}
\PP_y \{ \tau_{B_\delta(x)^c} < \infty\}=1.  
\end{align}
Otherwise, we say $x\in \X$ is \emph{transient}.      
\end{Definition}

Our next goal is to prove that points in $\X$ are either all recurrent, in which case we call $x_t$ \emph{recurrent}, or all transient, in which case we call $x_t$ \emph{transient}.   Thus, the dichotomy between transience and recurrence is established in the following proposition. Afterwards, we establish further properties of transience and recurrence.

\begin{Proposition}\label{prop:dicho}
Assume $x_t$ is  nice. 
If $x\in \X$ is recurrent, then all points in $\X$ are recurrent.  Consequently, either all points in $\X$ are recurrent or all points in $\X$ are transient. 
\end{Proposition}

\begin{proof}
Suppose $x\in \X$ is recurrent and let $y \in \X$ with $x\neq y$.  We show that $y$ is also recurrent.  Suppose that $\delta > 0$ is any positive real number such that $x\notin B_{2\delta}(y)$ and $\overline{B_{2\delta}(y)} \subset \X$.  
By irreducibility in  \cref{def:trav} and path continuity (if $x_t \in B_\delta (y)$ for some t, then the inclusion holds for rational $t$), 
 there exists $t_*>0$ such that 
\begin{align}
2\alpha := P_{t_*}(x, B_\delta(y)) >0.  
\end{align}   
Since $L$ satisfies the parabolic H\"{o}rmander condition on $\X$, $w\mapsto P_{t_*}(w, B_\delta(y))$ is continuous at $x$ by  \cref{rem:extn} and \cref{thm:kol}.  
In particular, there exists $\epsilon >0$ such that $B_\epsilon(x) \cap B_{2\delta}(y) =\emptyset$, $\overline{B_\epsilon(x)}\subset \X$ and $P_{t_*}(w, B_\delta(y)) \geq \alpha$ for all $w\in \overline{B_\epsilon(x)}$.  Define stopping times $\sigma_j, j=0,1,\ldots$, inductively as follows:
\begin{align*}
\sigma_1&= \inf\{ t\geq 0 \, : \, x_t \in B_\epsilon(x)\}, \\
\sigma_{2}&= \sigma_1 + t_*\\
\sigma_3&= \inf\{ t\geq \sigma_2 \, : \, x_t \in B_\epsilon(x) \} \\
\vdots & \qquad \vdots \\
\sigma_{2k}&= \sigma_{2k-1}+t_* \\ 
\sigma_{2k+1}& = \inf\{ t\geq \sigma_{2k} \, : \, x_t \in B_\epsilon(x) \}.
\end{align*}
for $k\geq 2$. 
Then for all $j\geq 1$, the stopping time $\sigma_j$ is almost surely finite since $x_t$ is nice and $x$ is recurrent.  Next, by the strong Markov property we have for each $j \geq 0$:
\begin{align*}
\PP_w \{ \tau_{B_\delta(y)^c} >\sigma_{2j+2} \} &= \E_w \E_w [\mathbf{1} \{ \tau_{B_\delta(y)^c}> \sigma_{2j+2} \} | \mathcal{F}_{\sigma_{2j+1}}]\\
&= \E_w \mathbf{1} \{ \tau_{B_\delta(y)^c}  > \sigma_{2j+1} \} \E_{x_{\sigma_{2j+1}}} \mathbf{1} \{ \tau_{B_\delta(y)^c} > t_* \}  \\
& \leq (1-\alpha) \PP_w\{ \tau_{B_\delta(y)^c}  > \sigma_{2j} \}.\end{align*}
Thus, by induction, $\PP_w \{ \tau_{B_\delta(y)^c} >\sigma_{2j} \} \leq (1-\alpha)^j$ for all $j\geq 0$. Finally, the Borel-Cantelli Lemma  implies $\PP_w \{ \tau_{B_{\delta}(y)^c} < \infty \} =1$ for all $w\in \X$, and therefore $y$ is recurrent.         
\end{proof}

We also have the following corollary of \cref{prop:dicho}.  

\begin{Corollary}
\label{cor:infof}
Suppose $x_t$ is nice.  If $x_t$ is recurrent, then for any $x, y\in \X$ and any open set $U_y\subset \X$ containing $y$:
\begin{align*}
\PP_x \{ \omega \, : \, \exists \, s_j(\omega) \in (0, \infty) \uparrow \infty \text{ for which } x_{s_j}\in U_y \text{ for all } j\} = 1.   
\end{align*}
\end{Corollary}

\begin{Remark}
\cref{cor:infof} states that if $x_t$ is recurrent, then for all $x\in \X$, almost surely the process started from $x$ visits  infinitely often any neighborhood of any $y \in \X$.   
\end{Remark}

\begin{proof}[Proof of \cref{cor:infof}]
Let $x, y\in \X$ and $U_y\subset \X$ be open with $y\in U_y$.  Fix $z\in \X$ with $z\neq x$ and $z\neq y$.  Choose $\delta >0$ such that $B_\delta(y)\subset U_y$, $B_\delta(z)\cap B_\delta(x)=\emptyset$, $B_\delta(z) \cap B_\delta(y) = \emptyset$ and $B_\delta(z) \subset \X$.  
Define stopping times $\sigma_i, i=0,1, \ldots$ inductively by
\begin{align*}
\sigma_0&=0;\\
\sigma_1&= \inf\{ t\geq \sigma_0 \,  : \, x_t\in B_\delta(y) \}\\
\sigma_2&= \inf\{ t\geq \sigma_1 +1 \, : \, x_t \in B_\delta(z) \} \\
\vdots & \qquad \vdots\\
\sigma_{2k+1}&= \inf\{ t\geq \sigma_{2k} \, : \, x_t \in B_\delta(y) \} \\
\sigma_{2k+2}&= \inf\{ t \geq \sigma_{2k+1}+1 \, : \, x_t \in B_\delta(z) \}.   
\end{align*}      
By recurrence, $\sigma_k$, $k \geq 0$ is almost surely finite.  The result follows by setting $s_j=\sigma_{2j+1}$.   
\end{proof}

The next proposition further explores implications of transience for a nice diffusion $x_t$.  

\begin{Proposition}\label{prop:trfin}
Suppose $x_t$ is nice and transient.  For any $x,y\in \X$ there exists $\delta >0$ small enough such that $\overline{B_\delta(y)} \subset \X$ and 
\begin{align}\label{eqn:dfws}
\PP_x\{\omega \,: \, \exists \, t_0(\omega) \in [0, \infty) \text{ such that }x_t(\omega) \notin B_\delta(y) \,\, \forall t\geq t_0(\omega) \} =1. 
\end{align}   
\end{Proposition}

\begin{proof}
Since $x_t$ is transient, there exist $\delta_1 >0$, $y \in \X$, and $z\in \X \setminus B_{2\delta_1}(y)$ such that 
\begin{align}
\label{eqn:comeback}
\PP_z\{ \tau_{B_{\delta_1}(y)^c}= \infty \}=: 2\alpha >0.  
\end{align}
By \cref{prop:exitprob}, there exists $\epsilon >0$ such that $\overline{B_\epsilon(z)}\subset \X$ and $$\PP_w\{ \tau_{B_{\delta_1}(y)^c}=\infty \}\geq \alpha$$ for all $w\in B_\epsilon(z).$ 
In order to obtain a contradiction, 
for every $\delta >0$ define 
\begin{equation}
E_\delta(x) :=  \{\omega \, : \, \exists \, s_j(\omega) \in (0, \infty) \uparrow \infty \text{ such that } x_{s_j}(\omega) \in B_{\delta}(y), x_0 = x \}
\end{equation}
and assume there is $x \in \X$ such that  for each $\delta > 0$
\begin{align}
\label{eqn:chyp}
\PP\{E_\delta(x) \}>0.  
\end{align}
Note that for each $t \geq 0$, one has $E_\delta(x) = E_\delta(x_t(x))$, where $x_t(x)$ is the process with $x_0(x)=x$. 
Next, we claim that, on the set $E_\delta(x)$ for all $\delta >0$ small enough, the process $x_t$ almost surely enters $\overline{B_{\epsilon/2}(z)} $.  

To prove the claim, first observe that since $\X$ is irreducible, 
$$
\PP_y \{ \tau_{B_{\epsilon/2}(z)^c}< \infty \}>0.
$$ 
Thus, by path continuity (cf. the proof of \cref{prop:dicho}) there exists $t_*>0$ such that $P_{t_*}(y, B_{\epsilon/2}(z)) =:2a >0$. 
Since $w\mapsto P_{t_*}(w, B_{\epsilon/2}(z))$ is continuous, there exists $\delta \in (0, \delta_1)$ such that  $\overline{B_\delta(y)}\subset \X$ and $P_{t_*}(w, B_{\epsilon/2}(z)) \geq a$ 
for all $w\in \overline{B_\delta(y)}$.  Fix such $\delta  \in (0, \delta_1)$ and inductively define stopping times $\zeta_i$, $i=1,2, \ldots$, by 
\begin{align*}
\zeta_1&= \inf\{ t\geq 0 \, : \, x_t \in B_\delta(y) \}\,,\\
\zeta_2&= \zeta_1 + t_* \,,\\
\vdots & \qquad \vdots\\
\zeta_{2k+1}&= \inf\{ t\geq \zeta_{2k} \, : \, x_t \in B_\delta(y)\} \,,\\
\zeta_{2k+2}&= \zeta_{2k+1} + t_*.  
\end{align*}  
By the definition of $E_\delta(x)$, 
the stopping times $\zeta_i$ are almost surely finite on $E_\delta(x)$.  Then, the strong Markov property (cf. the proof of \cref{prop:dicho}) yields
\begin{align}\label{eqn:asff}
\PP_x\{ \tau_{B_{\epsilon/2}(z)^c} >\zeta_{2j} \, |\,  E_\delta(x) \} \leq (1-a)^j   
\end{align}    
for any $j$.  Thus the  
Borel-Cantelli lemma implies that $\PP_x\{ \tau_{B_{\epsilon/2}(z)^c} < \infty \,|\, E_\delta(x) \} = 1$, establishing the claim.

 Next, define stopping times $ \sigma_i'$, 
$i = 1, 2, \ldots,$ by
\begin{align*}
\sigma_1'&= \inf\{ t\geq 0 \, : \, x_t \in   B_\delta(y) \} \,, \\
\sigma_2'&= \inf \{ t\geq \sigma_1' \, : \, x_t \in B_{\epsilon/2}(z) \} \,, \\
\vdots & \qquad \vdots \\
\sigma_{2k+1}' &= \inf \{ t\geq \sigma_{2k}' \, : \, x_t \in   B_\delta(y)\} \,, \\
\sigma_{2k+2}' &= \inf \{ t\geq \sigma_{2k+1}' \, : \, x_t \in B_{\epsilon/2}(z)\}.   
\end{align*}
By a similar argument to the one used above, it also follows that on the event $E_\delta(x)$, $\sigma_j'$ is finite almost surely for all $j\geq 1$.  
Observe that $\PP (\sigma_j' < \infty) \geq \PP (E_\delta(x)) \geq c$, where $c$ is independent of $j$. 
However, by the strong Markov property and iteration, if $j\geq 2$ we have
\begin{align*}
\PP_x \{ \sigma_{2j}' < \infty\} &= \E_x \E_x[ \mathbf{1} \{ \sigma_{2j}' < \infty \} \, | \, \mathcal{F}_{\sigma_{2j-2}'} ]\\
&= \E_x \mathbf{1} \{ \sigma_{2j-2}' < \infty \} \PP_{x_{\sigma_{2j-2}'}} \{ \sigma_2' < \infty \} \leq (1-\alpha) \PP_x\{\sigma_{2j-2}'  < \infty \}
  \leq (1-\alpha)^{j-1}.
\end{align*}   
Thus Borel-Cantelli implies that $\sigma_{2j}' < \infty$ for only finitely many $j$, a contradiction.    
\end{proof}

As an immediate consequence, we have the following corollary.

\begin{Corollary}
\label{cor:trans1}
If $x_t$ is transient, then for any compact set $K\subset \X$ and $x\in \X$
\begin{align*}
\PP_x\{ \omega \, : \, \exists \, t_0^K(\omega) \in [0, \infty) \text{ such that } x_t(\omega) \notin K \,\, \forall t\geq t_0^K(\omega) \} =1 
\end{align*} 
and 
\begin{align*}
\lim_{t \to \infty} \PP_x\{x_t \in K\} = 0 \,.
\end{align*}  
\end{Corollary}

\begin{proof}
Fix $x\in \X$ and 
for any $y \in K$ fix $\delta_y > 0$ and $t_0(\omega)= t_{0,y}(\omega)$ such that the conclusion of~\cref{prop:trfin} holds true. From the open cover $\{B_{\delta(y)}(y)\}_{y \in K}$ of $K$ choose a finite subcover and 
define $t_0^K$ to be 
the maximum of $t_{0,y}$ in this finite subcover.  
For the second conclusion, we note that 
$$
 \PP_x \{ x_t \in K \} \leq \PP_x\{ t_0^K(\omega) > t \}\rightarrow 0
$$ as $t\rightarrow \infty$, where $t_0^K$ is as in the first assertion.    
\end{proof}

\subsection{Invariant measures}

A central interest in the theory of stochastic differential equations is the large-time behavior of the process $x_t$.  In particular, we are interested in the relationship between 
recurrence, transience  and the existence of \emph{invariant measures}.  Such measures are the random analogues of equilibrium points of deterministic ordinary differential equations.

To introduce invariant measures, throughout this section we again assume $x_t$ is a nice diffusion.  In particular, $x_t$ is non-explosive process on $\X$, 
and consequently the process $x_t$ is Markov with Markov semigroup $(\mathcal{P}_t)_{t\geq 0}$.  Recalling that $\mathcal{B}$ denotes the Borel sigma algebra 
of subsets of $\X$, we call a positive, $\mathcal{B}$-measure $\mu$ an \emph{invariant measure}, 
if $\mu \mathcal{P}_t= \mu$ for all $t\geq 0$, where $\mu \mathcal{P}_t$ was defined in~\eqref{eqn:semim}.  An invariant measure $\mu$ with $\mu(\X)=1$ is called an \emph{invariant probability measure}.      

\begin{Remark}
Observe that if $\mu$ is an invariant probability measure, the equality $\mu \mathcal{P}_t=\mu$ for all $t\geq 0$ means that the process $x_t$ with 
initial distribution $\mu$ has the distribution $\mu$ for all times $t\geq 0$.  In other words, the statistics remain invariant under the dynamics.    
\end{Remark}

\begin{Remark}
It is common in the literature to implicitly assume that 
 an invariant measure is an invariant probability measure.  However, below we need to distinguish between invariant measures which are probability distributions and those which are not.  
\end{Remark}

We first show that a nice diffusion $x_t$ which is transient cannot have an invariant probability measure.  

\begin{Corollary}
\label{cor:noinv}
If $x_t$ is  nice diffusion which is transient, then $x_t$ cannot have an invariant probability measure.     
\end{Corollary}

\begin{proof}
Suppose to the contrary that there exists an invariant probability measure $\mu$.   Then, there exists a compact set $K\subset \RR^\mathfrak{m}$ with $K\subset \X$ and $\mu(K)>0$.  Since $\mu$ is invariant, 
\begin{align*}
0<\mu(K) = \int_{\X} \mu(dx) P_t(x, K) \qquad \text{ for all } t\geq 0.   
\end{align*}   
Using the Bounded Convergence Theorem and  \cref{cor:trans1}, it follows that $$\int_{\X} \mu(dx) P_t(x, K) \rightarrow 0$$ as $t\rightarrow \infty$.  
Hence, $\mu(K)=0$, a contradiction.  
\end{proof}

On the other hand, when $x_t$ is recurrent, one can always construct a $\sigma$-finite invariant measure using an embedded Markov chain via cycles.  
We provide details below,  but first we prove an auxiliary result that allows us to further categorize recurrence.  

\begin{Proposition}
\label{prop:posrec}
Let $x_t$ be a nice diffusion.  If $U\subset \X$ is a bounded, non-empty, open set with $\overline{U} \subset \X$ and $\E_x \tau_{U^c}<\infty$ for all $x\in \X$, then:
\begin{itemize}
\item[(i)]  $x_t$ is recurrent;
\item[(ii)] $\E_x \tau_{V^c}< \infty$ for all $x\in \X$ and 
any non-empty, open set $V\subset \X$ with $\overline{V}\subset \X$   
\end{itemize}

\end{Proposition}

\begin{Remark}
The argument is similar to previous cycle constructions, except that one has to control expected values rather than probabilities.  
\end{Remark}

\begin{proof}[Proof of  \cref{prop:posrec}]
We prove both conclusions simultaneously.  By making $V$ smaller, we may suppose without loss of generality that $V\subset \overline{V} \subset \X$ is an open ball.  
Fixing $x\in \X$, our goal is to show that $\E_x \tau_{\overline{V}^c}< \infty$.  

Let $\X_k \subset \X$ be an increasing sequence of bounded open sets with $\overline{\X_k} \subset \X_{k+1}$ and 
$\cup_k \X_k = \X$.  By compactness, there is $k_0 \in \N$ large enough so that $\X_k \supset \overline{U} \cup \overline{V}\cup\{x\}$ for each $k \geq k_0$.  
Since $x_t$ is nice, for each fixed $y \in \overline{U}$ there is 
$\ell = \ell(y) \geq k_0 + 1$ such that $u_\ell(y) := \PP_y\{ \tau_{\overline{V}^c}< \tau_{\X_{\ell}}\}  > 0$. 
Note that for each $k \geq k_0$ and $z \in Q := \X_k \setminus \overline{V}$ one has   
$\PP_z \{\tau_Q < \infty \} = 1$. Indeed, by \cref{prop:logexit},  $x_t$ almost surely leaves $\X_k$ in finite time.  
Setting $g=1$ on $\partial V$ and $g=0$ on $\partial \X_{\ell}$, we have 
$u_\ell(y) = \E_y g(x_{\tau_Q})$ and by  \cref{thm:dirichletrev},  $u_\ell(w)>0$ for all $w\in B_\epsilon(y)$ and some $\epsilon>0$.  Since $\ell \mapsto u_\ell(x)$ is non-decreasing for any 
$x \in \X_{k_0 + 1}$, the compactness implies the existence of 
 $k^* \geq k_0 + 1$ and $q>0$ such that 
\begin{align*}
 \inf_{y\in \overline{U}} \PP_y\{ \tau_{\overline{V}^c}< \tau_{\X_{k^*}}\} \geq q \,.
\end{align*}
Next, define stopping times $\eta_i$, $i=-2, 0, 1, 2\ldots $, as follows: $\eta_{-2}=0$ and
\begin{align*}
\eta_0&= \inf\{ t\geq 0   \, : \, x_t \in \overline{U} \}, \\
\eta_1&=\inf\{ t\geq \eta_0 \, : \, x_t\in \partial \X_{k^*} \},\\
\eta_2&= \inf\{ t\geq \eta_1   \, : \, x_t \in \overline{U} \},\\
\vdots &\qquad \vdots \\
\eta_{2i+1} &=  \inf\{ t\geq \eta_{2i}\, : \, x_t \in \partial \X_{k^*} \}, \\
\eta_{2i+2} &=\inf\{ t\geq \eta_{2i+1} \, : \, x_t \in \overline{U} \}.    
\end{align*}     
As in the proof of  \cref{prop:dicho}, using \cref{prop:logexit} one can show that each $\eta_j$ is almost surely finite and $\PP_x \{\tau_{\overline{V}^c} \geq \eta_{2N}\} \leq (1 - q)^{N-1}$. Then, by Borel-Cantelli,
$\tau_{\overline{V}^c}< \eta_{2N}$, for $\PP_x$-almost surely for some bounded random index $N$.  The conclusion in (i) follows as $x$ was arbitrary.  

To prove (ii), note that 
 \begin{align}
\label{eqn:pp1}
\E_x \tau_{\overline{V}^c} = \sum_{j=0}^\infty \E_x \tau_{\overline{V}^c} \mathbf{1}_{[\eta_{2(j-1)}, \eta_{2j})} (\tau_{\overline{V}^c})\leq    \sum_{j=0}^\infty \E_x \eta_{2j} \mathbf{1}_{[\eta_{2(j-1)}, \eta_{2j})} (\tau_{\overline{V}^c})\end{align}  
and if $E_\ell := \{ \omega \, : \, x_t \in \overline{V}\,\,  \textrm{ for some } \,\, t\in [\eta_{2(\ell-1)}, \eta_{2\ell}) \},$ then 
\begin{align}
\label{eqn:pp2}
\{\omega \, : \, \tau_{\overline{V}^c} \in [\eta_{2(j-1)}, \eta_{2j}) \} = \bigg[\bigcap_{\ell =1}^{j-1} E_\ell^c \bigg] \cap E_j.    
\end{align}  
Define $\alpha = \sup_{y\in \overline{U}} \E_y \eta_2$.  We next claim that $\alpha < \infty$.  
Indeed, for any $y\in \overline{U}$, we have $\eta_0 = 0$.  Moreover, by the strong Markov property,   
\begin{align*}
\E_y \eta_2 &=\E_y \eta_1+ \E_y \E_y [\eta_2 -\eta_1 | \mathcal{F}_{\eta_1}] = \E_y \eta_1 +\E_y \E_{x_{\eta_1}} \tau_{\overline{U}^c}  \leq \E_y \eta_1+ \sup_{z\in \partial \X_{k^*} }\E_z \tau_{\overline{U}^c}.      
\end{align*}
Applying \cref{prop:logexit} with $U = \X_{k^*}$, $\sup_{y\in \overline{U}} \eta_1 < \infty.$  Also, by assumption $\E_z \tau_{\overline{U}^c}< \infty$ for all $z\in \X$.  Applying  \cref{prop:expexit}(ii) (with $U$ 
replaced by $ \X \setminus \overline{U}$) and $x_t \in \X$ for each $t \geq 0$,  
yield that $z\mapsto \E_z \tau_{\overline{U}^c} $ is smooth on $\X \setminus \overline{U}$. In particular $z\mapsto \E_z \tau_{\overline{U}^c} $ is bounded on the compact set $\partial \X_{k^*}$, 
and the claim follows.      

In addition,  for our fixed $x \in \X_{k_*}$, one has $\E_x \eta_0 < \infty$, and consequently by \cref{prop:logexit} and the arguments above
\begin{align}\label{eqn:brca}
\beta := \E_x \eta_2 &=\E_x \eta_0+ \E_x [\eta_1 - \eta_0] + \E_x \E_x [\eta_2 -\eta_1 | \mathcal{F}_{\eta_1}] \\
&\leq \E_x \eta_0+ \sup_{y\in \overline{U}} \E_y \tau_{\X_{k_*}} + \sup_{y\in \partial \X_{k_*}} \E_y \tau_{\overline{U}^c} < \infty.  
\end{align}

We next claim  that  for $j\geq 2$ the following estimate holds:
\begin{align}\label{eqn:siin}
\E_x \eta_{2j} \mathbf{1}_{[\eta_{2(j-1)},\eta_{2j})  }(\tau_{\overline{V}^c}) \leq (1-q)^{j-2} \beta + (j-2) (1-q)^{j-2} \alpha + (1-q)^{j-1}\alpha.  
\end{align}       
We proceed by induction. If  $j=2$, then
\begin{align*}
\E_x \eta_4 \mathbf{1}_{[\eta_2, \eta_4)}(\tau_{\overline{V}^c}) &= \E_x \E_x [\eta_4 \mathbf{1}_{E_1^c \cap E_2 }| \mathcal{F}_{\eta_2} ] \leq \E_x \eta_2  + \E_x \mathbf{1}_{E_1^c} \E_{x_{\eta_2}} \eta_2 \leq \beta + (1-q) \alpha,   
\end{align*}
which establishes the base case of the induction argument.  Suppose \eqref{eqn:siin} holds for some $j\geq 2$.  Then,
\begin{align*}
\E_x \eta_{2(j+1)} \mathbf{1}_{[\eta_{2j},\eta_{2(j+1)}) }(\tau_{\overline{V}^c}) &= \E_x[ \E_x \eta_{2(j+1)} \mathbf{1}_{\bigcap_{\ell=1}^j E_\ell^c \cap E_{j+1}}| \mathcal{F}_{\eta_2}] \\
&= \E_x \eta_2 \PP_x\bigg(\bigcap_{\ell=1}^j E_\ell^c \cap E_{j+1} | \mathcal{F}_{\eta_2}\bigg)\\
& \qquad +\E_x[ \E_x (\eta_{2(j+1)}-\eta_2) \mathbf{1}_{\bigcap_{\ell=1}^j E_\ell^c \cap E_{j+1}}| \mathcal{F}_{\eta_2}]\\
& \leq \beta (1-q)^{j-1} + (1-q) \sup_{y\in \overline{U} } \E_y \eta_{2j} \mathbf{1}_{\bigcap_{\ell=1}^{j-1} E_i^c \cap E_{j}}\\
& \leq \beta (1-q)^{j-1}  + (j-1) (1-q)^{j-1} \alpha+ (1-q)^j \alpha. 
\end{align*}
This finishes the proof of \eqref{eqn:siin}.  Then, (ii) follows from \eqref{eqn:pp1}. 
\end{proof}

The previous result gives rise to the following definition.  
\begin{Definition}
Suppose $x_t$ is a nice diffusion and assume there exists a bounded, non-empty, open set $U\subset \X$ with $\overline{U} \subset \X$ such that 
$\E_y \tau_{U^c} < \infty$ for all $y \in \X$.  Then, we call $x_t$ \emph{positive recurrent}.  Otherwise, if $x_t$ is recurrent but not positive recurrent, we call $x_t$ \emph{null recurrent}.    
\end{Definition}

Note that by \cref{prop:posrec}, positive recurrence immediately implies recurrence. 
Next, we show that if $x_t$ is positive recurrent, it has an  invariant probability measure. 
The following result can be found in a number of references, see for example~\cite{Kliem_87, RB_06}.  We provide most of the details for completeness.

\begin{Proposition}
\label{prop:existinv}
Suppose $x_t$ is nice.  If $x_t$ is recurrent, then there exists a $\sigma$-finite invariant measure.  If $x_t$ is positive recurrent, then there exists an invariant probability measure.    
\end{Proposition}

\begin{proof}
Fix any open balls $U,V$ with $\overline{U}\subseteq V\subseteq \overline{V} \subseteq \X$. Denote 
$\Gamma_1= \partial V$ and $\Gamma_2 = \partial U$ and introduce stopping times $\sigma_i$, $i=0,1,\ldots$ defined by  
\begin{align*}
\sigma_0&=0 \,,\\
\sigma_1&= \inf\{ t\geq 0 \, : \, x_t \in \Gamma_1\} \,, \\
\sigma_2&= \inf\{ t\geq \sigma_1 \, : \, x_t \in \Gamma_2 \} \,, \\
\vdots & \qquad \vdots \\
\sigma_{2k+1} &= \inf \{ t\geq \sigma_{2k} \, : \, x_t \in \Gamma_1 \} \,, \\
\sigma_{2k+2}&= \inf\{ t \geq \sigma_{2k+1} \, : \, x_t \in \Gamma_2 \} \,. 
\end{align*}    
Since $x_t$ is a nice recurrent diffusion, each of the stopping times $\sigma_i$ is almost surely finite.
We can thus define a discrete-time Markov chain $\{X_n\}_{n \geq 0}$ on $\Gamma_2$ by $X_0=x\in \Gamma_2$ and $X_n = x_{\sigma_{2n}}$, $n\geq 1$.  
Because $\{X_n\}$ has compact state space $\Gamma_2$, it possesses an invariant probability measure $\nu$ supported on $\Gamma_2$ by the Krylov-Bogolyubov Theorem.\footnote{Here, if $Q(x,dy)$ 
denotes the one-step transition kernel of $X_n$, a probability measure is a invariant if $\mu Q= \mu$.}   Then, $\nu$ induces a measure $\mu$ on $\B$ by 
\begin{align}
\label{eqn:unnorm}
\mu(B) = \int_{\Gamma_2} \nu(dx) \E_x \int_0^{\sigma_2} \mathbf{1}_B(x_s) \, ds =: \int_{\Gamma_2} \nu(dx) \E_x \sigma^B \,,  
\end{align}
where $\sigma^B$ is the total time spent by the process $x_t$ in $B$ during one ``cycle" $[0, \sigma_2]$.  
The calculations in~\cite[staring on p. 31]{RB_06}, yield that $\mu$ is an invariant measure for $x_t$.  

To prove that $\mu$ is $\sigma$-finite, we show that for any compact set $K \subset \X$
\begin{align*}
\sup_{x\in \Gamma_2 } \E_x \sigma^{K} < \infty   \,,
\end{align*}
where 
\begin{equation}
\sigma^{K} = \int_0^{\sigma_2} \mathbf{1}_{K} (x_s) \, ds = \int_0^{\sigma_1} \mathbf{1}_{K} (x_s) \, ds + \int_{\sigma_1}^{\sigma_2} \mathbf{1}_{K} (x_s) \, ds =: 
\sigma^{K}_1 + \sigma^{K}_2 \,.
\end{equation}
Without loss of generality, by making $K$ larger, we can assume $\overline{U} \cup \overline{V} \subset K$.  

First observe that
\begin{align*}
\sup_{x\in \Gamma_2} \E_x \sigma^K \leq \sup_{x\in \overline{U}} \E_x \sigma_1 + \sup_{y\in \Gamma_1} \E_y \int_0^{\tau_{\overline{U}^c}} \mathbf{1}\{ x_s \in K \} \, ds.   
\end{align*} 
Using \cref{prop:logexit}, $\sup_{x\in \overline{U}} \E_x \sigma_1< \infty$.  For the other term, note that since $\Gamma_1 \subset K$ we have 
\begin{align*}
\sup_{y\in \Gamma_1} \E_y \int_0^{\tau_{\overline{U}^c}} \mathbf{1}\{ x_s \in K \} \, ds \leq \sup_{y \in K }  \E_y \int_0^{\tau_{\overline{U}^c}} \mathbf{1}\{ x_s \in K \} \, ds. \end{align*}
Following the arguments in~\cite{Kliem_87}, set $\eta=\int_0^{\tau_{\overline{U}^c}} \mathbf{1}\{ x_s \in K \} \, ds $ and let $A(t)= \{ \eta \geq t \}$.  By definition of $\eta$, note that $\eta \leq \tau_{\overline{U}^c}$.  Hence 
\begin{align*}
\PP_y \{A(t)\} \leq \PP_y \{ \tau_{\overline{U}^c}\geq t \}\leq \PP_y \{  \tau_{\overline{U}^c}> t/2\}.   
\end{align*}  
Thus using the fact that $x_t$ is recurrent, applying \cref{prop:exitprob}(i) and compactness of $K$ we deduce the existence of $\alpha< 1$ and large $t_0>0$ such that 
\begin{align}\label{eqn:hoa}
\PP_y \{A(t_0 )\} \leq \PP_y\{ \tau_{\overline{U}^c} \geq t_0 \}\leq  \PP_y\{ \tau_{\overline{U}^c} > t_0/2 \}  \leq \alpha
\end{align}
for all $y\in K$.  Fix such $t_0$ and define
\begin{align*}
\eta_K(t) = \inf\{ w \geq 0 \, :\, \textstyle{\int_0^{w} \mathbf{1}_K(x_s) \, ds} = t+t_0 \}.  
\end{align*}  
Then, by path continuity,  $x_{\eta_K(t)} \in K$.  Hence, by the strong Markov property and \eqref{eqn:hoa}
\begin{align*}
\PP_y\{ A(2(t+t_0)) \} &= \E_y \mathbf{1}\{ \eta \geq  2(t+t_0)\} = \E_y  \mathbf{1}\{ \eta \geq  t+t_0\} \E_{x_{\eta_K(t)}} \mathbf{1}\{ \eta \geq  t+t_0\} 
\leq \alpha^2  
\end{align*}
Repeating the process inductively yields $\PP_y \{ A(n(t+t_0))\} \leq \alpha^n$ for all $y\in K$. Finite expectation of $\eta$, hence sigma finiteness of $\mu$, now follows.

To prove that $\mu$ is a finite measure if $x_t$ is positive recurrent, note that $\sigma_1$ is almost surely finite and 
 $\mu(\X)= \int_{\Gamma_2} \nu(dx) \E_x \sigma_2$.  By the strong Markov property, for any $x\in \Gamma_2$ we have     
\begin{align*}
\E_x \sigma_2 = \E_x \sigma_1 + \E_x (\sigma_2 - \sigma_1) = \E_x \sigma_1 + \E_x \E_{x_{\sigma_1}}  \tau_{\overline{U}^c} \leq 
\sup_{x\in \Gamma_2} \E_x \sigma_1 +\sup_{y \in \Gamma_1} \E_{y}  \tau_{\overline{U}^c}.
\end{align*} 
By \cref{prop:logexit}, $\sup_{x\in \Gamma_2} \E_x \sigma_2 < \infty$.  By positive recurrence $\E_{y}  \tau_{\overline{U}^c} < \infty$ for each $y \in \X$, and by \cref{prop:expexit}(i), 
we obtain $\sup_{y \in \Gamma_1} \E_{y}  \tau_{\overline{U}^c}< \infty$, and the finiteness of 
 $\mu$ follows. Thus, $\mu$ can be normalized to an invariant probability measure,
and the proof is complete.
\end{proof}

\begin{Proposition}
\label{prop:unique}
Suppose $x_t$ is a nice diffusion.  If $x_t$ has an invariant probability measure $\tilde{\mu}$, then $\mu$ is unique and defined (cf. \eqref{eqn:unnorm}) by 
\begin{align}
\label{eqn:hatmu}
\tilde{\mu}(B)= \frac{1}{N}\int_{\Gamma_2} \nu(dx) \E_x \sigma^B \qquad \textrm{for any } B \in \mathcal{B} \,,
\end{align}   
where $N= \int_{\Gamma_2} \nu(dx) \E_x \sigma^{\X}< \infty$, and $\nu$, $\Gamma_2$ and $\sigma^A$  are as in the proof of \cref{prop:existinv}.
In addition, $\tilde{\mu}(B) > 0$ for any non-empty open $B \subset \X$. 
\end{Proposition}

\begin{proof}
Suppose $\hat{\mu}$ is an invariant probability measure and fix a non-empty, open set $B \subset  \X $.    
We may assume without loss of generality that $B$ is bounded.  We first claim that $\hat{\mu}(B) >0$. 
Since any invariant probability measure is a solution of $L^*\hat{\mu}=0$ on $\X$ in the sense of distributions, by hypoellipticity of $L^*$ on $\X$,
$\hat{\mu}$ has a continuous probability density $\hat{\rho}$ with 
respect to Lebesgue measure on $\X$.   Since $\hat{\mu}(\X)=1$, then $\hat{\rho} \geq c >0$ on some bounded open set $W\subset \X$ and some constant $c>0$.  
By \cref{def:anice}(iii) and path continuity, for
any $x\in W$ there exists $t>0$ with $\mathcal{P}_t(x, B) >0$. By applying \cref{rem:extn} to \cref{thm:kol}, $x \mapsto \mathcal{P}_t(x, B)$ is continuous, and therefore $\mathcal{P}_t(x,B) \geq c'>0$ for all  $x$ in an open 
subset $W'$ of $W$.  Then, 
\begin{align*}
\hat{\mu}(B) = \int_{\X} \hat{\mu}(dx) \mathcal{P}_t(x, B) \geq \int_{W'} \hat{\mu}(dx) \mathcal{P}_t(x, B) \geq cc'|W'|>0  \,,
\end{align*} 
where $|W'|$ is the Lebesgue measure of $W'$.  This finishes the proof of the claim.     

If there is more than one invariant probability measure, then by standard arguments we can choose two distinct ergodic invariant measures, which are in particular 
mutually singular. But we showed by the claim above that $\X$ belongs to the support of any invariant probability measure, and therefore such measure is unique. 

Since $x_t$ has an invariant probability measure $\hat{\mu}$, \cref{cor:noinv} implies that $x_t$ is recurrent and  by \cref{prop:existinv}, $\mu$ in~\eqref{eqn:unnorm} 
 is $\sigma$-finite.

Since $\hat{\mu}$ is the unique invariant probability measure, it is ergodic and  we proved that the support of $\hat{\mu}$ is $\X$. Then, 
by Birkhoff's Ergodic Theorem, for any compact set $K\subset \X$
\begin{align}
\label{eqn:ergodic}
\lim_{T \to \infty} \frac{1}{T}\int_0^T \mathbf{1}_K(x_t) \, dt = \hat{\mu}(K)
\end{align}  
for almost every $x\in \textrm{supp}(\hat{\mu}) = \X$. 
Hence, by the invariance of $\mu$, for any compact set $K\subset \X$ and $t \geq 0$, we have  
\begin{align*}
\mu(K) =  \int_{\X} \E_x \mathbf{1}_K(x_t) \mu(dx) =
\lim_{T\rightarrow \infty} \int_{\X} \mu(dx) \E_x \frac{1}{T}\int_0^T \mathbf{1}_K(x_s)  \, ds  \geq \mu(\X) \hat{\mu}(K) \,,
\end{align*}   
where the last inequality follows from Fatou's lemma and~\eqref{eqn:ergodic}.  
Since there exists a compact $K\subseteq \X$ so that $\hat{\mu}(K)>0$ and by $\sigma$-finiteness $\mu(K)< \infty$ we have shown that $\mu(\X) < \infty$.  

After normalization, $\mu$, and therefore $\tilde{\mu}$, is a well defined invariant probability measure and, by uniqueness, $\tilde{\mu} = \hat{\mu}$, as desired. 
\end{proof}

\begin{Remark}
Note that one can choose arbitrarily the sets $U, V$, $\Gamma_1=\partial V$ and $\Gamma_2=\partial U$ in the definition of $\mu$
in~\eqref{eqn:unnorm}, as long as $\overline{U}\subseteq V\subseteq \overline{V} \subseteq \X$ and $U,V$ are  nonempty bounded, open with smooth boundaries $\Gamma_1, \Gamma_2$, respectively.  
By  \cref{prop:unique}, different choices of $U$ and $V$ induce the same measure $\mu$ up to a normalization constant.  
\end{Remark}

The final result in this section establishes that in our context, the existence of an invariant probability measure implies that $x_t$ is positive recurrent.  

\begin{Theorem}
Suppose $x_t$ is a nice diffusion.  Then, 
$x_t$ has an invariant probability measure if and only if $x_t$ is positive recurrent.  
\end{Theorem}

\begin{proof}
If $x_t$ is positive recurrent, then there is an invariant probability measure, which is moreover unique, according to \cref{prop:existinv} and~\cref{prop:unique}. 

Conversely, if $x_t$ has an invariant probability measure $\tilde{\mu}$, then it is unique and given by~\eqref{eqn:hatmu}.  
Fix an open, non-empty, bounded set $V \subset \X$ with $\overline{V}\subset \X$. 
It suffices to show that $\E_x \tau_{\overline{V}^c}< \infty$ for all $x$ in a dense set $D \subset \X\setminus \overline{V} $, because then recurrence combined with \cref{prop:expexit}(ii) implies $\E_x \tau_{\overline{V}^c}< \infty$ for all $x \in \X\setminus \overline{V}$ and the result follows.

 Fix $x\in \X\setminus \overline{V}$ and $\epsilon >0$ such that $B_{\epsilon}(x)\subset\X\setminus \overline{V}$.  
 By hypothesis and \eqref{eqn:unnorm}, $\tilde{\mu}(\X)=1< \infty$, which implies $\E_y \sigma_2 < \infty$, for every $y\in \Gamma_2' \subset \Gamma_2$, where $\nu(\Gamma_2') =1$ .   
 By  \cref{prop:unique} one has $\tilde{\mu}(B_{\epsilon/2}(x))>0$, and there exists $x_* \in \Gamma_2'$ for which $\PP_{x_*}\{ \tau_{B^c_{\epsilon/2}(x)} < \sigma_2 \} >0$.  Denoting for simplicity $\eta = \tau_{B^c_{\epsilon/2}(x)}$,  the strong Markov property implies
\begin{align}
\infty &> \E_{x_*} \sigma_2 \geq \E_{x_*} \mathbf{1}\{ \eta <  \sigma_2\} \sigma_2 \\
&= \E_{x_*} \eta \mathbf{1}\{ \eta <  \sigma_2\} + 
\E_{x_*} \mathbf{1}\{ \eta <  \sigma_2\} \E_{x_*}[(\sigma_2 - \eta)| \mathcal{F}_{\eta}] 
\geq \E_{x_*} \mathbf{1}\{ \eta <  \sigma_2\} \E_{x_\eta} \tau_{\overline{V}^c}. 
\end{align}     
Since $\PP_{x_*} \{ \eta < \sigma_2 \} >0$, there exists $y\in B_{\epsilon}(x)$ for which $\E_y \tau_{\overline{V}^c} < \infty$.  This finishes the proof.  
\end{proof}

\begin{Remark}
\label{rem:gaps}
We indicate a gap in the proof of~\cite[Lemma~3.7]{Kliem_87} and note some missing details in~ \cite[Lemma~4.4]{Kliem_87}.  

The result~~\cite[Lemma~3.7]{Kliem_87} is crucially used   to establish the dichotomy for transient and recurrent points (see~\cite[Theorem~3.2]{Kliem_87}). 
 In our notation, the context of the argument is as follows.  It is assumed that for \emph{fixed} $x$, $\int_0^\infty \mathcal{P}_t(x, V) \, dt < \infty$ for some open neighborhood $V$ of $x$.  It is then claimed that 
\begin{align*}
\limsup_{y\to x}\int_0^\infty \mathcal{P}_t(y, V) \, dt\leq \int_0^\infty \limsup_{y\to x} \mathcal{P}_t(y, V) \, dt \leq \int_0^\infty \mathcal{P}_t(x, V) \, dt< \infty.\end{align*}
Thus it then follows that there exists an open neighborhood of $x$ where the integral is finite.  To the best of our knowledge, the limit-integral exchange was not justified and we could not find a
 simple solution.  In essence, our fix of the arguments in~\cite{Kliem_87} were presented this section.    

In addition, there are missing details in the proof of \cite[Lemma~4.4]{Kliem_87}.  Indeed it is claimed that Bony's form of the Harnack inequality applies without the assumptions made in Bony's original paper~\cite[Theorem~7.1]{Bony_69}.  Although we have seen that this is indeed true as claimed (see~\cref{thm:harnack}), it requires some nontrivial arguments 
like those presented in this paper.  Similarly, the claim in~\cite[Lemma~4.4]{Kliem_87}  is effectively~\cref{prop:expexit}.          
\end{Remark}

\subsubsection*{Acknowledgments}
J.F. and D.P.H. graciously acknowledge support from National Science Foundation grants DMS-1816408 (J.F.) and DMS-1855504 (D.P.H.).  We also acknowledge fruitful conversations on the topic of this paper with Nathan Glatt-Holtz.

\newcommand{\etalchar}[1]{$^{#1}$}


\begin{thebibliography}{CEHRB18}

\bibitem[AKDR66]{AZ_66}
J.~Az{\'e}ma, M.~Kaplan-Duflo, and D.~Revuz.
\newblock R{\'e}currence fine des processus de {M}arkov.
\newblock In {\em Annales de l'IHP Probabilit{\'e}s et statistiques}, volume~2,
  pages 185--220, 1966.

\bibitem[Bax11]{Bax_05}
P.~Baxendale.
\newblock T{E} {H}arris's contributions to recurrent {M}arkov processes and
  stochastic flows.
\newblock {\em The Annals of Probability}, pages 417--428, 2011.

\bibitem[BCCM05]{Bec_05}
J.~Bec, A.~Celani, M.~Cencini, and S.~Musacchio.
\newblock Clustering and collisions of heavy particles in random smooth flows.
\newblock {\em Phys. Fluids}, 17(7):073301, 2005.

\bibitem[BCH07]{Bec_07}
J.~Bec, M.~Cencini, and R.~Hillerbrand.
\newblock Clustering of heavy particles in random self-similar flow.
\newblock {\em Phys. Rev. E}, 75(2):025301, 2007.

\bibitem[BHW12]{BHW_12}
J.~Birrell, D.~P. Herzog, and J.~Wehr.
\newblock The transition from ergodic to explosive behavior in a family of
  stochastic differential equations.
\newblock {\em Stoch. Process. their Appl.}, 122(4):1519--1539, 2012.

\bibitem[BL20]{Bed_20}
J.~Bedrossian and K.~Liss.
\newblock Quantitative spectral gaps and uniform lower bounds in the small
  noise limit for {M}arkov semigroups generated by hypoelliptic stochastic
  differential equations.
\newblock {\em arXiv preprint arXiv:2007.13297}, 2020.

\bibitem[Bon69]{Bony_69}
J.-M. Bony.
\newblock Principe du maximum, in{\'e}galit{\'e} de {H}arnack et unicit{\'e} du
  probleme de {C}auchy pour les op{\'e}rateurs elliptiques
  d{\'e}g{\'e}n{\'e}r{\'e}s.
\newblock In {\em Annales de l'institut Fourier}, volume~19, pages 277--304,
  1969.

\bibitem[CEHRB18]{Cun_18}
N.~Cuneo, J.P. Eckmann, M.~Hairer, and L.~Rey-Bellet.
\newblock Non-equilibrium steady states for networks of oscillators.
\newblock {\em Electron. J. Probab.}, 23:1--28, 2018.

\bibitem[CFH21]{CFH_21}
M.~Carfagnini, J.~F\"{o}ldes, and D.~P. Herzog.
\newblock A functional law of the iterated logarithm for weakly hypoelliptic
  diffusions at time zero.
\newblock {\em arXiv preprint: 2106.13288}, pages 1--31, 2021.

\bibitem[CHSG21]{Cam_21}
E.~Camrud, D.~P. Herzog, G.~Stoltz, and M.~Gordina.
\newblock Weighted ${L}^2$-contractivity of {L}angevin dynamics with singular
  potentials.
\newblock {\em arXiv preprint arXiv:2104.10574}, 2021.

\bibitem[EGZ19]{EGZ_19}
A.~Eberle, A.~Guillin, and R.~Zimmer.
\newblock Couplings and quantitative contraction rates for {L}angevin dynamics.
\newblock {\em The Annals of Probability}, 47(4):1982--2010, 2019.

\bibitem[EM01]{EMat_01}
W.~E and J.~C. Mattingly.
\newblock Ergodicity for the {N}avier-{S}tokes equation with degenerate random
  forcing: finite-dimensional approximation.
\newblock {\em Comm. Pure Appl. Math}, 54(11):1386--1402, 2001.

\bibitem[FGHH21]{FGHH_20}
J.~F\"{o}ldes, N.~E. Glatt-Holtz, and D.~P. Herzog.
\newblock Sensitivity of steady states in a degenerately damped stochastic
  {L}orenz system.
\newblock {\em Stoch. Dyn.}, 21(8):Paper No. 2150055, 32, 2021.

\bibitem[Fri75]{Friedman_75}
A.~Friedman.
\newblock In {\em Stochastic differential equations and applications volume 1},
  page 225pp. Academic Press, 1975.

\bibitem[GHW11]{GHW_11}
K.~Gaw\c{e}dzki, D.P. Herzog, and J.~Wehr.
\newblock Ergodic properties of a model for turbulent dispersion of inertial
  particles.
\newblock {\em Commun. Math. Phys.}, 308(1):49--80, 2011.

\bibitem[GOS{\etalchar{+}}19]{Gia_19}
G.~Giacomin, S.~Olla, E.~Saada, H.~Spohn, and G.~Stoltz.
\newblock {\em Stochastic Dynamics Out of Equilibrium}.
\newblock Springer, 2019.

\bibitem[GS90]{GS_90}
N.~Garofalo and F.~Segala.
\newblock Estimates of the fundamental solution and wiener's criterion for the
  heat equation on the heisenberg group.
\newblock {\em Indiana University mathematics journal}, 39(4):1155--1196, 1990.

\bibitem[GT15]{GT_15}
D.~Gilbarg and N.~S. Trudinger.
\newblock {\em Elliptic partial differential equations of second order}, volume
  224.
\newblock springer, 2015.

\bibitem[Har56]{Har_56}
T.E. Harris.
\newblock The existence of stationary measures for certain markov processes
  proceedings of the third berkeley symposium on mathematical statistics and
  probability 1954--1955 2 univ, 1956.

\bibitem[Has60]{Khas_60-rus}
R.~Z. Has'minski\u{\i}.
\newblock Ergodic properties of recurrent diffusion processes and stabilization
  of the solution of the {C}auchy problem for parabolic equations.
\newblock {\em Teor. Verojatnost. i Primenen.}, 5:196--214, 1960.

\bibitem[HM11]{HM_11}
M.~Hairer and J.~C. Mattingly.
\newblock Yet another look at {H}arris' ergodic theorem for markov chains.
\newblock In {\em Seminar on Stochastic Analysis, Random Fields and
  Applications VI}, pages 109--117. Springer, 2011.

\bibitem[HM19]{HerMat_19}
D.~P. Herzog and J.~C Mattingly.
\newblock Ergodicity and {L}yapunov functions for {L}angevin dynamics with
  singular potentials.
\newblock {\em Comm. Pure Appl. Math.}, 72(10):2231--2255, 2019.

\bibitem[HN21]{HN_21}
D.~P. Herzog and H.~D. Nguyen.
\newblock Stability and invariant measure asymptotics in a model for heavy
  particles in rough turbulent flows.
\newblock {\em arXiv preprint arXiv:2104.08629}, 2021.

\bibitem[H{\"o}r67]{Hor_67}
L.~H{\"o}rmander.
\newblock Hypoelliptic second order differential equations.
\newblock {\em Acta Mathematica}, 119(1):147--171, 1967.

\bibitem[Kha60]{Khas_60}
R.~Z. Khas'minskii.
\newblock Ergodic properties of recurrent diffusion processes and stabilization
  of the solution to the {C}auchy problem for parabolic equations.
\newblock {\em Theory of Probability \& Its Applications}, 5(2):179--196, 1960.

\bibitem[Kha11]{Khas_11}
R.~Khasminskii.
\newblock {\em Stochastic stability of differential equations}, volume~66.
\newblock Springer Science \& Business Media, 2011.

\bibitem[Kli87]{Kliem_87}
W.~Kliemann.
\newblock Recurrence and invariant measures for degenerate diffusions.
\newblock {\em The annals of probability}, 15(2):690--707, 1987.

\bibitem[Kog17]{Kog_17}
A.~E. Kogoj.
\newblock On the {D}irichlet problem for hypoelliptic evolution equations:
  {P}erron--{W}iener solution and a cone-type criterion.
\newblock {\em J. Differential Equations}, 262(3):1524--1539, 2017.

\bibitem[Kol34]{Kol_34}
A.N. Kolmogorov.
\newblock Zuf\"{a}llige bewegungen.
\newblock {\em Ann. of Math.}, 35(2):116--117, 1934.

\bibitem[Lac97]{Lac_97}
A.~Lachal.
\newblock Local asymptotic classes for the successive primitives of {B}rownian
  motion.
\newblock {\em The Annals of Probability}, pages 1712--1734, 1997.

\bibitem[LSR10]{Stoltz_10}
T.~Leli\'{e}vre, G.~Stoltz, and M.~Rousset.
\newblock {\em Free energy computations: A mathematical perspective}.
\newblock World Scientific, 2010.

\bibitem[LSS20]{LSS_20}
B.~Leimkuhler, M.~Sachs, and G.~Stoltz.
\newblock Hypocoercivity properties of adaptive {L}angevin dynamics.
\newblock {\em SIAM J. Appl. Math.}, 80(3):1197--1222, 2020.

\bibitem[LTU17]{LTU_17}
E.~Lanconelli, G.~Tralli, and F.~Uguzzoni.
\newblock Wiener-type tests from a two-sided {G}aussian bound.
\newblock {\em Annali di Matematica Pura ed Applicata (1923-)},
  196(1):217--244, 2017.

\bibitem[MT57]{Mar_57}
G.~Maruyama and H.~Tanaka.
\newblock Some properties of one-dimensional diffusion processes.
\newblock {\em Memoirs of the Faculty of Science, Kyushu University. Series A,
  Mathematics}, 11(2):117--141, 1957.

\bibitem[MT93]{MT_93}
S.~P. Meyn and R.~L. Tweedie.
\newblock Stability of {M}arkovian processes {III}: {F}oster--{L}yapunov
  criteria for continuous-time processes.
\newblock {\em Advances in Applied Probability}, 25(3):518--548, 1993.

\bibitem[MT12]{MT_12}
S.~P. Meyn and R.~L. Tweedie.
\newblock {\em Markov chains and stochastic stability}.
\newblock Springer Science \& Business Media, 2012.

\bibitem[NS87]{NS_87}
P.~Negrini and V.~Scornazzani.
\newblock Wiener criterion for a class of degenerate elliptic operators.
\newblock {\em Journal of differential equations}, 66(2):151--164, 1987.

\bibitem[Oks13]{Oksen_13}
B.~Oksendal.
\newblock {\em Stochastic differential equations: an introduction with
  applications}.
\newblock Springer Science \& Business Media, 2013.

\bibitem[Ram97]{Ram_97}
S.~Ramaswamy.
\newblock Dirichlet problem for some hypoelliptic operators.
\newblock In {\em Proceedings of the Indian Academy of Sciences-Mathematical
  Sciences}, volume 107, pages 405--409. Springer, 1997.

\bibitem[RB06]{RB_06}
L.~Rey-Bellet.
\newblock Ergodic properties of {M}arkov processes.
\newblock In {\em Open quantum systems II}, pages 1--39. Springer, 2006.

\bibitem[Rom04]{Rom_04}
M.~Romito.
\newblock Ergodicity of the finite dimensional approximation of the 3d
  {N}avier--{S}tokes equations forced by a degenerate noise.
\newblock {\em J. Stat. Phys.}, 114(1):155--177, 2004.

\bibitem[SW03]{SW_03}
C.~Sch{\"u}tt and E.~Werner.
\newblock Polytopes with vertices chosen randomly from the boundary of a convex
  body.
\newblock In {\em Geometric aspects of functional analysis}, pages 241--422.
  Springer, 2003.

\bibitem[WM58]{Wat_58}
H.~Watanabe and M.~Motoo.
\newblock Ergodic property of recurrent diffusion processes.
\newblock {\em Journal of the Mathematical Society of Japan}, 10(3):272--286,
  1958.

\end{thebibliography}
\end{document}